\documentclass[11pt,a4paper]{article}
\usepackage{amsfonts}
\usepackage{amssymb}
\usepackage{amsthm}
\usepackage{amsmath}
\usepackage{cite}
\usepackage{mathrsfs}
\newtheorem{theorem}{\textbf{Theorem}}[section]
\newtheorem{lemma}{\textbf{Lemma}}[section]
\newtheorem{proposition}{\textbf{Proposition}}[section]
\newtheorem{corollary}{\textbf{Corollary}}[section]
\newtheorem{remark}{\textbf{Remark}}[section]
\newtheorem{definition}{\textbf{Definition}}[section]

\allowdisplaybreaks[4]

\def\be{\begin{equation}}
\def\ee{\end{equation}}
\def\bea{\begin{eqnarray}}
\def\eea{\end{eqnarray}}
\def\bt{\begin{theorem}}
\def\et{\end{theorem}}
\def\bl{\begin{lemma}}
\def\el{\end{lemma}}
\def\br{\begin{remark}}
\def\er{\end{remark}}
\def\bp{\begin{proposition}}
\def\ep{\end{proposition}}
\def\bc{\begin{corollary}}
\def\ec{\end{corollary}}
\def\bd{\begin{definition}}
\def\ed{\end{definition}}

\def\non{\nonumber }

\def\u{\mathbf{u}}

\begin{document}

\title{Well-posedness of a diffuse-interface model for two-phase incompressible flows with thermo-induced Marangoni effect}

\author{
{\sc Hao Wu} \footnote{School of Mathematical Sciences and Shanghai
Key Laboratory for Contemporary Applied Mathematics, Fudan
University, 200433 Shanghai, China, Email:
\textit{haowufd@yahoo.com}}
}
\date{\today}
\maketitle


\begin{abstract}
We investigate a non-isothermal diffuse-interface model that describes the dynamics of two-phase incompressible flows with thermo-induced Marangoni
effect. The governing PDE system consists of the Navier--Stokes equations coupled with convective phase-field and energy transport equations,
in which the surface tension, fluid viscosity and thermal diffusivity are allowed to be temperature dependent functions.
First, we establish the existence and uniqueness of local strong solutions when the spatial dimension is two and three. Then in the two dimensional case, assuming that the $L^\infty$-norm of the initial temperature is suitably bounded with respect to the coefficients of the system, we prove the existence of global weak solutions as well as the existence and uniqueness of global strong solutions.


\end{abstract}

\section{Introduction}
\setcounter{equation}{0}

The Marangoni effect \cite{T1855,M1871} is an interesting phenomenon where mass transfer occurs due to differences in the surface tension that can either be attributed to non-uniform distributions of the surfactant \cite{MA00} or the existence of temperature
gradient in a neighborhood of the interface \cite{SS59}. The Marangoni effect turns out to dominate the convection when the depth of the fluid layer is sufficiently small, and it has many important applications in complex fluids, liquid-gas systems and ocean-geophysical dynamics \cite{BB03,BB06,YFLS05}.
For instance, the so-called B\'enard--Marangoni convection is a thermal convective motion due to the bulk forces resulting from the thermally induced density difference (buoyancy) and the interfacial forces resulting from the temperature-dependence of the surface tension for the free surface
(thermo-capillarity, i.e., the thermal-induced Marangoni effect).

Diffuse-interface models have been shown to be very powerful and efficient for modelling the dynamics of interfaces in multi-phase systems (see, e.g., \cite{AMW} and the references cited therein). In these models, sharp-interfaces of the macroscopically immiscible fluids are replaced by a thin layer (i.e., the diffuse-interface) with steep change on properties of different components. Comparing with the classical sharp-interface model, the diffuse-interface model allows for topological changes of  interfaces and has many advantages in numerical simulations for the interfacial motions \cite{FLSY05, LS03,YFLS04,LT98,HLLW11}.

In this paper, we aim to investigate a diffuse-interface model proposed by Liu et al \cite{LSFY05} (see also \cite{SLX09,LS03,GLW}) that describes the thermo-induced Marangoni effect for mixtures of two Newtonian flows (with matched density $\rho=1$ for the sake of simplicity). In this model, a phase function $\phi$ is introduced as the volume fraction of the two components. Then the interaction between different components is modeled by the elastic (mixing) energy of Ginzburg--Landau type \cite{LS03}:
 \be
 E(\phi)=\int_\Omega \left(\frac12|\nabla \phi|^2+W(\phi)\right)dx.
 \ee
 The gradient part $\int_\Omega \frac12|\nabla \phi|^2 dx$ represents the tendency to be homogeneous of the two-phase mixture, i.e., the ``phillic" interaction, while the bulk energy $\int_\Omega W(\phi) dx$ represents the ``phobic" interaction, or the tendency to be separated for different phases (see \cite{LS03}).
 Here, we consider the energy density function $W$ of the following typical form
\be
W(\phi)=\frac{1}{4\varepsilon^2} (\phi^2-1)^2, \label{W}
\ee
which can be viewed as a smooth double-well polynomial approximation of the physically relevant logarithmic potential \cite{58}. The small parameter $\varepsilon$ denotes the capillary width (i.e., width of the diffuse-interface) that reflects the competition between the two opposite tendencies \cite{LT98,LS03}. Then the governing PDE system (in a non-dimensional form) can be given as follows \cite{LSFY05,SLX09}:
\bea
&& \u_t+ \u\cdot\nabla \u-\nabla \cdot(2\mu(\theta)\mathcal{D} \u) +\nabla P\non\\
&& \quad =-\nabla \cdot \sigma + (\mathrm{Ra} \theta-\mathrm{Ga})g\mathbf{e}_n,\qquad\quad \quad \ (t,x)\in (0,T)\times\Omega,\label{1}\\
&& \nabla \cdot \u=0,\qquad\qquad\qquad\qquad\qquad\qquad\quad\,  (t,x)\in (0,T)\times\Omega,\label{2}\\
&& \phi_t+\u\cdot \nabla \phi-\gamma(\Delta \phi - W'(\phi))=0,\qquad\ \,  (t,x)\in (0,T)\times\Omega,\label{3}\\
&& \theta_t+\u\cdot\nabla \theta -\nabla \cdot(\kappa(\theta)\nabla \theta)=0,\qquad\qquad  (t,x)\in (0,T)\times\Omega,\label{4}
\eea
subject to the boundary and initial conditions:
\bea
&& \u|_\Gamma=\mathbf{0},\quad \phi|_\Gamma=\phi_b(x), \quad  \theta|_\Gamma=0, \qquad\quad  (t,x)\in (0,T)\times\Gamma,\label{bc1}\\
&&  \u|_{t=0}=\u_0(x),\quad \phi|_{t=0}=\phi_0(x), \quad \theta|_{t=0}=\theta_0(x), \qquad x\in \Omega. \label{ini}
\eea

Here, $\Omega$ is assumed to be a bounded domain in $\mathbb{R}^n$ ($n=2,
3$) with smooth boundary $\Gamma$.
Functions $\u$, $P$ and $\theta$ stand for the fluid velocity, the pressure, and the relative temperature (with respect to the background temperature $\theta_b$, which is assumed to be a constant here for the sake of simplicity), respectively. In the Navier--Stokes equation \eqref{1}, $\mathcal{D}\u=\frac12(\nabla \u+\nabla^T \u)$ corresponds to the symmetric part of the velocity gradient. The term $(\mathrm{Ra} \theta-\mathrm{Ga})g\mathbf{e}_n$ denotes the buoyancy force, in which $g$ is the gravitational acceleration constant and the constants $\mathrm{Ra}$, $\mathrm{Ga}$ are related to the Rayleigh number and Galileo number, respectively. The constant $\gamma>0$ in \eqref{3} represents the elastic relaxation time. The induced stress tensor $\sigma$ in \eqref{1} can be derived within the energetic variational framework by using the least action principle (see \cite{LSFY05,SLX09} and also \cite{GLW}) such that
\be
\sigma= \lambda(\theta)\nabla \phi\otimes \nabla \phi+\lambda(\theta)\left(\frac12|\nabla \phi|^2+W(\phi)\right)\mathbb{I}.\label{sigma}
\ee
In the formula \eqref{sigma}, the symbol  $\otimes$ denotes the usual Kronecker product, i.e., $(\mathbf{a}\otimes \mathbf{b})_{ij}=a_ib_j$ for vectors $\mathbf{a},\mathbf{b} \in
\mathbb{R}^n$ and $\mathbb{I}$ is the $n \times n$ identity matrix.
In order to model the thermal induced Marangoni effect, the temperature dependent surface tension coefficient
$\lambda$ is approximated by the E\"otv\"os rule and it takes the form of a linear function on the temperature such that
\be
 \lambda(\theta)=\lambda_0(a-b\theta),
 \non
 \ee
 where the coefficients $\lambda_0>0$, $a>0$, $b \neq 0$ are assumed to be constants (see \cite{SLX09,LSFY05,GLW}, and also \cite{BB03,BB06}). The coefficient $\lambda_0$ is proportional to the interface width $\varepsilon$, while the constants $a$ and $b$ are related to the capillary
number and the Marangoni number, respectively (see \cite{GLW}).

In this paper, we consider the general case that the fluid viscosity $\mu$ and the thermal diffusivity $\kappa$ are also allowed to depend on the temperature $\theta$, which are physically important in the study of non-isothermal fluids (see, e.g., \cite{LB96} and references cited therein).
More precisely, throughout the paper $\mu(\cdot)$ and $\kappa(\cdot)$ are supposed to be strictly positive smooth functions defined on $\mathbb{R}$ that satisfy
\be
\mu(\cdot), \ \kappa(\cdot)\in C^2(\mathbb{R})\quad\text{and}\quad  \mu(s)>0, \quad \kappa(s)> 0, \quad \forall\,s\in \mathbb{R}.\label{muka}
\ee
We remark that thanks to the maximum principle for the temperature variable $\theta$ (see Lemma \ref{mtheta}), no positive upper and lower bounds for $\mu(\cdot)$, $\kappa(\cdot)$ will be assumed (see Remark \ref{uplow} and Section 4).

The goal of this paper is to study well-posedness of the initial boundary value problem \eqref{1}--\eqref{ini} under the general assumption \eqref{muka}.
We shall first establish the existence and uniqueness of local strong solutions in both $2D$ and $3D$ cases (see Theorem \ref{str3D}).
Then in the $2D$ case, under the assumption that the $L^\infty$-norm of the initial temperature $\theta_0$ is suitably bounded with respect to those coefficients of the PDE system \eqref{1}--\eqref{4}, we prove the existence of global weak solutions (see Theorem \ref{weake1}) as well as the existence and uniqueness of global strong solutions (see Theorem \ref{str2D}).

 The coupled system \eqref{1}--\eqref{4} is a highly nonlinear PDE system that consists of the Navier--Stokes equations for the velocity $\u$, a convective Allen--Cahn type equation for the phase function $\phi$ and an energy transport equation for the temperature $\theta$. It contains
 the well-known Navier--Stokes--Allen--Cahn system \cite{FHL07, GG10, GG10b, LS03,SLX09} and the heat-conductive Boussinesq system \cite{HL05,LPZ15,huang,LPZ10,LB99,SZ13} as subsystems. A simplified version of the system \eqref{1}--\eqref{4} with \emph{constant} viscosity and thermal diffusivity has recently been considered in \cite{WX13}, where the authors proved the existence of global weak/strong solutions and investigated its long-time behavior as well as stability properties. However, because of the temperature-dependent fluid viscosity and thermal diffusivity, the arguments therein fail to apply here.

 The major challenges in mathematical analysis of the problem \eqref{1}--\eqref{ini} come from the highly nonlinear couplings between those equations due to the temperature-dependence of the surface tension parameter $\lambda$, the fluid viscosity $\mu$ and the thermal diffusivity $\kappa$. The strong nonlinear structure of problem \eqref{1}--\eqref{ini} brings us many difficulties to obtain necessary \emph{a priori} estimates. Similar difficulties have been found for the Boussinesq system with temperature dependent viscosity and thermal diffusivity (namely, without coupling with the phase-field equation in system \eqref{1}--\eqref{4}, and the nonlinearity of the highest-order in the Navier--Stokes equations, i.e., the induced capillary stress tensor $\sigma$, is neglected). We refer to \cite{LB99} for the existence of global weak solutions as well as the existence and uniqueness of local strong solutions for general data in both $2D$ and $3D$, see also \cite{LB96a} for the existence of global strong solutions with small initial data. When the spatial dimension is two, global existence of strong solutions either in a bounded domain or in the whole space has been proved in \cite{SZ13,WZF11}, while in \cite{huang} the global attractor was established in a periodic channel. We also refer to \cite{LPZ15}, where global well-posedness and long-time behavior were proved for the $2D$ Boussinesq system with partial dissipation (i.e., non-constant thermal diffusivity and zero fluid viscosity). In these previous works, delicate estimates were performed to overcome the related mathematical difficulties. It is worth noting that our non-isothermal diffuse-interface system \eqref{1}--\eqref{ini} has an even more complicated coupling structure than the classical Boussinesq system mentioned above. The temperature dependence of the surface tension $\lambda$ destroys the energy dissipation property of problem \eqref{1}--\eqref{ini}, which is important for the existence of global weak/strong solutions. To this end, we recall that the isothermal version of system \eqref{1}--\eqref{4} without the Boussinesq approximation term obeys the following dissipative energy law (see e.g., \cite{GLW,LSFY05,SLX09,HLLW11})
 \bea
 && \frac{d}{dt}\left[\frac12\int_\Omega |\u|^2 dx+
  \lambda\int_\Omega  \left(\frac{1}{2} |\nabla \phi|^2 + W(\phi) \right)dx\right]\non\\
 & =& -\mu\int_\Omega |\nabla{\u}|^2dx-\lambda\gamma\int_\Omega |-\Delta \phi +W'(\phi)|^2 dx,\label{bbel}
 \eea
 which provides the energy-level estimate of solutions to the isothermal Navier--Stokes--Allen--Cahn system and plays an essential role in the study of its well-posedness as well as long-time behavior (see \cite{GG10, GG10b}). There, certain special cancellation between the highly nonlinear induced stress term $\sigma$ in equation \eqref{1} and the convection term $\u\cdot\nabla \phi$ in equation \eqref{3} is crucial for the derivation of \eqref{bbel} (see \cite{GG10, YFLS04}, see also \cite{LL95} for the nematic liquid crystal system). However, for our problem \eqref{1}--\eqref{ini} with temperature dependent surface tension coefficient $\lambda$, similar dissipative energy law like \eqref{bbel} cannot be expected in general, because the specific cancellation mentioned above is no longer valid.

 We shall generalize several techniques in the literature to overcome the mathematical difficulties due to those strong nonlinear couplings caused by the temperature dependent coefficients. For instance, since the energy-level estimates and higher-order estimates for problem \eqref{1}--\eqref{ini} with general initial data cannot be obtained in a separate way, in order to prove the existence of local strong solutions, we try to construct a novel higher-order differential inequality (see Lemma \ref{high3d}) and combine it with a small data argument for the shifted temperature $\hat{\theta}=\theta-\theta_0$ in the sprit of \cite{LB99}. On the other hand, maximum principles for the phase function $\phi$ and the temperature $\theta$ (see Lemmas \ref{mphi}, \ref{mtheta}) are crucial to reduce the difficulties from high-order nonlinearities. In particular, we find that if the $L^\infty$-norm of the initial temperature is properly bounded (only depending on coefficients of the system and the domain $\Omega$, but not on the initial data), we are able to obtain some global estimates (see Propositions \ref{BEL1}, \ref{low-estimate}) and prove the existence of global weak as well as strong solutions in $2D$. Some dissipative estimates of the system can also be revealed (see Lemma \ref{BEL2}), which further imply the long-time convergence of global strong solutions. Besides, in order to obtain higher-order spatial estimates for the velocity $\u$ and the temperature $\theta$, one need to make use of a suitable temperature transformation \eqref{KK} induced by the temperature dependent thermal diffusivity $\kappa$ as well as properties of the Stokes problem with non-constant viscosity (see e.g., \cite{SZ13}).

The current work can be viewed as a preliminary step towards the theoretical analysis of some more refined diffuse-interface models. As we shall see, the $L^\infty$-estimate for the phase function $\phi$ turns out to be crucial in the subsequent analysis. It is an open question whether results similar to those for the system \eqref{1}--\eqref{ini} can be obtained if the Allen--Cahn type equation \eqref{3} is replaced by the Cahn--Hilliard equation \cite{58}, which is a fourth-order parabolic equation preserving the total mass $\int_\Omega \phi dx$ but losing the maximum principle for the phase function $\phi$ (see, e.g., \cite{LS03,AMW,SLX09,GLW}). We refer to \cite{Ab,B,GG10a,GG10b} and the references cited therein for analysis on the isothermal Navier--Stokes--Cahn--Hilliard systems. It is worth mentioning that in the recent work \cite{ERS14,ERS14a}, a thermodynamically consistent diffuse-interface model describing two-phase flows of incompressible fluids in a non-isothermal setting has been proposed and analyzed (with constant surface tension coefficient $\lambda$). It would be interesting to include Marangoni effects in the model studied therein. At last, we also refer to \cite{Z11} for the Cahn--Hilliard--Boussinesq equation with the specific assumption that $\lambda$ and $\kappa$ are positive constants and $\mu=0$.

The remaining part of this paper is organized as follows. In Section 2, we introduce the functional settings and state the main results. Section 3 is devoted to the existence and uniqueness of local strong solutions to problem \eqref{1}--\eqref{ini} with general initial data in both $2D$ and $3D$ cases (see Theorem \ref{str3D}). In Section 4, under suitable assumptions on the $L^\infty$-norm of the initial temperature $\theta_0$, we first establish the existence of global weak solutions (Theorem \ref{weake1}) and then prove the existence as well as uniqueness of global strong solutions in $2D$ (see Theorem \ref{str2D}). In the Appendix, we briefly sketch the semi-Galerkin approximate schemes that are used in the proofs for existence results.

\section{Preliminaries and Main Results}
\setcounter{equation}{0}

\subsection{Functional setup and notations}
Let $X$ be a Banach or Hilbert space, whose norm is denoted by $\|\cdot\|_X$. $X'$ indicates the
dual space of $X$ and $\langle\cdot,\cdot\rangle_{X',X}$ denotes the
corresponding duality products. The boldface letter $\mathbf{X}$ stands for the
vectorial space $X^n$ endowed with the product structure.
We denote by $L^p(\Omega)$ and $W^{m,p}(\Omega)$ the usual  Lebesgue spaces and Sobolev spaces of real measurable functions on the domain $\Omega$. When $p=2$, $W^{m,p}(\Omega)$ will be denoted by $H^m(\Omega)$ and in particular, $H^0(\Omega)=L^2(\Omega)$. $H^1_0(\Omega)$ is the closure of $C_0^\infty(\Omega)$ in the $H^1$-norm, and its dual space is simply denoted by $H^{-1}(\Omega)$.
 We denote the H\"older space on $\Omega$ by $ C^\alpha(\overline{\Omega})$ with $\alpha\in (0, 1)$ . For any $f \in C^\alpha(\overline{\Omega})$, $[f]_\alpha$ represents the H\"older semi-norm of $f$ that $[f]_\alpha =\sup_{x,y\in \Omega, x\neq y} \frac{|f(x)-f(y)|}{|x-y|^\alpha}$.
If $I$ is an interval of $\mathbb{R}^+$, we use the function space $L^p(I;X)$ with $1 \leq p \leq +\infty$, which consists of $p$-integrable
functions with values in the Banach space $X$. For the sake of simplicity, the inner product in $L^2(\Omega)$ and its associate norm will be denoted by $(\cdot,\cdot)$ and  $\|\cdot\|$, respectively.

Let $\mathcal{V}=\{\mathbf{u}\in C_0^\infty(\Omega, \mathbb{R}^n),\ \nabla\cdot \mathbf{u}=0\ \text{in}\ \Omega\}$.
We denote the space $\mathbf{H}$ (or $\mathbf{V}$) the closure of $\mathcal{V}$ in $\mathbf{L}^2(\Omega)$
(or $\mathbf{H}^1_0(\Omega)$):
 \bea
&& \mathbf{H}=\{\u\in \mathbf{L}^2(\Omega): \nabla \cdot \u=0\ \text{in}\ \Omega, \ \ \u\cdot
\mathbf{n}=0 \ \text{on}\ \Gamma\},\non\\
&& \mathbf{V}=\{\u\in
\mathbf{H}_0^1(\Omega): \nabla \cdot \u=0\ \text{in}\ \Omega\}.\non
 \eea
 Hence, $\mathbf{H}$ and $\mathbf{V}$ are Hilbert spaces with norms $\|\cdot\|$ and $\|\cdot\|_{\mathbf{H}^1}$, respectively.

 Let $\Pi$ be the orthogonal projection of $\mathbf{L}^2(\Omega)$ onto $\mathbf{H}$ related to the usual
Helmholtz decomposition. We recall the Stokes operator $S: \mathbf{H}^2(\Omega) \cap
\mathbf{V} \rightarrow \mathbf{H}$ such that for all $\u \in D(S):=\mathbf{H}^2(\Omega) \cap
\mathbf{V}$, $S\u=\Pi (-\Delta \u)= -\Delta \u+\nabla \pi
\in \mathbf{H}$. It is well-known that its inverse $S^{-1}$ is a compact linear operator on $\mathbf{H}$
and $\|S\cdot\|$ gives a norm on $D(S)$ that is equivalent to the usual
$\mathbf{H}^2$-norm. Besides, we have the following estimates (see \cite[Lemma 3.4]{LB99} and \cite{Te01})
\bl\label{stoo}
For any $\u \in D(S)$, consider the Helmholtz decomposition $S\u= -\Delta \u+\nabla \pi$ where the pressure $\pi$ is taken such that $\int_\Omega \pi dx=0$.
Then for any $\nu> 0$, there exists a positive constant $C_\nu$ independent of $\u$, it holds
\be
 \|\pi\| \leq \nu\|S{\u}\|+C_\nu\|\nabla{\u}\|. \label{Stokes I}
 \ee
Moreover, there exists a positive constant $c=c(n, \Omega)$ such that
 \be \|\u\|_{\mathbf{H}^2}+\|\pi\|_{H^1 \backslash
{\mathbb{R}}} \leq c\|S\u\|. \label{Stokes II}
 \ee
\el

In the following text, for two $n\times n$ matrices $M_1, M_2$, we denote $M_1 : M_2=\mathrm{trace}(M_1 M_2^T)$. The upper case letters $C$, $C_i$ will stand for genetic constants
possibly depending on the domain $\Omega$, the coefficients $a$, $b$, $\lambda_0$, $\mu$, $\kappa$, $\gamma$,
 $\varepsilon$ as well as the boundary and initial data, while lower case letters $c$, $c_i$ will denote interpolation/embedding constants that only depend on $\Omega$ and also the spatial dimension $n$. These constants may vary in the same line in the subsequent estimates and their special dependence will be pointed out explicitly in the text if necessary.

\subsection{Main results}

Now we introduce the notions of weak and strong solutions to problem \eqref{1}--\eqref{ini} considered in this paper:

\begin{definition}[Weak solutions]
Let $n=2$. For any $T \in (0, +\infty)$, $(\u_0, \phi_0, \theta_0)$
$\in$ $\mathbf{H} \times (H^1(\Omega)\cap L^\infty(\Omega))\times (H^1_0(\Omega)\cap L^\infty(\Omega))$ with $\phi_b\in H^\frac32(\Gamma)$  and  $\phi_0|_\Gamma=\phi_b$, the triple $(\u, \phi, \theta)$
satisfying
 \bea
  &&\u \in L^\infty(0, T; \mathbf{H}) \cap L^2(0, T;
\mathbf{V}),\label{w1} \\
&&\phi\in L^\infty(0, T; {H}^1(\Omega)\cap L^\infty(\Omega)) \cap L^2(0,
T; {H}^2(\Omega)),\label{w2}\\
&& \theta\in L^\infty(0, T; {H}^1_0(\Omega)\cap L^\infty(\Omega)) \cap L^2(0, T;
{H}^2(\Omega)),\label{w3}\\
&& \u_t\in L^2(0,T; \mathbf{V}'),\quad \phi_t, \theta_t\in L^2(0,T; L^2(\Omega)),
 \eea
is called a weak solution of problem \eqref{1}--\eqref{ini}, if
 \bea
 && \langle \u_t, \mathbf{v}\rangle_{\mathbf{V}', \mathbf{V}}+ \int_\Omega (\u\cdot\nabla)\u\cdot \mathbf{v} dx
 + 2\int_\Omega \mu(\theta) \mathcal{D} {\u} : \mathcal{D}{\mathbf{v}} dx \non\\
 && \qquad = \int_\Omega [\lambda(\theta)\nabla\phi\otimes\nabla\phi] : \nabla \mathbf{v} dx+ \int_\Omega  \theta \mathbf{g} \cdot \mathbf{v} dx,\quad \forall\,
 \mathbf{v}\in \mathbf{V},\non\\
 && \phi_t+\u\cdot \nabla\phi+\gamma (-\Delta \phi+W'(\phi)) =0,\quad \mathrm{a.e.\ in}\  (0, T)\times \Omega, \non\\
 &&  \theta_t + \u\cdot \nabla\theta - \nabla\cdot(\kappa(\theta)\nabla \theta)=0,\qquad \quad \ \mathrm{a.e.\ in}\  (0, T)\times \Omega,\non
 \eea
  and it fulfills the boundary conditions $\phi|_\Gamma =\phi_b$, $\theta|_\Gamma=0$ as well as the initial conditions in \eqref{ini}. Here we simply denote the vector $\mathbf{g}=\mathrm{Ra} g\mathbf{e}_n$.
 \end{definition}
 \begin{remark}
  (1) In the above variational form for the fluid velocity $\u$,  we have used the following facts due to the incompressibility condition \eqref{2} such that for any $\mathbf{v}\in \mathbf{V}$, it holds
 \bea
 \int_\Omega \left\{\nabla \cdot \left[\lambda(\theta)\Big(\frac12|\nabla\phi|^2+ W'(\phi)\Big)\mathbb{I}\right] \right\}\cdot \mathbf{v}\,dx
 =\int_\Omega \nabla P \cdot \mathbf{v}dx= \int_\Omega \mathbf{e}_n\cdot \mathbf{v} dx=0.\non
\eea
We have written the buoyancy force as $(\mathrm{Ra} \theta-\mathrm{Ga})g\mathbf{e}_n=\theta \mathbf{g}-\mathrm{Ga}g\mathbf{e}_n$, then the part $-\mathrm{Ga}g\mathbf{e}_n$ can be simply absorbed into the pressure $P$.

(2) By the interpolation theorem \cite{SI}, it easily follows that $\u\in C([0,T]; \mathbf{H})$ and also $\phi, \theta\in C([0, T]; H^1(\Omega))$. Thus the initial condition \eqref{ini} make sense.
 \end{remark}

 \begin{definition}[Strong solutions] \label{def of strong solution}
Let $n=2,3$. For any $T \in (0, +\infty)$, $(\u_0, \phi_0, \theta_0)$
$\in$ $\mathbf{V} \times H^2(\Omega)\times (H^2(\Omega)\cap H^1_0(\Omega))$ with $\phi_b\in H^\frac52(\Gamma)$ and $\phi_0|_\Gamma=\phi_b$, we say that
the triple $(\u, \phi, \theta)$ is a strong solution to problem
\eqref{1}--\eqref{ini}, if
 \bea
 &&\u \in L^\infty(0, T; \mathbf{V}) \cap L^2(0, T; \mathbf{H}^2(\Omega))\cap H^1(0,T; \mathbf{H}),\non\\
&& \phi \in L^\infty(0, T; {H}^2(\Omega))
\cap L^2(0, T; {H}^3(\Omega)),\non\\
&& \phi_t \in L^\infty(0, T; L^2(\Omega))\cap L^2(0, T; {H}_0^1(\Omega))\non\\
&& \theta \in L^\infty(0, T; {H}^2(\Omega)\cap H^1_0(\Omega))\cap L^2(0, T; {H}^3(\Omega)),\non\\
&& \theta_t \in L^\infty(0, T; L^2(\Omega))\cap L^2(0, T; {H}_0^1(\Omega)),\non
 \eea
 and $(\u, \phi, \theta)$ satisfies the system \eqref{1}--\eqref{4} a.e. in  $[0,T]\times \Omega$ as well as the boundary and initial conditions \eqref{bc1}--\eqref{ini}.
\end{definition}

Now we state the main results of this paper.\medskip

(A) \emph{Local strong solutions in both $2D$ and $3D$}.
\begin{theorem} \label{str3D}
Let $n=2,3$. Suppose that $\mu(\cdot)$ and $\kappa(\cdot)$ fulfill the assumption \eqref{muka}. For any initial data $(\u_0, \phi_0, \theta_0)$
$\in$ $\mathbf{V} \times H^2(\Omega)\times (H^2(\Omega)\cap H^1_0(\Omega))$, $\phi_b \in H^\frac52(\Gamma)$ with $\phi_0|_\Gamma=\phi_b$ satisfying $|\phi_0|\leq 1$ in $\Omega$ and $|\phi_b|\leq 1$ on $\Gamma$, there exists a time $T^*>0$ depending on $\|\u_0\|_{\mathbf{V}}$, $\|\phi_0\|_{H^2}$, $\|\theta_0\|_{H^2}$, $\Omega$ and coefficients of the system such that problem \eqref{1}--\eqref{ini} admits a unique local strong solution $(\u, \phi, \theta)$ on $[0,T^*]$.
\end{theorem}

(B) \emph{Global weak solutions in $2D$ under bounded initial temperature variation}.

\begin{theorem} \label{weake1}
Let $n=2$. Suppose that $\mu(\cdot)$ and $\kappa(\cdot)$ fulfill the assumption \eqref{muka}. For any initial data $(\u_0, \phi_0, \theta_0)$
$\in$ $\mathbf{H} \times (H^1(\Omega)\cap L^\infty(\Omega))\times
 (H^1_0(\Omega)\cap L^\infty(\Omega))$, $\phi_b\in H^\frac32(\Gamma)$ with $\phi_0|_\Gamma=\phi_b$ satisfying $|\phi_0|\leq 1$ a.e. in $\Omega$ and $|\phi_b|\leq 1$ on $\Gamma$, we consider problem \eqref{1}--\eqref{ini}.

 (1) There exists a constant $\Theta_1>0$ depending only on $\Omega$, the viscosity function $\mu(\cdot)$, and the coefficients $\gamma$, $\lambda_0$, $a$, $b$ $\mathrm{(\text{see}}$ $\eqref{Theta1}$ for its detailed form$\mathrm{)}$, such that if we further assume $\|\theta_0\|_{L^\infty}\leq \Theta_1$, then for
 arbitrary time $T\in(0,+\infty)$, problem \eqref{1}--\eqref{ini} admits at least one global weak solution $(\u, \phi, \theta)$ on $[0,T]$.

 (2) There exists a constant $\Theta_2\in (0, \Theta_1]$ depending on $\Theta_1$, the thermal diffusivity $\kappa(\cdot)$ and  $\Omega$ $\mathrm{(\text{see}}$ $\eqref{2.7}$ for its detailed form$\mathrm{)}$ such that if we further assume $\|\theta_0\|_{L^\infty}\leq \Theta_2$, then problem \eqref{1}--\eqref{ini} admits at least one global weak
solution $(\u, \phi, \theta)$ on $[0,+\infty)$, which is uniformly bounded in $\mathbf{H}\times (H^1(\Omega)\cap L^\infty(\Omega))\times (H^1_0(\Omega)\cap L^\infty(\Omega))$ for all $t\geq 0$.
\end{theorem}

(C) \emph{Global strong solutions in} $2D$ \emph{under bounded initial temperature variation}.

\begin{theorem} \label{str2D}
Let $n=2$. Suppose that $\mu(\cdot)$ and $\kappa(\cdot)$ fulfill the assumption \eqref{muka}. For any initial data $(\u_0, \phi_0, \theta_0)$
$\in$ $\mathbf{V} \times H^2(\Omega)\times (H^2(\Omega)\cap H^1_0(\Omega))$,
 $\phi_b \in H^\frac52(\Gamma)$ with $\phi_0|_\Gamma=\phi_b$ satisfying $|\phi_0|\leq 1$ in $\Omega$, $|\phi_b|\leq 1$ on $\Gamma$, we consider problem \eqref{1}--\eqref{ini}.

 (1) If $\theta_0$ satisfies $\|\theta_0\|_{L^\infty}\leq \Theta_1$ $\mathrm{(}$see \eqref{Theta1}$\mathrm{)}$, then for arbitrary time $T\in(0,+\infty)$, problem \eqref{1}--\eqref{ini} admits a unique global strong solution $(\u, \phi, \theta)$ on $[0,T]$.

 (2) If $\theta_0$ satisfies  $\|\theta_0\|_{L^\infty}\leq \Theta_2$ $\mathrm{(}$see \eqref{2.7}$\mathrm{)}$, then problem \eqref{1}--\eqref{ini} admits a unique global strong solution $(\u, \phi, \theta)$ on $[0,+\infty)$
 that is uniformly bounded in $\mathbf{V}\times H^2(\Omega)\times H^2(\Omega)$ for all $t\geq 0$.
 Moreover, we have
 \be
\lim_{t\to +\infty} (\|\u(t)\|_{\mathbf{V}}+\|\Delta \phi(t)-W'(\phi(t))\|+\|\theta(t)\|_{H^2})=0.\non
\ee
\end{theorem}

\begin{remark}
Similar results could be obtained for other types of boundary conditions with minor modifications in the proofs.
For instance, the no-slip boundary condition on $\u$ could be replaced by the free-slip boundary condition, while the Dirichlet boundary condition for $\phi$ could be replaced by the homogeneous Neumann boundary condition that accounts for an angle condition of the interface on the boundary $\partial \Omega$. Besides, the homogeneous Dirichlet boundary condition for $\theta$ can be easily generalized to a non-homogeneous one by using the shifting method in \cite{LB99}.
\end{remark}

\section{Local Well-posedness}\label{locwell}
\setcounter{equation}{0}

In this section, we prove the existence and uniqueness of local strong solutions to problem \eqref{1}--\eqref{ini} in both $2D$ and $3D$.  Due to the highly nonlinear structure of the system, the derivation of proper \emph{a priori} higher-order estimates are much more involved than the case with constant coefficients.

\subsection{Preliminaries}

\subsubsection{Useful inequalities} First, we recall some inequalities that will be  frequently used in this paper. For the sake of convenience, we will interchangeably use the following equivalent norms:
\bea
&& \|\nabla f\|\simeq\|f\|_{H^1}, \quad \forall\, f\in H^1_0(\Omega),\qquad \|\Delta f\|\simeq\|f\|_{H^2}, \quad \forall\, f\in H^2(\Omega)\cap H^1_0(\Omega),\non\\
&& \|\nabla \u\| \simeq \|\u\|_{\mathbf{V}}, \quad \forall\, \u\in \mathbf{V},\qquad \|\Delta \u\|\simeq\|\u\|_{\mathbf{H}^2}, \quad \forall\, \u\in \mathbf{H}^2(\Omega)\cap\mathbf{V}.\non
\eea
Next, for the elliptic boundary value problem $-\Delta g= h$ with nonhomogeneous Dirichlet boundary condition $g|_\Gamma=g_b$, we deduce from the classical elliptic regularity theorem that
\be
\|g\|_{H^{k+2}(\Omega)}\leq c(n, \Omega)\Big(\|h\|_{H^k(\Omega)}+\|g\|+\|g_b\|_{H^\frac{2k+3}{2}(\Gamma)}\Big),\quad k=\{0,1\}.\non
\ee

The following interpolation inequalities can be found in classical literature, e.g., \cite{Te01}:
\begin{lemma}[Gagliardo--Nirenberg inequality]
Let $j, m$ be arbitrary integers satisfying $0\leq j< m$ and let $1\leq q, r\leq +\infty$, $\frac{j}{m}\leq a\leq 1$ such that
$$\frac{1}{p}-\frac{j}{n}=a\left(\frac{1}{r}-\frac{m}{n}\right)+(1-a)\frac{1}{q}.$$
Suppose $\Omega\in \mathbb{R}^n$ is a bounded domain with smooth boundary. For any $f\in W^{m,r}(\Omega)\cap L^q(\Omega)$, there are two constants $c_1, c_2$ such that
$$\|\partial^jf\|_{L^p}\leq c_1\|\partial^m f\|_{L^r}^a\|f\|_{L^q}^{1-a}+c_2\|f\|_{L^q},$$
with the following exception: if $1<r<+\infty$ and $m-j-\frac{n}{r}$ is a nonnegative integer, then the above inequality holds only for $\frac{j}{m}\leq a<1$.
\end{lemma}

\begin{lemma}[Agmon's inequality] \label{Agmon}
Suppose that $\Omega\in \mathbb{R}^n$ $(n=2,3)$ is a bounded domain with smooth boundary. For any $f\in H^2(\Omega)$, it holds
\be
\|f\|_{L^\infty}\leq c\|f\|_{L^2}^\frac12\|f\|_{H^2}^\frac12, \quad \text{if}\ n=2,\quad \text{and}\ \ \|f\|_{L^\infty}\leq c\|f\|_{H^1}^\frac12\|f\|_{H^2}^\frac12, \quad \text{if}\ n=3.\non
\ee
\end{lemma}

\subsubsection{Maximum principles for $\phi$ and $\theta$}
One important feature of problem \eqref{1}--\eqref{ini} is that the phase function $\phi$ and the temperature $\theta$ satisfy suitable weak maximum principles. These facts will be important in our subsequent proofs.

First, thanks to the double-well structure of $W(\phi)$ (see \eqref{W}), similar to the simplified Ericksen--Leslie system for nematic liquid crystal flows \cite{LL95, C09}, one can easily prove that
\begin{lemma} \label{mphi} Let $n=2,3$. Consider the initial boundary value problem
 \be
 \begin{cases}
& \phi_t+ \u\cdot \nabla\phi=\gamma (\Delta \phi- W'(\phi)), \quad (t,x)\in (0,T)\times \Omega,\\
& \phi|_\Gamma=\phi_b(x), \qquad\qquad \qquad \qquad  \ \ (t,x)\in (0,T)\times \Gamma,\\
& \phi|_{t=0}=\phi_0(x), \qquad\qquad\qquad \qquad  x\in \Omega. \label{mmphi}
\end{cases}
 \ee
Suppose that $\u \in L^\infty(0, T; \mathbf{H}) \cap L^2(0, T;
\mathbf{V})$, $\phi_0\in  H^1(\Omega)$, $\phi_b \in H^\frac32(\Gamma)$ and $\phi_0|_\Gamma=\phi_b$ satisfying $|\phi_0|\leq 1$ a.e. in $\Omega$, $|\phi_b|\leq 1$ on $\Gamma$.  If $\phi\in L^\infty(0,T; H^1(\Omega))\cap L^2(0, T; H^2(\Omega))$ is a weak solution to problem \eqref{mmphi}, then  $|\phi(t,x)| \leq 1$, a.e. in $\Omega$ for all $t\in [0,T]$.
\end{lemma}
Next, concerning the temperature equation with convection, we have the following $L^\infty$-estimate for $\theta$ (see \cite{LB99}):
\begin{lemma} \label{mtheta}
Let $n=2, 3$ and $\kappa(\cdot)$ satisfies \eqref{muka}. Consider the initial boundary value problem
 \be
 \begin{cases}
&\theta_t+ \u\cdot \nabla\theta=\nabla\cdot(\kappa(\theta)\nabla \theta),\quad (t,x)\in (0,T)\times \Omega,\\
&\theta|_\Gamma=0, \qquad\qquad \qquad\qquad \quad\  (t,x)\in  (0,T)\times\Gamma, \\
& \theta|_{t=0}=\theta_0(x),\qquad\qquad\qquad\quad  x\in \Omega.
\end{cases}\label{eqthe}
 \ee
Suppose that $\theta_0(x)\in  L^\infty(\Omega)$ and $\u \in L^\infty(0, T; \mathbf{H}) \cap L^2(0, T;
\mathbf{V})$, in addition, when $n=3$ we also assume $\u \in L^\infty(0, T; \mathbf{L}^3(\Omega))$.
If $\theta\in L^\infty(0,T;  L^2(\Omega))\cap
L^2(0, T; H^1_0(\Omega))$ is a weak solution of problem \eqref{eqthe},
 then $\|\theta(t)\|_{L^\infty(\Omega)}\leq \|\theta_0\|_{L^\infty(\Omega)}$ for all $t\in [0,T]$.
\end{lemma}

\subsubsection{Modifications of the viscosity and thermal diffusivity} \label{uplow}

  Thanks to the maximum principle for the temperature $\theta$ (i.e., Lemma \ref{mtheta}),
  in the study of local strong solutions ($n=2,3$), inspired by the argument in \cite{LB96,LB99}, one can first transform the original problem \eqref{1}--\eqref{ini} into an
equivalent one by properly modifying the fluid viscosity $\mu(\cdot)$ and the thermal diffusivity $\kappa(\cdot)$ that are only assumed to satisfy \eqref{muka} \emph{outside} the interval $[-\|\theta_0\|_{L^\infty}, \|\theta_0\|_{L^\infty}]$ (here assuming that $\|\theta_0\|_{L^\infty}>0$). In particular, this type of modification will have no influence on the $L^\infty$-estimate for $\theta$.

\begin{remark}
The case $\|\theta_0\|_{L^\infty}=0$ is trival. Indeed, by Lemma \ref{mtheta}, we have $\theta=0$ and thus problem \eqref{1}--\eqref{ini} will be simply reduced to the isothermal Navier--Stokes--Allen--Cahn system \cite{GG10, GG10b}. Therefore, throughout the paper, we shall focus on the nontrivial case $\|\theta_0\|_{L^\infty}>0$.
\end{remark}

The required modification can be constructed in a simple way. For instance, taking $$r=\frac13\|\theta_0\|_{L^\infty}>0,$$ we introduce the cut-off function $h(s)\in C^\infty_0(\mathbb{R})$ such that
\begin{equation}
h(s)=(\mathbf{1}_{[-4r, 4r]} * g_r)(s)=\int_\mathbb{R} \mathbf{1}_{[-4r, 4r]}(\tau) g_r(s-\tau) d\tau,\non
\end{equation}
where $\mathbf{1}_{[-4r, 4r]}$ is the characteristic function on the interval $[-4r,4r]$ and
\begin{equation}
g_r(s)=\Big(\int_\mathbb{R}g(s)ds\Big)^{-1} r^{-1} g\Big(\frac{s}{r}\Big) \quad\mathrm{with}\quad
g(s)=
\begin{cases}
\exp\big(\frac{1}{s^2-1}\big), \quad\, \mathrm{if}\ |s|<1,\\
0, \qquad \qquad\quad \, \mathrm{if}\ |s|\geq 1.
\end{cases}\non
\end{equation}
It easily follows that $h(s)=1$ for $|s|\leq 3r$, $0< h(s)<1$ for $|s|\in (3r,5r)$, and $h(s)=0$ for $|s|\geq 5r$. Denote
\bea
&&\underline{\mu}=\frac12\inf\{\mu(s): |s|\leq 5r\},\quad \overline{\mu}=2\sup\{\mu(s): |s|\leq 5r\},\non\\
&&\underline{\kappa}=\frac12\inf\{\kappa(s): |s|\leq 5r\},\quad \overline{\kappa}=2\sup\{\kappa(s): |s|\leq 5r\}.\non
\eea
Then we set for $s\in \mathbb{R}$
\begin{equation}
\mu^*(s)=(\mu(s)-\underline{\mu})h(s)+\underline{\mu}\quad \mathrm{and}\quad \kappa^*(s)=(\kappa(s)-\underline{\kappa})h(s)+\underline{\kappa}.\non
\end{equation}
It is easy to verify that the modified functions $\mu^*(\cdot)$ and $\kappa^*(\cdot)$ belong to $C^2(\mathbb{R})$ and satisfy
\begin{align}
&\mu^*(s)=\mu(s),\quad \kappa^*(s)=\kappa(s),\qquad \forall\, s\in [-\|\theta_0\|_{L^\infty},\|\theta_0\|_{L^\infty}].\non
\end{align}
Moreover,  $\mu^*(\cdot)$ and $\kappa^*(\cdot)$ are positive constants outside the interval $(-5r, 5r)$. As a result,
\begin{align}
&0<\underline{\mu}\leq \mu^*(s)\leq \overline{\mu},\quad 0<\underline{\kappa}\leq \kappa^*(s)\leq \overline{\kappa},\quad\ \ \forall\, s\in \mathbb{R},\label{bdee}\\
&(\mu^*)',\ \  (\mu^*)'',\ \  (\kappa^*)', \ \ (\kappa^*)''\ \ \mathrm{are\ bounded},\quad \forall\, s\in\mathbb{R}.\label{bdeeb}
\end{align}

\begin{remark}\label{uplowremark}
In the remaining part of Section \ref{locwell}, we shall assume that the fluid viscosity $\mu(\cdot)$ and thermal diffusivity $\kappa(\cdot)$ are modified to be $\mu^*(\cdot)$ and $\kappa^*(\cdot)$ respectively as in the above argument and drop the superscript $^*$ for the sake of simplicity. Then necessary estimates are always obtained along with these modified coefficients satisfying the properties \eqref{bdee}, \eqref{bdeeb}.
\end{remark}

\subsubsection{Variable transformation for $\theta$}
In order to overcome the difficulty from the temperature-dependence of the thermal diffusivity $\kappa$ in equation \eqref{4}, we introduce the new variable $\vartheta$ in sprit of \cite{SZ13}:
  \be
  \vartheta=\int_0^\theta \kappa(s) ds.\label{KK}
  \ee
 Since $\kappa$ is a positive $C^2$ function, there exists a strictly increasing $C^3$ function $\chi(\cdot)$ such that $\chi(\vartheta)=\chi\big(\int_0^\theta \kappa(s) ds\big)=\theta$ and
$$\chi'(\vartheta)=\frac{1}{\kappa(\theta)}, \quad \chi''(\vartheta)=-\frac{\kappa'(\theta)}{\kappa^3(\theta)},\quad \chi'''(\vartheta)=-\frac{\kappa''(\theta)}{\kappa^4(\theta)}+\frac{3\kappa'(\theta)^2}{\kappa^5(\theta)}.
$$
Under the transformation \eqref{KK},  equation \eqref{4} for $\theta$ can be re-written into the following form in terms of the new variable $\vartheta$
\be
\begin{cases}
 \vartheta_t+\u\cdot \nabla \vartheta-\frac{1}{\chi'(\vartheta)}\Delta \vartheta=0,\quad\qquad  (t,x)\in (0,T)\times \Omega,\\
 \vartheta|_\Gamma=0, \qquad\qquad \qquad \qquad \qquad\quad \ (t,x)\in  (0,T)\times\Gamma,\\
 \vartheta|_{t=0}=\vartheta_0(x)=\int_0^{\theta_0(x)}\kappa(s)ds, \qquad x\in\Omega.
 \end{cases}\label{newtheta1}
\ee
 On the other hand, by definition of $\vartheta$, we can deduce from subsection \ref{uplow} (see Remark \ref{uplowremark}) the following estimates and relations on norms of $\theta$ and $\vartheta$:
 \be
 \|\vartheta_0\|_{L^\infty}\leq \overline{\kappa}\|\theta_0\|_{L^\infty},\quad  \|\vartheta_0\|_{H^1}\leq \overline{\kappa}\|\theta_0\|_{H^1},\quad \underline{\kappa}\leq \frac{1}{\chi'(\vartheta)}\leq \overline{\kappa}.\label{temesa}
 \ee
 \be
  \underline{\kappa}\|\theta_t\| \leq \|\vartheta_t\|\leq \overline{\kappa}\|\theta_t\|,\qquad  \underline{\kappa} \|\nabla \theta\| \leq \|\nabla \vartheta\| \leq  \overline{\kappa}\|\nabla \theta\|,\label{ttt1}
 \ee
 \bea
  \|\nabla \vartheta_t\|
  &\leq& \|\kappa'(\theta)\|_{L^\infty}\|\theta_t\nabla\theta\|+\|\kappa(\theta)\|_{L^\infty}\|\nabla \theta_t\|\non\\
 &\leq& C(\|\theta_t\|_{L^4}\|\nabla\theta\|_{\mathbf{L}^4}+ \|\nabla \theta_t\|),\label{ttt2}
 \eea
 \bea
 \|\nabla^2 \theta\| &\leq& \|\chi'(\vartheta)\|_{L^\infty}\|\nabla^2 \vartheta\|+ \|\chi''(\vartheta)\|_{L^\infty}\|\nabla \vartheta\|^2_{\mathbf{L}^4}\non\\
 &\leq& C(\|\nabla^2\vartheta\|+\|\nabla \vartheta\|_{\mathbf{L}^4}^2),\label{ttt3}
 \eea
 \bea
 \|\nabla^2 \vartheta\|&\leq &\|\kappa(\theta)\|_{L^\infty}\|\nabla^2 \theta\|+ \|\kappa'(\theta)\|_{L^\infty}\|\nabla \theta\|^2_{\mathbf{L}^4}\non\\
 &\leq& C(\|\nabla^2\theta\|+\|\nabla \theta\|_{\mathbf{L}^4}^2),\label{ttt4}
 \eea
 \bea
 \|\nabla^3 \theta\|&\leq & \|\chi'(\vartheta)\|_{L^\infty}\|\nabla^3\vartheta\|+3\|\chi''(\vartheta)\|_{L^\infty}\|\nabla^2\vartheta\|_{\mathbf{L}^4}\|\nabla \vartheta\|_{\mathbf{L}^4}\non\\
 && +\|\chi'''(\vartheta)\|_{L^\infty}\|\nabla \vartheta\|_{\mathbf{L}^6}^3\non\\
 &\leq& C(\|\nabla^3 \vartheta\|+\|\nabla^2\vartheta\|_{\mathbf{L}^4}\|\nabla \vartheta\|_{\mathbf{L}^4}+\|\nabla \vartheta\|_{\mathbf{L}^6}^3),\label{ttt5}
 \eea
 \bea
\|\nabla^3 \vartheta\|&\leq& C(\|\nabla^3 \theta\|+\|\nabla^2\theta\|_{\mathbf{L}^4}\|\nabla \theta\|_{\mathbf{L}^4}+\|\nabla \theta\|_{\mathbf{L}^6}^3),\label{ttt6}
\eea
 where the constant $C$ only depends on the domain $\Omega$ as well as the upper and lower bounds of the modified thermal diffusivity $\kappa$ given in subsection \ref{uplow} (recall also Remark \ref{uplowremark}).

The following elementary estimates on parabolic equation with convection will be useful in the subsequent proofs (we refer to \cite[Section 4]{SZ13} for the case $n=2$, while the case $n=3$ can be proved in a similar way using the Sobolev embedding theorems in $3D$):
\bl \label{newt}
For $n=2,3$, consider the following parabolic problem
\be
\begin{cases}
 \vartheta_t+\u\cdot \nabla \vartheta-\tilde{\kappa}(\vartheta)\Delta \vartheta=0,\quad (t,x)\in (0,T)\times \Omega,\\
 \vartheta|_\Gamma=0, \qquad \qquad \qquad \qquad \quad (t,x)\in  (0,T)\times\Gamma,\\
 \vartheta|_{t=0}=\vartheta_0, \qquad \qquad\qquad \qquad  x\in \Omega.
 \end{cases}\label{newtheta}
\ee
Suppose that $\tilde{\kappa}(\cdot)$ is a smooth function with positive upper and lower bounds $0<\kappa_L\leq \tilde{\kappa}(s)\leq \kappa_U<+\infty$ for $s\in \mathbb{R}$, then the solution $\vartheta$ to problem \eqref{newtheta} satisfies the following differential inequality
\be
\frac{d}{dt}\|\nabla \vartheta\|^2+\frac{1}{\kappa_U}\|\vartheta_t\|^2\leq
\begin{cases}
C(1+\|\u\|^2\|\nabla \u\|^2)\|\nabla \vartheta\|^2,\quad \text{for}\ n=2,\\
C(1+\|\nabla \u\|^4)\|\nabla \vartheta\|^2 ,\qquad \ \ \ \text{for}\ n=3,\\
\end{cases}
\ee
where the constant $C$ depends on $\Omega, n, \kappa_U$. Besides, it holds
\be
\|\nabla^2 \vartheta\|\leq \begin{cases}
C(\|\vartheta_t\|+\|\u\|\|\nabla \u\|\|\nabla \vartheta\|),\quad \text{for}\ n=2,\\
C(\|\vartheta_t\|+\|\nabla \u\|^2\|\nabla \vartheta\|),\qquad\,\text{for}\ n=3,
\end{cases}\label{h2the}
\ee
where the constant $C$ depends on $\Omega, n, \kappa_L$.
\el

\subsection{\emph{A priori} estimates}

In this section, we derive a specific differential inequality that will be crucial in obtaining higher-order estimates for the strong solutions to problem \eqref{1}--\eqref{ini}.

Inspired by \cite{LB99}, it is useful to introduce the following shifted temperature
\be
\hat{\theta}=\theta-\theta_0,\label{hattheta}
\ee
 which satisfies the following parabolic equation subject to homogeneous Dirichlet boundary condition and zero initial datum:
\be
 \begin{cases}
&\hat{\theta}_t+ \u\cdot \nabla \hat{\theta}-\nabla \cdot(\kappa(\hat{\theta}+\theta_0)\nabla \hat{\theta})=-\u\cdot \nabla \theta_0+\nabla\cdot(\kappa(\hat{\theta}+\theta_0)\nabla \theta_0),\\
&\qquad \qquad\qquad\qquad\qquad\qquad\qquad\qquad\qquad \quad \ (t,x)\in (0,T)\times \Omega,\\
&\hat{\theta}|_\Gamma=0, \qquad\qquad \qquad  \ (t,x)\in  (0,T)\times\Gamma, \\
&\hat{\theta}|_{t=0}=0,\qquad\qquad\quad \ \ x\in \Omega.
\end{cases}\label{hateqthe}
 \ee
Next, we define the functional
\bea
\mathcal{H}(t)&=&\|\u\|^2+\|\nabla \u\|^2+2\int_\Omega \mu(\theta)|\mathcal{D} \u|^2 dx+a\lambda_0\|\nabla \phi\|^2+2a\lambda_0 \int_\Omega W(\phi) dx\non\\
&& +\|\Delta \phi-W'(\phi)\|^2+\|\nabla \hat \theta\|^2+\|\theta_t\|^2.\label{H}
\eea
Then we can deduce that
\begin{lemma}\label{high3d}
Suppose that $n=2,3$, $\theta_0\in H^2(\Omega)\cap H^1_0(\Omega)$, $|\phi_0|\leq 1$  in $\Omega$, $|\phi_b|\leq 1$ on $\Gamma$ and $\phi_b \in H^\frac52(\Gamma)$. Moreover,
we assume that the non-constant viscosity $\mu(\cdot)$ and thermal diffusivity $\kappa(\cdot)$ are taken as in subsection \ref{uplow} satisfying \eqref{bdee}-\eqref{bdeeb}. Let $(\u, \phi, \theta)$ be a smooth solution to problem \eqref{1}--\eqref{ini}. Then the following differential inequality holds:
 \bea
 &&\frac{d}{dt}\mathcal{H}(t) +\|\u_t\|^2+\gamma\|\nabla(\Delta\phi-W'(\phi))\|^2+\underline{\kappa}\|\Delta \hat\theta\|^2 +\underline{\kappa}\|\nabla \theta_t\|^2 \non\\
 && \quad +\big[\,\underline{\mu}-2\nu(\|\nabla \hat{\theta}(t)\|_{\mathbf{L}^4}+\|\nabla \theta_0\|_{\mathbf{L}^4})\big]\|S\u\|^2 \non\\
&\leq & C_{1}\big(\|\nabla \hat{\theta}(t)\|_{\mathbf{L}^4}+\|\nabla \theta_0\|_{\mathbf{L}^4}+1\big)^8\,\big(\mathcal{H}(t)+1\big)^3,\label{3d1}
  \eea
  where $\nu>0$ is an arbitrary constant and the constant $C_{1}$ depends on $\nu$, $\|\phi_0\|_{L^\infty}$, $\|\phi_b\|_{H^\frac52(\Gamma)}$, $\|\theta_0\|_{H^2}$, $\Omega$,  $\underline{\kappa}$, $\underline{\mu}$, $\gamma$ and coefficients of the system.
\end{lemma}
\begin{proof} For the sake of simplicity, in the remaining part of the proof we only treat the case $n=3$. Similar result can be obtained for $n=2$ with only minor modifications due to the Sobolev embedding theorems. It is important to note that we have assumed $\mu(\cdot)$, $\kappa(\cdot)$ are taken in such a way as in subsection \ref{uplow} (see Remark \ref{uplowremark}). Moreover, in the subsequent proof of this lemma, we shall use Sobolev embedding theorem to control the $L^\infty$-norm of $\theta$ instead of the maximum principle Lemma \ref{mtheta}, since the latter is not valid in the corresponding Galerkin approximate scheme (i.e., Type A in the Appendix).

The proof of Lemma \ref{high3d} consists of several steps.

\textbf{Step 1. Lower-order estimate for $\u$ and $\phi$}.  Multiplying equation \eqref{1} with $\u$ and equation \eqref{3} with
$-a\lambda_0(\Delta\phi-W'(\phi))$, respectively, integrating over $\Omega$ and adding the resultants together, we have
  \bea
&&\frac12\frac{d}{dt}\Big(\|\u\|^2+a\lambda_0\|\nabla\phi\|^2+2a\lambda_0\int_{\Omega}W(\phi)dx
 \Big)\non\\
 &&\quad + 2\int_\Omega \mu(\theta) |\mathcal{D}{\u}|^2 dx +a\lambda_0\gamma\|\Delta\phi-W'(\phi)\|^2 \non\\
&=&\int_{\Omega}\big[\lambda(\theta)\nabla\phi\otimes\nabla\phi\big]:
\nabla \u\,dx+\int_{\Omega}\theta \mathbf{g} \cdot
\u\,dx \non\\
&& \quad + a\lambda_0\int_{\Omega}(\u\cdot\nabla\phi)(\Delta\phi-W'(\phi)) \,dx\non\\
&:=& J_1+J_2+J_3.\label{up1}
 \eea
By Poincar\'e's inequality,  we easily get
 \bea
J_2 &\leq & |\mathbf{g}|\|\theta\|\|\u\|
\leq c_P|\mathbf{g}|\|\theta\|\|\nabla \u\| \non\\
& \leq &\frac{\underline{\mu}}{4}\|\nabla \u\|^2+\frac{c_P^2|\mathbf{g}|^2}{\underline{\mu}}\|\theta\|^2,\label{J2}
 \eea
where the constant $c_P$ only depends  on $n$, $\Omega$. Next, using the identity
$$
\nabla \cdot (\nabla \phi \otimes \nabla \phi)=\Delta \phi \nabla \phi +\nabla\left(\frac{|\nabla \phi|^2}{2}\right),
$$
together with the H\"{o}lder inequality, Sobolev embedding theorem, Poincar\'e's inequality and Young's inequality, after integration by parts, we
deduce that
\bea && J_1+J_3\non\\
&=&
-a\lambda_0\int_{\Omega}\u\cdot\nabla\left(\frac{|\nabla\phi|^2}{2}+W'(\phi)\right)dx
-b\lambda_0\int_{\Omega} \theta (\nabla\phi\otimes \nabla\phi):\nabla \u\,dx\non\\
&\leq& |b|\lambda_0\|\theta\|_{L^\infty}\|\nabla
\u\|\|\nabla\phi\|_{\mathbf{L}^4}^2  \non\\
&\leq& \frac{\underline{\mu}}{4}\|\nabla
\u\|^2+\frac{c|b|^2\lambda_0^2}{\underline{\mu}}\|\nabla \theta\|_{\mathbf{L}^4}^2\|\nabla\phi\|_{\mathbf{L}^4}^4\non\\
&\leq& \frac{\underline{\mu}}{4}\|\nabla \u\|^2+\frac{c|b|^2\lambda_0^2}{\underline{\mu}}\|\nabla \theta\|_{\mathbf{L}^4}^2\|\phi\|_{H^2}^2\|\phi\|_{L^\infty}^2\non\\
&\leq& \frac{\underline{\mu}}{4}\|\nabla \u\|^2+\frac{c|b|^2\lambda_0^2}{\underline{\mu}}\|\nabla \theta\|_{\mathbf{L}^4}^2\|\phi\|_{L^\infty}^2\non\\
&&\quad \times \Big(\|\Delta\phi-W'(\phi)\|^2+\|W'(\phi)\|^2+\|\phi\|^2+\|\phi_b\|_{H^\frac32(\Gamma)}^2\Big).\label{J1J3}
 \eea
In \eqref{J1J3}, we have used the fact $\|\theta\|_{L^\infty}\leq c\|\nabla \theta\|_{\mathbf{L}^4}$ for any $\theta\in H^2(\Omega)\cap H^1_0(\Omega)$ and the following estimate derived from the Gagliardo--Nirenberg inequality ($n=3$)
 \be
 \|\nabla \phi\|_{\mathbf{L}^4}\leq c\|\phi\|_{H^2}^\frac12\|\phi\|_{L^\infty}^\frac12, \quad \forall\, \phi\in H^2,\label{GN}
 \ee
with the elliptic estimate $\|\phi\|_{H^2}\leq c(\|\Delta \phi\|+\|\phi\|+\|\phi_b\|_{H^\frac32(\Gamma)})$.

Hence, we can conclude from the estimates \eqref{J2}, \eqref{J1J3}, Lemma \ref{mphi}, Poincar\'e's inequality,
the fact $2\|\mathcal{D}\u\|^2=\|\nabla \u\|^2$ and the variable transformation \eqref{hattheta} that
\bea
&& \frac12 \frac{d}{dt}\Big(\|\u\|^2+a\lambda_0\|\nabla\phi\|^2+2a\lambda_0\int_{\Omega}W(\phi)dx
 \Big)+\frac{\underline{\mu}}{2}\|\nabla{\u}\|^2 \non\\
 & \leq & C\|\nabla \theta\|_{\mathbf{L}^4}^2(\|\Delta\phi-W'(\phi)\|^2+1)+C\|\theta\|^2\non\\
 &\leq&  C (\|\nabla \hat{\theta}\|_{\mathbf{L}^4}^2+\|\nabla \theta_0\|_{\mathbf{L}^4}^2)\|\Delta\phi-W'(\phi)\|^2+C(\|\hat\theta\|^2+\|\theta_0\|^2)\non\\
 &\leq&  C (\|\nabla \hat{\theta}\|_{\mathbf{L}^4}^2+\|\nabla \theta_0\|_{\mathbf{L}^4}^2)\|\Delta\phi-W'(\phi)\|^2 +C\|\nabla \hat\theta\|^2+C.\label{3d3}
 \eea
 \par
 \textbf{Step 2. $\mathbf{H}^1$-estimate for $\u$.} Since $\u_t \in
\mathbf{H}$, it follows that $-(\Delta \u, \u_t)=(S\u, \u_t)$.
Then multiplying equation \eqref{1} by $S\u=-\Delta \u +\nabla \pi$ and integrating over $\Omega$, we get
 \bea &&\frac12\frac{d}{dt}\|\nabla \u\|^2+\int_\Omega \mu(\theta) |S{\u}|^2 dx \non\\
&=&-\int_\Omega (\u\cdot\nabla \u)\cdot S{\u} dx
   +2\int_\Omega \mu'(\theta)(\nabla \theta \cdot \mathcal{D}\u)\cdot S\u dx
   +\int_\Omega \mu(\theta)\nabla \pi\cdot S\u dx\non\\
&&\quad -\int_\Omega \{\nabla \cdot[\lambda(\theta) (\nabla \phi\otimes \nabla \phi)]\} \cdot  S\u dx
  + \int_\Omega \theta \mathbf{g}\cdot S\u dx\non\\
&=&\sum_{m=1}^5 I_m. \label{ttt}
 \eea

In what follows, we denote by $\epsilon>0$ a constant that can be chosen arbitrary small if necessary.
Using the H\"older inequality, the Sobolev embedding theorem, \eqref{Stokes II}, \eqref{hattheta} and Young's inequality, the terms $I_1, I_2, I_5$ can be estimated as follows:
\bea
I_1&\leq& c\|\u\|_{\mathbf{L}^6}\|\nabla \u\|_{\mathbf{L}^3}\|S\u\|\leq c\|\nabla \u\|^\frac32\|S \u\|^\frac32\non\\
&\leq& \epsilon \|S \u\|^2+ C\|\nabla \u\|^6,\non
\eea
\bea
I_2&\leq& \|\mu'(\theta)\|_{L^\infty}\|\nabla \theta\|_{\mathbf{L}^4}\|\nabla \u\|_{\mathbf{L}^4}\|S \u\|\non\\
&\leq& C\|\nabla \theta\|_{\mathbf{L}^4}\|\nabla \u\|^\frac14\|\Delta \u\|^\frac34\|S\u\|\non\\
&\leq& C(\|\nabla \hat{\theta}\|_{\mathbf{L}^4}+\|\nabla \theta_0\|_{\mathbf{L}^4})\|\nabla \u\|^\frac14\|S\u\|^\frac74\non\\
&\leq& \epsilon \|S\u\|^2+ C(\|\nabla \hat{\theta}\|_{\mathbf{L}^4}^8+\|\nabla \theta_0\|_{\mathbf{L}^4}^8)\|\nabla \u\|^2,\non
\eea
\bea
I_5&\leq& \epsilon\|S\u\|^2+C\|\theta\|^2\leq \epsilon\|S\u\|^2+C\|\hat{\theta}\|^2+C\|\theta_0\|^2\non\\
&\leq& \epsilon\|S\u\|^2+C\|\nabla \hat{\theta}\|^2+C.\non
\eea

Concerning the term $I_3$, using integration by parts and the estimates \eqref{Stokes I}, \eqref{Stokes II} for the pressure,  we have
\bea
I_3&=& -\int_\Omega \pi \mu'(\theta)\nabla\theta\cdot  S\u dx\non\\
&\leq& \|\mu'(\theta)\|_{L^\infty}\|\pi\|_{L^4}\|\nabla \theta\|_{\mathbf{L}^4}\|S\u\|\non\\
&\leq& C\|\pi\|_{H^1}^\frac34\|\pi\|^\frac14 \|\nabla \theta\|_{\mathbf{L}^4} \|S\u\|\non\\
&\leq& C\|S \u\|^\frac34(\nu\|S\u\|+C_\nu\|\nabla \u\|)^\frac14(\|\nabla \hat{\theta}\|_{\mathbf{L}^4}+ \|\nabla \theta_0\|_{\mathbf{L}^4})\|S\u\|\non\\
&\leq& \epsilon\|S\u\|^2+ \nu (\|\nabla \hat{\theta}\|_{\mathbf{L}^4}+ \|\nabla \theta_0\|_{\mathbf{L}^4})\|S\u\|^2+ C(\|\nabla \hat{\theta}\|_{\mathbf{L}^4}^8+ \|\nabla \theta_0\|_{\mathbf{L}^4}^8)\|\nabla \u\|^2,\non
\eea
where the constant $\nu>0$ is arbitrary and in particular, it is independent of $\epsilon$.

Next, for $I_4$, we have
\bea
 I_4 &\leq&  \|\nabla \cdot[\lambda(\theta) \nabla \phi\otimes\nabla \phi]\|\|S\u\|\non\\
    & \leq &  \epsilon\|S\u\|^2+ C\|\nabla \cdot[\lambda(\theta) \nabla \phi\otimes\nabla \phi]\|^2\non\\
 &\leq&  \epsilon\|S\u\|^2 + C\|\lambda'(\theta)\|_{L^\infty}^2\|\nabla \theta\|_{\mathbf{L}^4}^2\|\nabla \phi\|_{\mathbf{L}^8}^4\non\\
 && \quad + C \|\lambda(\theta)\|_{L^\infty}^2(\|\Delta\phi\|^2+\|\nabla^2 \phi\|^2)\|\nabla\phi\|_{\mathbf{L}^\infty}^2\non\\
 &:=&  \epsilon\|S\u\|^2+I_{4a}+I_{4b}.\non
 \eea
It follows from the Gagliardo--Nirenberg inequality, Lemma \ref{mphi} and Young's inequality that
 \bea
 I_{4a} &\leq& C\|\nabla \theta\|_{\mathbf{L}^4}^2\|\nabla \phi\|_{\mathbf{L}^6}^\frac72\|\nabla \phi\|_{\mathbf{H}^2}^\frac12\non\\
 &\leq& C(\|\nabla \hat{\theta}\|_{\mathbf{L}^4}^2+\|\nabla \theta_0\|_{\mathbf{L}^4}^2)
        \big(\|\Delta \phi\|+\|\phi\|+\|\phi_b\|_{H^\frac32(\Gamma)}\big)^\frac72\non\\
 &&\quad \times \big(\|\Delta\phi\|_{H^1}+\|\phi\|+\|\phi_b\|_{H^\frac52(\Gamma)}\big)^\frac12  \non\\
  &\leq& C(\|\nabla \hat{\theta}\|_{\mathbf{L}^4}^2+\|\nabla \theta_0\|_{\mathbf{L}^4}^2) (\|\Delta\phi-W'(\phi)\|+C)^\frac72\non\\
  &&\quad \times (\|\nabla(\Delta\phi-W'(\phi))\|+\|\Delta\phi-W'(\phi)\|+\|W''(\phi)\nabla \phi\|+\|W'(\phi)\|+C)^\frac12 \non\\
  &\leq& \epsilon\|\nabla(\Delta\phi-W'(\phi))\|^2+ C(\|\nabla \hat{\theta}\|_{\mathbf{L}^4}^2+\|\nabla \theta_0\|_{\mathbf{L}^4}^2)^\frac43 (\|\Delta\phi-W'(\phi)\|^{\frac{14}{3}}+C)\non\\
  &&\quad + C(\|\nabla \hat{\theta}\|_{\mathbf{L}^4}^2+\|\nabla \theta_0\|_{\mathbf{L}^4}^2) (\|\Delta\phi-W'(\phi)\|^4+\|\nabla \phi\|^4+C)\non\\
  &\leq& \epsilon\|\nabla(\Delta\phi-W'(\phi))\|^2+ C(\|\nabla \hat{\theta}\|_{\mathbf{L}^4}^4+\|\nabla \theta_0\|_{\mathbf{L}^4}^4)(\|\Delta\phi-W'(\phi)\|^{6}+\|\nabla \phi\|^6)\non\\
  &&\quad + C(\|\Delta\phi-W'(\phi)\|^{2}+\|\nabla \phi\|^2) +C(\|\nabla \hat{\theta}\|_{\mathbf{L}^4}^4+\|\nabla \theta_0\|_{\mathbf{L}^4}^4+1),\non
 \eea
 and
 \bea
 I_{4b}
 &\leq& C(1+\|\theta\|_{L^\infty}^2)(\|\Delta\phi\|^2+\|\nabla^2 \phi\|^2)\|\nabla \phi\|_{\mathbf{H}^2}\|\nabla \phi\|_{\mathbf{H}^1}
 \non\\
 &\leq& C(1+\|\hat{\theta}\|_{L^\infty}^2+\|\theta_0\|_{L^\infty}^2)\| \phi\|_{H^2}^3\|\phi\|_{H^3}
 \non\\
&\leq& C(1+\|\nabla \hat{\theta}\|_{\mathbf{L}^4}^2+\|\nabla \theta_0\|_{\mathbf{L}^4}^2)(\|\Delta\phi-W'(\phi)\|^3+C)\Big(\|\nabla(\Delta\phi-W'(\phi))\|\non\\
&&\quad +\|\Delta\phi-W'(\phi)\|+\|W''(\phi)\nabla \phi\|+\|W'(\phi)\|+\|\phi_b\|_{H^\frac52(\Gamma)}+\|\phi\|\Big)\non\\
&\leq&  \epsilon\|\nabla(\Delta\phi-W'(\phi))\|^2+C(1+\|\nabla \hat{\theta}\|_{\mathbf{L}^4}^4+\|\nabla \theta_0\|_{\mathbf{L}^4}^4)(\|\Delta\phi-W'(\phi)\|^{6}+1)\non\\
&&\quad  +C(1+\|\nabla \hat{\theta}\|_{\mathbf{L}^4}^2+\|\nabla \theta_0\|_{\mathbf{L}^4}^2)(\|\Delta\phi-W'(\phi)\|^{4}+\|\nabla \phi\|^4+1)\non\\
&\leq& \epsilon\|\nabla(\Delta\phi-W'(\phi))\|^2+ C(\|\nabla \hat{\theta}\|_{\mathbf{L}^4}^4+\|\nabla \theta_0\|_{\mathbf{L}^4}^4)(\|\Delta\phi-W'(\phi)\|^{6}+\|\nabla \phi\|^6)\non\\
  &&\quad + C(\|\Delta\phi-W'(\phi)\|^{6}+\|\nabla \phi\|^6)+C(\|\nabla \hat{\theta}\|_{\mathbf{L}^4}^4+\|\nabla \theta_0\|_{\mathbf{L}^4}^4+1).\non
 \eea
 As a result, we deduce from \eqref{ttt}, the above estimates and a further application of Young's inequality that
 \bea &&\frac12\frac{d}{dt}\|\nabla \u\|^2+ \Big[\underline{\mu}-5\epsilon-\nu (\|\nabla \hat{\theta}\|_{\mathbf{L}^4}+ \|\nabla \theta_0\|_{\mathbf{L}^4})\Big]\|S{\u}\|^2 \non\\
&\leq& 2\epsilon\|\nabla(\Delta\phi-W'(\phi))\|^2 \non\\
&& \quad  +C(\|\nabla \hat{\theta}\|_{\mathbf{L}^4}^8+\|\nabla \theta_0\|_{\mathbf{L}^4}^8)(\|\nabla \u\|^2+\|\Delta\phi-W'(\phi)\|^{6}+\|\nabla \phi\|^6)\non\\
&& \quad  +C(\|\nabla \u\|^6+\|\Delta\phi-W'(\phi)\|^{6}+\|\nabla \phi\|^6+\|\nabla \hat{\theta}\|^2)\non\\
&& \quad  +C(\|\nabla \hat{\theta}\|_{\mathbf{L}^4}^4+\|\nabla \theta_0\|_{\mathbf{L}^4}^4+1).\label{3d5}
 \eea

 On the other hand, multiplying the equation \eqref{1} by $\u_t$ and integrating over $\Omega$, we obtain that
 \bea &&\frac{d}{dt}\int_\Omega \mu(\theta)|\mathcal{D}\u|^2dx+\|\u_t\|^2 \non\\
&=&-\int_\Omega(\u\cdot \nabla \u) \cdot \u_t dx +\int_\Omega \mu'(\theta)\theta_t |\mathcal{D} \u|^2 dx\non\\
&&\quad -\int_\Omega \nabla \cdot[\lambda(\theta) (\nabla \phi\otimes \nabla \phi)]\cdot  \u_t dx  + \int_\Omega \theta \mathbf{g}\cdot \u_t dx\non\\
&=&\sum_{m=1}^4 K_m. \label{tA-part1}
 \eea
The terms $K_1, K_2 , K_4$ on the right-hand side of \eqref{tA-part1} can be estimated as follows:
 \bea K_1 &\leq& \|\u_t\|\|\u\|_{\mathbf{L}^6}\|\nabla \u\|_{\mathbf{L}^3}\non\\
 &\leq& c\|\u_t\|\|\nabla \u\|^\frac32\|\nabla \u\|_{\textbf{H}^1}^\frac12
 \non\\
 &\leq&
\epsilon\|\u_t\|^2+ \epsilon\|S \u\|^2+ C\|\nabla \u\|^6, \non
 \eea
 \bea
 K_2&\leq& c\|\mu'(\theta)\|_{L^\infty}\|\theta_t\|\|\mathcal{D} \u\|_{\mathbf{L}^4}^2 \non\\
 &\leq& C\|\theta_t\|\|\nabla \u\|^\frac12\|\nabla \u\|_{\mathbf{H}^1}^\frac32\non\\
 &\leq& \epsilon \|S \u\|^2+ C (\|\nabla \u\|^6+\|\theta_t\|^6),\non
 \eea
 \bea
 K_4 &\leq& C\|\theta\|\|\u_t\|\non\\
 &\leq&  \epsilon\|\u_t\|^2+C(\|\hat{\theta}\|^2+\|\theta_0\|^2)\non\\
 &\leq& \epsilon\|\u_t\|^2+C(\|\nabla \hat{\theta}\|^2+1).\non
 \eea
 For the third term $K_3$, we have
 \bea
 K_3 &\leq&  \|\nabla \cdot(\lambda(\theta) \nabla \phi\otimes\nabla \phi)\|\|\u_t\|\non\\
    & \leq &
     \epsilon\|\u_t\|^2+ C\|\nabla \cdot(\lambda(\theta) \nabla \phi\otimes\nabla \phi)\|^2,\label{Kes3}
 \eea
 where the last term on the right-hand side of \eqref{Kes3} can be estimated exactly as $I_{4a}$ and $I_{4b}$ above.
 Then we infer from \eqref{tA-part1}, the above estimates and Young's inequality that
\bea &&\frac{d}{dt}\int_\Omega \mu(\theta)|\mathcal{D}\u|^2dx+(1-3\epsilon)\|\u_t\|^2 \non\\
&\leq&  2\epsilon\|S\u\|^2+2\epsilon\|\nabla(\Delta\phi-W'(\phi))\|^2\non\\
&& \quad +C(\|\nabla \hat{\theta}\|_{\mathbf{L}^4}^8+\|\nabla \theta_0\|_{\mathbf{L}^4}^8)(\|\Delta\phi-W'(\phi)\|^{6}+\|\nabla \phi\|^6)\non\\
&& \quad +C(\|\nabla \u\|^6+\|\theta_t\|^6+\|\Delta\phi-W'(\phi)\|^{6}+\|\nabla \phi\|^6+\|\nabla \hat{\theta}\|^2)\non\\
&& \quad +C(\|\nabla \hat{\theta}\|_{\mathbf{L}^4}^4+\|\nabla \theta_0\|_{\mathbf{L}^4}^4+1). \label{3d5a}
 \eea

\par \textbf{Step 3. $H^2$-estimate for $\phi$.}
Due to equation \eqref{3} and the boundary conditions \eqref{bc1}, it holds
$\Delta\phi-W'(\phi)\big|_{\Gamma}=0$ (similar situation can be found in \cite{LL95} for the liquid crystal system). Then we can compute that
\bea &&\frac12\frac{d}{dt} \|\Delta\phi-W'(\phi)\|^2+\gamma\|\nabla(\Delta\phi-W'(\phi))\|^2  \non\\
&=& -\gamma\int_\Omega W''(\phi)|\Delta\phi-W'(\phi)|^2 dx - 2\int_{\Omega}(\Delta\phi-W'(\phi))
\nabla \u :\nabla^2\phi \,dx \non\\
&&\quad -\int_\Omega (\Delta \u\cdot \nabla \phi) (\Delta \phi-W'(\phi)) dx\non\\
&=&\sum_{m=5}^7 K_m. \label{tA-part2}
 \eea
The first term $K_5$ can be simply estimated by using Lemma \ref{mphi}
 \be
 K_5\leq  \gamma \|W''(\phi)\|_{L^\infty}\|\Delta\phi-W'(\phi)\|^2
\leq C\|\Delta\phi-W'(\phi)\|^2. \non
 \ee
Then for terms $K_6$ and $K_7$, it follows from the H\"{o}lder inequality, Gagliardo--Nirenberg inequality, Poincar\'e's inequality and Young's inequality that
\bea
K_6
&\leq&C\|\Delta\phi-W'(\phi)\|_{L^6}\|\nabla
\u\|_{\mathbf{L}^3}\|\phi\|_{H^2}  \non\\
&\leq&C\|\nabla(\Delta\phi-W'(\phi))\|\|\nabla
\u\|^\frac12\|\nabla\u\|_{\mathbf{H}^1}^\frac12(\|\Delta\phi\|+\|\phi\|+\|\phi_b\|_{H^\frac32(\Gamma)})  \non\\
&\leq&\epsilon\|\nabla(\Delta\phi-W'(\phi))\|^2+C\|\nabla
\u\|\|\Delta \u\|\big(\|\Delta\phi-W'(\phi)\|^2+\|W'(\phi)\|^2+1\big)\non\\
 &\leq&   \epsilon\|\nabla(\Delta\phi-W'(\phi))\|^2+\epsilon \|S \u\|^2+C(\|\nabla
\u\|^2+\|\nabla \u\|^6+\|\Delta\phi-W'(\phi)\|^6), \non
 \eea
 \bea
 K_7&\leq& \|\Delta \u\| \| \nabla \phi\|_{\mathbf{L}^6} \| \Delta \phi-W'(\phi)\|_{L^3}\non\\
 &\leq& \epsilon \|S \u\|^2+C\| \phi\|_{H^2}^2\|\nabla (\Delta \phi-W'(\phi))\|\|\Delta \phi-W'(\phi)\|\non\\
 &\leq& \epsilon \|S \u\|^2+  \epsilon\|\nabla(\Delta\phi-W'(\phi))\|^2 \non\\
 &&\quad +C(\|\Delta\phi-W'(\phi)\|^6+\|\Delta\phi-W'(\phi)\|^2).\non
 \eea
 As a result, we infer from \eqref {tA-part2} and the above estimates that
 \bea &&\frac12\frac{d}{dt} \|\Delta\phi-W'(\phi)\|^2+(\gamma-2\epsilon)\|\nabla(\Delta\phi-W'(\phi))\|^2  \non\\
&\leq& 2\epsilon \|S \u\|^2 +C(\|\nabla \u\|^6+\|\Delta\phi-W'(\phi)\|^6)+C. \label{3d4}
 \eea

 \par\textbf{Step 4. $H^1$-estimate for $\theta$.} We shall estimate $H^1$-norm of the shifted temperature $\hat{\theta}$ instead of the original one $\theta$.
 Multiplying the equation in \eqref{hateqthe} for $\hat{\theta}$ by $-\Delta\hat{\theta}$, integrating over $\Omega$, we have
  \bea
&&\frac12\frac{d}{dt}\|\nabla\hat{\theta}\|^2 +\int_\Omega \kappa(\hat{\theta}+\theta_0)|\Delta\hat{\theta}|^2dx \non\\
&=&\int_{\Omega}(\u\cdot\nabla) (\hat{\theta}+\theta_0)\Delta\hat{\theta} \,dx -\int_\Omega \kappa(\hat{\theta}+\theta_0) \Delta \theta_0\Delta \hat{\theta} dx\non\\
 &&\quad -\int_\Omega \kappa'(\hat{\theta}+\theta_0) |\nabla (\hat{\theta}+\theta_0)|^2\Delta \hat{\theta} dx\non\\
 &:=&\sum_{m=8}^{10}K_m.\label{htheh1}
 \eea
Using the H\"{o}lder inequality, the Sobolev embedding theorem and Young's inequality, we obtain that
\bea
K_8&\leq& \|\u\|_{\mathbf{L}^4}\|\nabla (\hat{\theta}+\theta_0)\|_{\mathbf{L}^4}\|\Delta\hat{\theta}\|\non\\
   &\leq& C\|\nabla \u\|(\|\nabla \hat{\theta}\|_{\mathbf{L}^4}+\|\nabla \theta_0\|_{\mathbf{L}^4})\|\Delta\hat{\theta}\|\non\\
   &\leq& \epsilon \|\Delta\hat{\theta}\|^2+C\|\nabla \u\|^4+C(\|\nabla \hat{\theta}\|_{\mathbf{L}^4}^4+\|\nabla \theta_0\|_{\mathbf{L}^4}^4),\non
\eea
\bea
K_9 &\leq& \|\kappa(\hat{\theta}+\theta_0)\|_{L^\infty}\| \Delta \theta_0\|\|\Delta \hat{\theta}\|\non\\
&\leq& \epsilon\|\Delta \hat{\theta}\|^2+C,\non
\eea
\bea
K_{10} &\leq& \|\kappa'(\hat{\theta}+\theta_0)\|_{L^\infty}\|\nabla (\hat{\theta}+\theta_0)\|_{\mathbf{L}^4}^2\|\Delta \hat{\theta}\|\non\\
       &\leq& \epsilon \|\Delta \hat{\theta}\|^2+ C(\|\nabla \hat{\theta}\|^4_{\mathbf{L}^4}+\|\nabla \theta_0\|^4_{\mathbf{L}^4}).\non
\eea
Thus, we infer from \eqref{htheh1} and the above estimates that
\be
\frac12\frac{d}{dt}\|\nabla\hat{\theta}\|^2 +(\underline{\kappa}- 3\epsilon)\|\Delta\hat{\theta}\|^2
\leq C\|\nabla \u\|^4+ C(\|\nabla \hat{\theta}\|^4_{\mathbf{L}^4}+\|\nabla \theta_0\|^4_{\mathbf{L}^4}+1). \label{htheh1a}
 \ee

\par \textbf{Step 5. $L^2$-estimate for $\theta_t$.} Differentiating equation \eqref{4} for $\theta$ with respect to time, multiplying the resultant by $\theta_t$ and integrating over $\Omega$, we obtain
 \bea
&& \frac12\frac{d}{dt}\|\theta_t\|^2 +\int_\Omega \kappa(\theta) |\nabla \theta_t|^2dx
\non\\
& = &
 -\int_\Omega \kappa'(\theta)\theta_t \nabla \theta\cdot \nabla \theta_t dx -\int_\Omega (\u_t\cdot \nabla \theta)\theta_t dx-\int_\Omega (\u\cdot \nabla \theta_t)\theta_t dx\non\\
&:=&\sum_{m=11}^{13}K_m. \label{temperature-1}
 \eea
 Thanks to the incompressibility of $\u$, it is obvious that $K_{13}=0$. Next, we estimate the terms $K_{12}, K_{13}$ by using the Sobolev embedding theorem:
 \bea
K_{11}  &\leq& \|\kappa'(\theta)\|_{L^\infty}\|\theta_t\|_{L^4}\| \nabla \theta\|_{\mathbf{L}^4}\|\nabla \theta_t\|\non\\
&\leq& C\| \nabla \theta\|_{\mathbf{L}^4}\|\theta_t\|^\frac14\|\nabla \theta_t\|^\frac74\non\\
&\leq& \epsilon \|\nabla \theta_t\|^2+ C (\|\nabla \hat{\theta}\|_{\mathbf{L}^4}^8+\|\nabla \theta_0\|_{\mathbf{L}^4}^8)\|\theta_t\|^2,\non
 \eea
 \bea
 K_{12}&\leq& \|\u_t\|\|\nabla \theta\|_{\mathbf{L}^4}\|\theta_t\|_{L^4}\non\\
 &\leq& \epsilon \|\u_t\|^2+ C (\|\nabla \hat{\theta}\|_{\mathbf{L}^4}^2+\|\nabla \theta_0\|_{\mathbf{L}^4}^2)\|\nabla \theta_t\|^\frac32\|\theta_t\|^\frac12\non\\
 &\leq& \epsilon \|\u_t\|^2+ \epsilon \|\nabla \theta_t\|^2+ C (\|\nabla \hat{\theta}\|_{\mathbf{L}^4}^8+\|\nabla \theta_0\|_{\mathbf{L}^4}^8)\|\theta_t\|^2.\non
 \eea
The above estimates together with Young's inequality yield that
 \be\frac12\frac{d}{dt}\|\theta_t\|^2 +(\underline{\kappa}-2\epsilon) \|\nabla \theta_t\|^2
 \leq  \epsilon \|\u_t\|^2+  C (\|\nabla \hat{\theta}\|_{\mathbf{L}^4}^8+\|\nabla \theta_0\|_{\mathbf{L}^4}^8)\|\theta_t\|^2.\label{3d6}
 \ee

\textbf{Step 6.} Combing the differential inequalities \eqref{3d3}, \eqref{3d5}, \eqref{3d5a}, \eqref{3d4}, \eqref{htheh1a} and \eqref{3d6}, using Young's inequality and taking the coefficient $\epsilon>0$ sufficiently small in all these inequalities such that
\be
0<\epsilon<\frac{1}{18}\min\{\underline{\mu},\ \underline{\kappa},\ \gamma,\ 1\},\label{smaep}
\ee
then we can conclude
\bea
&& \frac{d}{dt}\mathcal{H}(t)+\underline{\mu}\|\nabla \u\|^2+\|\u_t\|^2+\gamma\|\nabla(\Delta\phi-W'(\phi))\|^2+\underline{\kappa}\|\Delta \hat\theta\|^2 +\underline{\kappa}\|\nabla \theta_t\|^2 \non\\
&& \quad +\Big[\underline{\mu}-2\nu(\|\nabla \hat{\theta}\|_{\mathbf{L}^4}+\|\nabla \theta_0\|_{\mathbf{L}^4})\Big]\|S\u\|^2 \non\\
&\leq& C(\|\nabla \hat{\theta}\|_{\mathbf{L}^4}^8+\|\nabla \theta_0\|_{\mathbf{L}^4}^8)(\|\nabla \u\|^2+\|\Delta \phi-W'(\phi)\|^6+\|\nabla \phi\|^6+\|\theta_t\|^2)\non\\
&&\quad +C(\|\nabla \u\|^6+\|\Delta \phi-W'(\phi)\|^6+\|\nabla \phi\|^6+\|\theta_t\|^6+\|\nabla \hat\theta\|^2)\non\\
&&\quad +C(\|\nabla \hat{\theta}\|^4_{\mathbf{L}^4}+\|\nabla \theta_0\|^4_{\mathbf{L}^4}+1)\non\\
&\leq& C(\|\nabla \hat{\theta}\|_{\mathbf{L}^4}^8+\|\nabla \theta_0\|_{\mathbf{L}^4}^8+1)(\|\nabla \u\|^6+\|\Delta \phi-W'(\phi)\|^6+\|\nabla \phi\|^6+\|\theta_t\|^6+\|\nabla \hat\theta\|^6) \non\\
&&\quad  +C(\|\nabla \hat{\theta}\|^8_{\mathbf{L}^4}+\|\nabla \theta_0\|^8_{\mathbf{L}^4}+1), \non
\eea
where $\mathcal{H}(t)$ is given in \eqref{H}. By the definition of $\mathcal{H}(t)$, we easily arrive at our conclusion \eqref{3d1}. The proof is complete.
\end{proof}

\subsection{Proof of Theorem \ref{str3D}}

 \subsubsection{Existence}
  Existence of local strong solutions can be proved by means of a suitable semi-Galerkin approximate scheme (for instance, Type A in the Appendix section), which preserves the maximum principle for the phase function $\phi$ (see, e.g., \cite{C09,LL95} for a similar argument for the liquid crystal system), but not for the temperature $\theta$.

 \textbf{Step 1. The auxiliary problem}. First, we consider the auxiliary initial boundary value problem (IBVP for short) of problem \eqref{1}--\eqref{ini}, in which the viscosity $\mu$ and the thermal diffusivity $\kappa$ have been modified as in subsection \ref{uplow} (see Remark \ref{uplowremark}).

 For each $m\in \mathbb{N}$, we can construct an approximate solution $(\u^m, \phi^m, \theta^m)$ on certain time interval $[0,T_m]$ (see Proposition \ref{ppn3} in Appendix).  On the other hand, the (formal) differential inequality \eqref{3d1} obtained in Lemma \ref{high3d} can be justified by this approximate scheme. Therefore, we are able to apply \eqref{3d1} to derive uniform \emph{a priori} estimates for the approximate solutions (and we drop the superscript $m$ below for the sake of simplicity).

 Since $(\u_0, \phi_0, \theta_0)\in \mathbf{V}\times H^2(\Omega)\times (H^2(\Omega)\cap H^1_0(\Omega))$, it is straightforward to verify that
 $$
 \mathcal{H}(0)\leq C_0(\|\u_0\|_{\mathbf{V}}, \|\phi_0\|_{H^2}, \|\theta_0\|_{H^2}):=C_0<+\infty.
 $$
 It follows from the definition of $\hat{\theta}$ that $\hat{\theta}(0)=0$ and as a result, we have $\|\nabla \hat{\theta}(0)\|_{\mathbf{L}^4}=0$. Then in the differential inequality \eqref{3d1}, we choose the positive constant $\nu$ as follows
\be
\nu= \frac{\underline{\mu}}{4(\|\nabla \theta_0\|_{\mathbf{L}^4}+1)},\label{smallnu}
\ee
where $\underline{\mu}$ is taken as in subsection \ref{uplow} (recall also Remark \ref{uplowremark}).
As a consequence,
 the constant $C_1$ on the right-hand side of \eqref{3d1} is now fixed after we set the value of $\nu$.

In \eqref{htheh1a}, taking $\epsilon>0$ sufficiently small (recall \eqref{smaep}), we get
\be
\|\nabla\hat{\theta}(t)\|^2\leq C\int_0^t(\|\nabla \u(\tau)\|^4+ \|\nabla \hat{\theta}(\tau)\|^4_{\mathbf{L}^4}) d\tau +Ct,\label{hatheh1}
\ee
where $C$ is independent of $t$.  On the other hand, it follows from \eqref{hateqthe} that
\bea
\|\Delta\hat{\theta}(t)\| &\leq& C\|\theta_t(t)\|+ C\|\nabla \u(t)\|^2+ C\|\nabla \hat{\theta}(t)\|^2_{\mathbf{L}^4}+C, \label{hatheh1a}
 \eea
 where $C$ may depend on $\underline{\kappa}$ (as taken in subsection \ref{uplow}), $\|\theta_0\|_{H^2}$ and $\Omega$.
 Then by the Gagliardo--Nirenberg inequality we infer that (using the fact $\hat{\theta}|_\Gamma=0$)
\bea
&& \|\nabla \hat{\theta}(t)\|_{\mathbf{L}^4}\non\\
&\leq&
 c\|\nabla \hat{\theta}(t)\|^\frac14\|\Delta \hat{\theta}(t)\|^\frac34 \non\\
 &\leq& C_2 \left[\int_0^t\big(\mathcal{H}(\tau)^2+ \|\nabla \hat{\theta}(\tau)\|^4_{\mathbf{L}^4}\big) d\tau +t\right]^\frac18\left(\mathcal{H}(t)^2+\|\nabla \hat{\theta}(t)\|^4_{\mathbf{L}^4}+1\right)^\frac38.\label{TheL4}
\eea

Now our aim is to prove that there exists a time $T_*>0$ such that
\be
\|\nabla \hat{\theta}(t)\|_{\mathbf{L}^4}<1,\quad \forall \,t\in [0,T^*].\label{smallL4}
\ee
Since $\|\nabla \hat{\theta}(0)\|_{\mathbf{L}^4}=0$, then by continuity, the inequality \eqref{smallL4} at least holds for sufficiently small time $t$.
First, it is easy to find a sufficiently small time $T^*\in (0,1]$ that the following inequalities are satisfied
\bea
&&4C_1\big(\|\nabla \theta_0\|_{\mathbf{L}^4}+2\big)^8(C_0+1)^2\,T^* \leq \frac12\ln\left(\frac32\right),\label{CC1}\\
&& C_2\left(4C_0^2+4C_0+3\right)^\frac12\left(T^*\right)^\frac18 \leq \frac12.\label{CC2}
\eea
Below we shall show that this time $T^*$ is exactly what we are looking for. The proof is done by a contradiction argument in the sprit of \cite{LB99}. Suppose that there is a $T_0\in (0, T^*)$ satisfying
\be
\|\nabla \hat{\theta}(t)\|_{\mathbf{L}^4}<1\quad \text{for}\ t\in [0, T_0), \quad \text{but}\ \|\nabla \hat{\theta}(T_0)\|_{\mathbf{L}^4}=1.\label{ctheL4}
\ee
It follows from \eqref{3d1}, \eqref{smallnu} and assumption \eqref{ctheL4} that for $t\in [0,T_0]$, it holds
\bea
 &&\frac{d}{dt}\left(\mathcal{H}(t)+1\right) +\frac{\underline{\mu}}{2}\|S\u\|^2+\|\u_t\|^2+\gamma\|\nabla(\Delta\phi-W'(\phi))\|^2\non\\
 &&\quad +\underline{\kappa}\|\Delta \hat\theta\|^2 +\underline{\kappa}\|\nabla \theta_t\|^2 \non\\
&\leq & C_{1}\left(\|\nabla \theta_0\|_{\mathbf{L}^4}+2\right)^8\left(\mathcal{H}(t)+1\right)^3.\label{3d1aa}
  \eea
Recall the following inequality (see \cite[Lemma 3.2]{LB99}):
\bl\label{Gron}
Let $\varphi(t)$ and $\psi(t)$ be nonnegative functions satisfying
\be
\begin{cases}
\varphi'(t)+\psi(t)\leq C_*\varphi(t)^3+h(t)\varphi(t)+f(t), \quad t\in (0,T],\\
\varphi(t)=\varphi_0\geq 0,
\end{cases}\non
\ee
where $h(\cdot)$ and $f(\cdot)$ are nonnegative continuous functions. Let $T_1\in (0,T]$ and
$$ C_*\int_0^{T_1} \left(2\varphi_0+2\int_0^s f(\tau)d\tau \right)^2ds+\int_0^{T_1} h(s) ds\leq \frac12\ln\left(\frac32\right).$$
Then, for all $t\in [0,T_1]$, the following inequality holds
$$
\varphi(t)+\int_0^t \psi(s)ds \leq 2\left(\varphi_0+\int_0^t f(s) ds\right).
$$
\el
\noindent Let us now take
\bea
\varphi(t)&=&\mathcal{H}(t)+1 \quad \text{with}\quad \varphi(0)=\mathcal{H}(0)+1=C_0+1,\non\\
\psi(t)&=&\frac{\underline{\mu}}{2}\|S\u\|^2+\|\u_t\|^2+\gamma\|\nabla(\Delta\phi-W'(\phi))\|^2+\underline{\kappa}\|\Delta \hat\theta\|^2 +\underline{\kappa}\|\nabla \theta_t\|^2,\non\\
C_* &=& C_1(\|\nabla \theta_0\|_{\mathbf{L}^4}+2)^8,\non\\
 f(t)&=& 0,\non\\
 h(t)&=& 0.\non
\eea
Then by \eqref{CC1}, we can verify that all the assumptions in Lemma \ref{Gron} are fulfilled on $[0,T_0]\subset [0,T^*]$ and as a consequence, it holds
\be
\mathcal{H}(t)\leq 2C_0+1,\quad \forall\, t\in [0,T_0].\label{HT0}
\ee
Hence, we deduce from \eqref{TheL4}, \eqref{CC2}, \eqref{ctheL4} and \eqref{HT0} that
\bea
&& \|\nabla \hat{\theta}(T_0)\|_{\mathbf{L}^4}\non\\
&\leq& C_2 \left(\int_0^{T_0}\big[\mathcal{H}(\tau)^2+ \|\nabla \hat{\theta}(\tau)\|^4_{\mathbf{L}^4}\big] d\tau +T_0\right)^\frac18
       \left(\mathcal{H}(T_0)^2+\|\nabla \hat{\theta}(T_0)\|^4_{\mathbf{L}^4}+1\right)^\frac38\non\\
&\leq&  C_2 \left[(4C_0^2+4C_0+3)T_0\right]^\frac18\left(4C_0^2+4C_0+3\right)^\frac38\non\\
&<& C_2 \left(4C_0^2+4C_0+3\right)^\frac12\left(T^*\right)^\frac18\non\\
&\leq& \frac12<1,\non
\eea
which leads to a contradiction with the definition of $T_0$ given in \eqref{ctheL4}.

As a result, for the time $T^*$ we have chosen above, \eqref{smallL4} is satisfied and the differential inequality \eqref{3d1aa} holds on $[0,T^*]$. Then it follows from Lemma \ref{Gron}  that $\mathcal{H}(t)\leq 2C_0+1$ for all $t\in [0,T^*]$. Using Lemma \ref{mphi}, the definition of $\mathcal{H}(t)$ (see \eqref{H}),  \eqref{hatheh1a}, \eqref{smallL4} and the elliptic estimate for $\phi$, we can deduce that
\bea
 && \|\u(t)\|_{\textbf{H}^1}+\|\phi(t)\|_{H^2}+ \|\hat{\theta}(t)\|_{H^2}+\|\theta_t(t)\|\leq C(T^*),\quad \forall\, t\in [0,T^*],\label{eeS1}\\
 && \int_0^{T^*} \left(\|\u(t)\|^2_{\textbf{H}^2}+\|\phi(t)\|_{H^3}^2+\|\u_t(t)\|^2+\|\theta_t(t)\|_{H^1}^2 \right) dt\leq C(T^*).\label{eeS2}
\eea
It follows from equation \eqref{3} and the above estimates that
\bea
\|\phi_t(t)\| &\leq & \|\u(t)\|_{\mathbf{L}^3}\|\nabla \phi(t)\|_{\mathbf{L}^6}+\gamma\|\Delta\phi(t)-W'(\phi(t))\|\non\\
&\leq& C\sup_{t\in[0,T^*]}\left(\|\u(t)\|_{\mathbf{H}^1}\|\phi(t)\|_{H^2}+ \gamma \mathcal{H}(t)^\frac12\right)\non\\
&\leq& C(T^*),\quad \forall\, t\in [0,T^*],\non
\eea
and
\bea
&&\int_0^{T^*} \|\nabla \phi_t(t)\|^2 dt\non\\
&\leq& 2\int_0^{T^*} \|\nabla(\u\cdot\nabla \phi)\|^2dt+2\gamma^2\int_0^{T^*}\|\nabla(\Delta\phi-W'(\phi))\|^2 dt\non\\
&\leq& C\int_0^{T^*} \big(\|\nabla \u\|_{\mathbf{L}^3}^2\|\nabla \phi\|_{\mathbf{L}^6}^2 +\|\u\|_{\mathbf{L}^3}^2\|\nabla^2\phi\|_{\mathbf{L}^6}^2\big)dt+C(T^*)\non\\
&\leq& C\sup_{t\in[0,T^*]}\|\phi(t)\|_{H^2}^2\int_0^{T^*}\|\nabla \u(t)\|_{\mathbf{H}^1}^2 dt\non\\
&& \quad + C\sup_{t\in[0,T^*]}\|\u(t)\|_{\mathbf{H}^1}^2\int_0^{T^*}\|\phi(t)\|_{H^3}^2 dt+C(T^*)\non\\
&\leq& C(T^*),\non
\eea
which yield the estimates for $\|\phi_t\|_{L^\infty(0,T^*; L^2(\Omega))}$ and $\|\phi_t\|_{L^2(0,T^*;H^1(\Omega))}$.

For every $m\in \mathbb{N}$, the approximate solution $(\u^m, \phi^m, \theta^m)$ constructed in the semi-Galerkin scheme (Type A in Appendix) fulfills the above estimates that are independent of the approximate parameter $m$ on the uniform time interval $[0,T^*]$. Thus, the approximate solution $(\u^m, \phi^m, \theta^m)$ can be extended from $[0,T_m]$ to $[0,T^*]$. Then one can apply the classical compactness argument to pass to the limit as $m \to +\infty$ and prove the existence of a local strong solution $(\u, \phi, \theta)$ to the \emph{auxiliary} IBVP of problem \eqref{1}--\eqref{ini} (keeping subsection \ref{uplow} and Remark \ref{uplowremark} in mind). Since this procedure is standard provided that the above uniform estimates are available, we omit the details here.

It remains to verify that the local strong solution satisfies $\theta\in L^2(0,T^*; H^3(\Omega))$. To this end, we apply the elliptic regularity theorem to equation \eqref{newtheta1} for the transformed temperature $\vartheta$ introduced in \eqref{KK}. Using the higher-order estimate \eqref{eeS1}, the Sobolev embedding theorem with those relations \eqref{ttt1}--\eqref{ttt4} and Young's inequality, we have for a.e. $t\in[0,T^*]$
\bea
 \|\vartheta\|_{H^3}
&\leq& c(\|\chi'(\vartheta)(\vartheta_t+\u\cdot\nabla \vartheta)\|_{H^1}+\|\vartheta\|)\non\\
&\leq& c\|\chi'(\vartheta)\|_{L^\infty}\|\vartheta_t+\u\cdot\nabla \vartheta\|_{H^1}
       +c\|\nabla \chi'(\vartheta)\|_{\mathbf{L}^6}\|\vartheta_t+\u\cdot\nabla \vartheta\|_{L^3}+c\|\vartheta\|\non\\
&\leq& c \underline{\kappa}^{-1} (\|\vartheta_t\|_{H^1}+\|\u\|_{\mathbf{L}^6}\|\nabla \vartheta\|_{\mathbf{L}^3}
       +\|\nabla \u\|_{\mathbf{L}^6}\|\nabla \vartheta\|_{\mathbf{L}^3}
       +\|\u\|_{\mathbf{L}^6}\|\nabla^2\vartheta\|_{\mathbf{L}^3})\non\\
&&  \quad   +c \Big\|\frac{\kappa'(\theta)}{\kappa(\theta)^2}\Big\|_{L^\infty}\|\nabla \vartheta\|_{\mathbf{L}^6}(\|\vartheta_t\|_{L^3}
       +\|\u\|_{\mathbf{L}^6}\|\nabla \vartheta\|_{\mathbf{L}^6})+c\|\vartheta\|\non\\
&\leq& C\|\vartheta_t\|_{H^1}+ C\|\u\|_{\mathbf{H}^1}\|\vartheta\|_{H^2}^\frac12\|\vartheta\|_{H^1}^\frac12+C\|\u\|_{\mathbf{H}^2}\|\vartheta\|_{H^2}^\frac12\|\vartheta\|_{H^1}^\frac12
       \non\\
&&  \quad   +C\|\u\|_{\mathbf{H}^1}\|\vartheta\|_{H^3}^\frac12\|\vartheta\|_{H^2}^\frac12
        +C\|\vartheta\|_{H^2}\big(\|\vartheta_t\|_{H^1}+\|\u\|_{\mathbf{H}^1}\|\vartheta\|_{H^2}\big)+c\|\vartheta\|\non\\
&\leq& \frac12 \|\vartheta\|_{H^3}+ C\|\theta_t\|_{H^1}+ C\|\u\|_{\mathbf{H}^2}+C,\label{varthH33D}
\eea
which yields that
\be
 \|\vartheta\|_{H^3}\leq C\|\theta_t\|_{H^1}+ C\|\u\|_{\mathbf{H}^2}+C,\label{varthH33Da}
 \ee
 where the constant $C$ in the above estimates depends on the constant $C(T^*)$ in \eqref{eeS1} and the domain $\Omega$.
 Then we infer from relations \eqref{ttt3}, \eqref{ttt5}, the estimates \eqref{eeS1}, \eqref{eeS2} and \eqref{varthH33Da} that
 $$\int_0^{T^*} \|\theta(t)\|_{H^3}^2 dt \leq C(T^*).$$

  \textbf{Step 2. The original problem.} Finally, we are able to apply the maximum principle for the temperature $\theta$ (see Lemma \ref{mtheta}) to conclude that the local strong solution $(\u, \phi, \theta)$ to the \emph{auxiliary} IBVP of problem \eqref{1}--\eqref{ini} indeed satisfies $\|\theta(t)\|_{L^\infty}\leq \|\theta_0\|_{L^\infty}$ for all $t\in [0,T^*]$. Therefore, it follows from subsection \ref{uplow} that the triple $(\u, \phi, \theta)$ is actually a local strong solution to the \emph{original} problem \eqref{1}--\eqref{ini}. Hence, the proof for the existence of a strong solution is complete.

\subsubsection{Continuous dependence and uniqueness}
 Uniqueness of local strong solutions to  problem \eqref{1}--\eqref{ini} follows from a continuous dependence result on the initial data. Let $(\u_1, \phi_1, \theta_1)$ and $(\u_2, \phi_2, \theta_2)$
be two strong solutions on the time interval $[0,T^*]$ that start from the initial data
 $(\u_{01}, \phi_{01}, \theta_{01})$, $(\u_{02}, \phi_{02}, \theta_{02})$ $\in \mathbf{V}
\times H^2(\Omega) \times \big(H^2(\Omega)\cap H^1_0(\Omega)\big)$, respectively. Denote the differences by
\bea
&& \bar{\u}=\u_1-\u_2, \ \ \bar{\phi}=\phi_1-\phi_2, \ \ \bar{\theta}=\theta_1-\theta_2,\non\\
&& \bar{\u}_0=\u_{01}-\u_{02}, \ \ \bar{\phi}_0=\phi_{01}-\phi_{02}, \ \ \bar{\theta}_0=\theta_{01}-\theta_{02}.\non
\eea
 We can see that $(\bar \u, \bar \phi, \bar \theta)$ satisfy
 \bea
 &&\langle \bar \u_t, \mathbf{v}\rangle_{\mathbf{V}', \mathbf{V}}+
 \int_\Omega (\u_1\cdot\nabla \bar \u + \bar \u\cdot\nabla \u_2)\cdot \mathbf{v} dx
 +2\int_\Omega\mu(\theta_1) \mathcal{D}\bar{\u} : \mathcal{D} \mathbf{v}dx \non\\
 && \quad +2 \int_\Omega [\mu(\theta_1)-\mu(\theta_2)]\mathcal{D} \u_2 : \mathcal{D} \mathbf{v} dx\non\\
 &=&\int_\Omega [\lambda(\theta_1)\nabla\phi_1\otimes\nabla\phi_1-\lambda(\theta_2)\nabla\phi_2\otimes\nabla\phi_2] : \nabla \mathbf{v} dx
 + \int_\Omega  \bar \theta \mathbf{g} \cdot \mathbf{v} dx,\, \forall\,
 v\in \mathbf{V},\label{dns}\\
 &&\bar \phi_t+ \u_1\cdot \nabla \bar \phi + \bar \u \cdot \nabla\phi_2=\gamma [\Delta \bar\phi- W'(\phi_1)+W'(\phi_2)],\label{dphi}\\
 &&\bar \theta_t+ \u_1 \cdot \nabla \bar \theta+\bar \u \cdot\nabla \theta_2=\nabla \cdot[\kappa(\theta_1)\nabla \bar\theta]+\nabla \cdot[(\kappa(\theta_1)-\kappa(\theta_2))\nabla \theta_2],\label{dthe}\\
 &&  \bar \phi|_\Gamma=0, \quad  \bar \theta|_\Gamma=0,\label{dbc1}\\
&& \bar \u|_{t=0}=\bar{\mathbf{u}}_0, \quad \bar \phi|_{t=0}=\bar{\phi}_0,  \quad \bar\theta|_{t=0}=\bar{\theta}_0. \label{dini}
 \eea
Taking $\mathbf{v}=\bar{\u}$ in \eqref{dns}, multiplying \eqref{dphi} by
$-\Delta\bar{\phi}$ and \eqref{dthe} by
$-\Delta \bar{\theta}$, respectively, integrating over $\Omega$ and adding up
these three resultants, then performing integration by parts and
using the incompressible condition for the velocities, we get
\bea
&&\frac12\frac{d}{dt}\big(\|\bar{\u}\|^2+\|\nabla\bar{\phi}\|^2+\|\nabla \bar{\theta}\|^2
\big) + 2\int_\Omega \mu(\theta_1)|\mathcal{D}\bar{\u}|^2dx\non\\
&&\quad +\gamma\|\Delta\bar{\phi}\|^2 +\int_\Omega \kappa(\theta_1)|\Delta \bar{\theta}|^2dx\non\\
&=&-\int_\Omega (\bar{\u}\cdot\nabla{\u_2})\cdot \bar{\u} dx
   -2\int_\Omega [\mu(\theta_1)-\mu(\theta_2)] \mathcal{D} \u_2 : \mathcal{D} \bar \u dx\non\\
&& + \int_\Omega \lambda(\theta_1)(\nabla\phi_1\otimes\nabla\phi_1-\nabla\phi_2\otimes\nabla\phi_2) : \nabla \bar \u dx\non\\
&& +\int_\Omega [\lambda(\theta_1)-\lambda(\theta_2)](\nabla\phi_2\otimes\nabla\phi_2): \nabla \bar \u dx +\int_\Omega \bar \theta \mathbf{g}\cdot \bar \u dx\non\\
&& + \int_\Omega (\u_1 \cdot \nabla \bar \phi) \Delta \bar \phi dx + \int_\Omega (\bar \u\cdot \nabla \phi_2 )\Delta \bar \phi dx + \gamma \int_\Omega [W'(\phi_1)-W'(\phi_2)]\Delta \bar \phi dx\non\\
&& + \int_\Omega (\u_1 \cdot \nabla \bar \theta) \Delta \bar \theta dx +  \int_\Omega (\bar \u \cdot \nabla \theta_2)\Delta \bar \theta dx -\int_\Omega \kappa'(\theta_1)(\nabla \theta_1\cdot\nabla \bar \theta)\Delta \bar\theta dx\non\\
&& -\int_\Omega [\kappa(\theta_1)-\kappa(\theta_2)]\Delta \theta_2 \Delta \bar \theta dx
   -\int_\Omega [\kappa'(\theta_1)\nabla \theta_1-\kappa'(\theta_2)\nabla \theta_2]\cdot \nabla \theta_2 \Delta \bar \theta dx
   \non\\
&:=& \sum_{m=1}^{13} D_m.\label{eni}
 \eea
Using the $L^\infty$-estimate for $\phi_i, \theta_i$ (recall Lemmas \ref{mphi}, \ref{mtheta}), the H\"older inequality, Agmon's inequality and the Sobolev embedding theorem ($n=3$), we proceed to estimate the right-hand side of \eqref{eni} term by term.
 \bea D_1&\leq& \|\bar{\u}\|_{\mathbf{L}^4}^2\|\nabla{\u_2}\|\non\\
     &\leq& C\|\nabla\bar{\u}\|^\frac32 \|\bar{\u}\|^\frac12\|\nabla \u_2\|\non\\
     &\leq& \epsilon \underline{\mu}\|\nabla\bar{\u}\|^2+C\|\nabla \u_2\|^4\|\bar{\u}\|^2,\non
 \eea
 \bea
 D_2&\leq& 2\|\mu(\theta_1)-\mu(\theta_2)\|_{L^\infty}\|\nabla \u_2\|\|\nabla \bar \u\|\non\\
 &\leq& C\|\mu'\|_{L^\infty}\|\bar\theta\|_{L^\infty}\|\nabla \u_2\|\|\nabla \bar \u\|\non\\
 &\leq& \epsilon \underline{\mu}\|\nabla\bar{\u}\|^2 +C\|\nabla  \u_2\|^2 \|\Delta \bar \theta\|\|\nabla \bar\theta\|\non\\
 &\leq& \epsilon \underline{\mu}\|\nabla\bar{\u}\|^2+\epsilon \underline{\kappa} \|\Delta \bar \theta\|^2+ C\|\nabla  \u_2\|^4\|\nabla \bar\theta\|^2,\non
 \eea
 \bea D_3 &\leq& \|\lambda(\theta_1)\|_{L^\infty}(\|\nabla \phi_1\|_{\mathbf{L}^6}+\|\nabla \phi_2\|_{\mathbf{L}^6})\|\nabla \bar\phi\|_{\mathbf{L}^3}\|\nabla \bar \u\|\non\\
 &\leq& C(\|\phi_1\|_{H^2}+\|\phi_2\|_{H^2})\|\nabla \bar\phi\|_{\mathbf{H}^1}^\frac12\|\nabla \bar\phi\|^\frac12\|\nabla \bar \u\|\non\\
 &\leq& \epsilon \underline{\mu}\|\nabla \bar \u\|^2+\epsilon\gamma\|\Delta\bar \phi\|^2+C (\|\phi_1\|_{H^2}^4+\|\phi_2\|_{H^2}^4)\|\nabla \bar \phi\|^2,\non
 \eea
 \bea
D_4&\leq& \lambda_0|b|\|\bar \theta\|_{L^\infty}\|\nabla \phi_2\|_{\mathbf{L}^4}^2\|\nabla\bar \u\|\non\\
&\leq& C\|\bar \Delta \theta\|^\frac12\|\nabla \bar \theta\|^\frac12\|\phi_2\|_{H^2}^2\|\nabla\bar \u\|\non\\
&\leq&\epsilon \underline{\mu}\|\nabla\bar{\u}\|^2+\epsilon \underline{\kappa} \|\Delta \bar \theta\|^2+ C\|\phi_2\|_{H^2}^4\|\nabla \bar\theta\|^2,\non
 \eea
 \be
 D_5\leq C\|\bar\theta\|\|\bar \u\|\leq C\|\bar \u\|^2+C\|\nabla \bar \theta\|^2,\non
 \ee
\bea
 D_6 +D_9 &\leq& \|\u_1\|_{\mathbf{L}^6}(\|\nabla \bar \phi\|_{\mathbf{L}^3}\|\Delta \bar \phi\|+\|\nabla \bar \theta\|_{\mathbf{L}^3}\|\Delta \bar \theta\|)\non\\
 &\leq& C\|\nabla \u_1\|(\|\nabla \bar\phi\|_{\mathbf{H}^1}^\frac12\|\nabla \bar\phi\|^\frac12\|\Delta \bar \phi\|+\|\nabla \bar \theta\|_{\mathbf{H}^1}^\frac12\|\nabla \bar \theta\|^\frac12\|\Delta \bar \theta\|)\non\\
 &\leq& \epsilon \gamma \|\Delta \bar \phi\|^2+ \epsilon \underline{\kappa}\|\Delta \bar\theta\|^2+C\|\nabla \u_1\|^4(\|\nabla \bar\phi\|^2+\|\nabla \bar\theta\|^2),\non
\eea
\bea
 D_7+D_{10}&\leq& \|\bar \u\|_{\mathbf{L}^3}(\|\nabla \phi_2\|_{\mathbf{L}^6}\|\Delta  \bar \phi\|+\|\nabla \theta_2\|_{\mathbf{L}^6}\|\Delta\bar \theta\|)\non\\
 &\leq& C\|\bar\u\|^\frac12\|\nabla \u\|^\frac12(\|\phi_2\|_{H^2}\|\Delta  \bar \phi\|+\|\theta_2\|_{H^2}\|\Delta \bar\theta\|)\non\\
 &\leq& \epsilon \underline{\mu}\|\nabla\bar{\u}\|^2+\epsilon \gamma \|\Delta \bar \phi\|^2+ \epsilon \underline{\kappa}\|\Delta \bar\theta\|^2+C(\|\phi_2\|_{H^2}^4+\|\theta_2\|_{H^2}^4)\|\bar \u\|^2,\non
\eea
 \bea
D_8 &\leq& \gamma \Big\|\bar\phi\int_0^1 W''(s
\phi_1+(1-s)\phi_2) ds\Big\|_{L^\infty}\|\Delta \bar{\phi}\| \non\\
&\leq&
C\gamma\|W''\|_{L^\infty}\|\bar{\phi}\|_{L^\infty}\|\Delta
\bar{\phi}\|\non\\
&\leq& C\gamma\|\Delta \bar \phi\|^\frac32 \|\nabla\bar\phi\|^\frac12 \non\\
&\leq&
\epsilon \gamma\|\Delta\bar{\phi}\|^2+C\|\nabla\bar{\phi}\|^2,\non
 \eea
 \bea D_{11} &\leq& \|\kappa'(\theta_1)\|_{L^\infty}\|\nabla \theta_1\|_{\mathbf{L}^6}\|\nabla \bar \theta\|_{\mathbf{L}^3}\|\Delta \bar \theta\|\non\\
 &\leq& C\|\theta_1\|_{H^2}\|\nabla \bar \theta\|^\frac12\|\Delta \bar \theta\|^\frac32\non\\
 &\leq& \epsilon\underline{\kappa}\|\Delta \bar \theta\|^2+ C \|\theta_1\|_{H^2}^4\|\nabla \bar \theta\|^2,\non
 \eea
 \bea
D_{12}&\leq&\|\kappa'\|_{L^\infty} \|\bar \theta\|_{L^\infty}\|\Delta\theta_2\|\|\Delta \bar \theta\|\non\\
&\leq& C\|\theta_2\|_{H^2} \|\nabla \bar \theta\|^\frac12\|\Delta \bar \theta\|^\frac32\non\\
&\leq& \epsilon \underline{\kappa}\|\Delta \bar \theta\|^2+ C\|\theta_2\|_{H^2}^4\|\nabla \bar\theta\|^2,\non
 \eea
\bea
D_{13}&\leq& \|\kappa'(\theta_1)\|_{L^\infty}\|\nabla \bar \theta\|_{\mathbf{L}^3}\|\nabla \theta_2\|_{\mathbf{L}^6}\| \Delta \bar \theta\|
           +\|\kappa''\|_{L^\infty}\|\bar \theta\|_{L^\infty}\|\nabla \theta_2\|_{\mathbf{L}^4}^2\| \Delta \bar \theta\|\non\\
&\leq& C\|\nabla \bar \theta\|_{\mathbf{H}^1}^\frac12\|\nabla \bar \theta\|^\frac12\|\theta_2\|_{H^2}\| \Delta \bar \theta\|
 + C\|\Delta \bar\theta\|^\frac12\|\nabla \bar\theta\|^\frac12\|\theta_2\|_{H^2}^2\|\Delta \bar \theta\|\non\\
&\leq& \epsilon \underline{\kappa}\|\Delta \bar \theta\|^2+C\|\theta_2\|_{H^2}^4\|\nabla \bar \theta\|^2.\non
\eea
Taking the constant $\epsilon>0$ to be sufficiently small in the above estimates, we infer from \eqref{eni} that
 \bea && \frac{d}{dt}(\|\bar{\u}\|^2+\|\nabla \bar{\phi}\|^2+\|\nabla \bar{\theta}\|^2)
+\underline{\mu}\|\nabla\bar{\u}\|^2+\gamma\|\Delta\bar{\phi}\|^2+\underline{\kappa}\|\Delta\bar{\theta}\|^2
\non\\
&\leq& Q(t)(\|\bar{\u}\|^2+\|\nabla \bar{\phi}\|^2+\|\nabla \bar{\theta}\|^2), \quad \forall\, t\in [0,T^*], \non
 \eea
  where $$Q(t)=C(\|\nabla \u_1\|^4+\|\nabla \u_2\|^4+\|\phi_1\|_{H^2}^4+ \|\phi_2\|_{H^2}^4+\|\theta_1\|_{H^2}^4+\|\theta_2\|_{H^2}^4)$$
with $C$ being a constant that may depend  on  $\Omega$, $\|\phi_0\|_{L^\infty}$, $\|\theta_0\|_{L^\infty}$ and
coefficients of the system.

Since the $\mathbf{V}\times H^2(\Omega)\times
H^2(\Omega)$-norm of these two strong solutions $(\u_i, \phi_i, \theta_i)$ ($i=1,2$) to problem
\eqref{1}--\eqref{ini} are bounded on $[0,T^*]$ (recall \eqref{eeS1}), then the quantity $Q(t)$ is also bounded on $[0,T^*]$ such that $Q(t)\leq C(T^*)$ for $t\in [0,T^*]$. As a consequence, it easily follows from the Gronwall inequality that
\bea
&&\|\bar{\u}(t)\|^2+\|\nabla \bar{\phi}(t)\|^2+\|\nabla \bar{\theta}(t)\|^2
+\int_0^t\Big(\underline{\mu}\|\nabla\bar{\u}\|^2+\gamma\|\Delta\bar{\phi}\|^2+\underline{\kappa}\|\Delta\bar{\theta}\|^2\Big)d\tau\non\\
&\leq& \left(\|\bar{\u}_0\|^2+\|\nabla \bar{\phi}_0\|^2+\|\nabla \bar{\theta}_0\|^2\right)e^{C(T^*)t},\quad \forall\, t\in [0,T^*].\label{conti}
\eea
The above estimate indicates that the local strong solution to problem \eqref{1}--\eqref{ini} continuously depends on its initial data in $\mathbf{H}\times H^1(\Omega)\times H^1(\Omega)$, which also immediately yields the uniqueness result.

The proof of Theorem \ref{str3D} is now complete.

\begin{remark}
Because of  the highly nonlinear structure of problem \eqref{1}--\eqref{ini}, we are not able to improve the continuous dependence result for strong solutions to the higher-order space like $\mathbf{V}\times H^2\times H^2$ when the spatial dimension is three. The two dimensional case might be possible, but the proof would be rather involved. We refer to \cite{huang} for a continuous dependence result in the $H^2$-topology for the $2D$ Boussinesq system with variable viscosity and thermal diffusivity (without coupling with the phase-field equation).
\end{remark}


\section{Global Solutions in $2D$}
\setcounter{equation}{0}

\subsection{Global weak solutions}
Similar to Section 3, the existence of global weak solutions can be proved by using a suitable semi-Galerkin approximate scheme, provided that certain uniform lower-order global-in-time estimates can be obtained. However, such global estimates are in general not available due to the lack of dissipative energy law for problem \eqref{1}--\eqref{ini} (cf. Step 1 in the proof of Lemma \ref{high3d}). In order to avoid this difficulty, we shall assume that the initial temperature variation i.e., $\|\theta_0\|_{L^\infty}$ is suitably bounded. Moreover, this bound will be preserved in the associated Galerkin approximation (i.e., Type B in Appendix).

\subsubsection{$L^\infty$-bound for the initial temperature}
 For any function $\phi\in H^2(\Omega)$ with $\phi|_{\Gamma}=\phi_b\in H^\frac32(\Gamma)$, we recall the following interpolation inequality as well as the elliptic estimate in $2D$:
\bea
\|\nabla \phi\|_{\mathbf{L}^4}^2 &\leq& c_1\|\phi\|_{H^2}\|\phi\|_{L^\infty},\label{GNaa}\\
\|\phi\|_{H^2} &\leq& c_2\Big(\|\Delta\phi\|+\|\phi\|+\|\phi_b\|_{H^\frac32(\Gamma)}\Big).\label{GNaab}
\eea
Similarly, we have the interpolation inequality in $2D$ for any $\theta\in H^2(\Omega)\cap H^1_0(\Omega)$:
\bea
\|\nabla \theta\|_{\mathbf{L}^4}^2 &\leq&  c_3\|\Delta \theta\|\|\theta\|_{L^\infty},\label{GNbb}
\eea
 In \eqref{GNaa}--\eqref{GNbb}, $c_1$, $c_2$ and $c_3$ are positive constants that only depend on $\Omega$.

Since the viscosity function $\mu(\cdot)$ is continuous and $\mu(s)>0$ for all $s\in \mathbb{R}$, we deduce that
 \be
 \Theta_1:= \sup_{l\geq 0}\left\{\min\left\{l,\ \ \frac{1}{2c_1c_2|b|}\sqrt{\frac{a\gamma\min_{s\in[-l, l]}\mu(s)}{2\lambda_0}}\right\}\right\}\label{Theta1}
 \ee
 is a finite positive constant. Indeed, by monotonicity, $\Theta_1$ satisfies the relation
 $$\Theta_1=\frac{1}{2c_1c_2|b|}\sqrt{\frac{a\gamma\min_{s\in[-\Theta_1, \Theta_1]}\mu(s)}{2\lambda_0}}.$$
 After fixing the number $\Theta_1$, we set (compare with subsection \ref{uplow})
 \bea
 && \underline{\mu}:=\min_{s\in[-\Theta_1, \Theta_1]}\mu(s)>0,\label{lomu}\\
 && \overline{\mu}:=\max_{s\in[-\Theta_1, \Theta_1]}\mu(s)>0.\label{higmu}
 \eea
Similarly, since the thermal diffusivity $\kappa(\cdot)$ is a $C^2$ function and $\kappa(s)>0$ for all $s\in \mathbb{R}$, we can set
 \bea
 &&\underline{\kappa}:=\min_{s\in[-\Theta_1, \Theta_1]}\kappa(s)>0,\label{lokap}\\
 &&\overline{\kappa}:=\max_{s\in[-\Theta_1, \Theta_1]}\kappa(s)>0.\label{higkap}
 \eea
 Moreover, there exists a constant $\Theta_2\in (0, \Theta_1]$ such that
\be
\Theta_2=\max\left\{l: \ l\in (0,\Theta_1] \ \ \text{and}\ \ \max_{s\in [-l, l]}l|\kappa'(s)|\leq \frac{\underline{\kappa}}{4c_3}\right\}. \label{2.7}
\ee
\begin{remark}\label{lubd2d}
In the remaining part of this section, we shall assume that the $L^\infty$-norm of the initial temperature $\theta_0$ satisfies either $\|\theta_0\|_{L^\infty}\leq \Theta_1$ or $\|\theta_0\|_{L^\infty}\leq \Theta_2\in (0, \Theta_1]$. It is easy to see from the definition that
$$\min_{s\in[-\Theta_1, \Theta_1]}\mu(s)\leq \min_{s\in[-\Theta_2, \Theta_2]}\mu(s)\leq \max_{s\in[-\Theta_2, \Theta_2]}\mu(s)\leq \max_{s\in[-\Theta_1, \Theta_1]}\mu(s),$$
$$\min_{s\in[-\Theta_1, \Theta_1]}\kappa(s)\leq \min_{s\in[-\Theta_2, \Theta_2]}\kappa(s)\leq \max_{s\in[-\Theta_2, \Theta_2]}\kappa(s)\leq \max_{s\in[-\Theta_1, \Theta_1]}\kappa(s).$$
Then by the weak maximum principle Lemma \ref{mtheta}, for both cases we can simply choose the positive lower and upper bounds $\underline{\mu}$, $\overline{\mu}$, $\underline{\kappa}$, $\overline{\kappa}$ for the viscosity $\mu(\cdot)$ and thermal diffusivity $\kappa(\cdot)$ as in \eqref{lomu}--\eqref{higkap}.
\end{remark}

\subsubsection{\emph{A priori} estimates}\label{sec:glow2D}
 After the above preparations, in what follows, we derive some \emph{a priori} estimates on lower-order norms of solutions to problem \eqref{1}--\eqref{ini}.
 \bp \label{BEL1}
 Let $n=2$. Suppose that $|\phi_0|\leq 1$ a.e. in $\Omega$, $|\phi_b|\leq 1$ on $\Gamma$ and  $\theta_0$ satisfies the assumption $\|\theta_0\|_{L^\infty}\leq \Theta_1$, where $\Theta_1$ is given in \eqref{Theta1}. Then for arbitrary $T\in(0,+\infty)$, the smooth solution $(\u, \phi, \theta)$ to
problem \eqref{1}--\eqref{ini} satisfies the following estimates
 \be
 \begin{cases}
 \|\u\|_{L^\infty(0,T; \mathbf{H})}+\|\u\|_{L^2(0,T;\mathbf{V})}\leq C_T,\\
 \|\phi\|_{L^\infty(0,T; H^1)}+\|\phi\|_{L^2(0,T; H^2)}\leq C_T,\\
 \|\theta\|_{L^\infty(0,T; H^1)}+\|\theta\|_{L^2(0,T; H^2)}+\|\theta_t\|_{L^2(0,T;L^2)}\leq C_T,
 \end{cases}\label{es1}
 \ee
 where $C_T$ is a constant depending on $\|\u_0\|$, $\|\phi_0\|_{H^1\cap L^\infty}$,  $\|\phi_b\|_{H^\frac32(\Gamma)}$, $\|\theta_0\|_{H^1}$, $\Theta_1$, $T$, $\Omega$ and coefficients of the system.
\ep
\begin{proof}
 Similar to Step 1 of the proof to Lemma \ref{high3d}, we still have the differential equality \eqref{up1}. Recalling the estimate \eqref{J1J3} for the sum $J_1+J_3$, using weak maximum principles for $\phi$, $\theta$ (see Lemmas \ref{mphi}, \ref{mtheta}), and the estimates \eqref{GNaa}, \eqref{GNaab}, \eqref{Theta1}, \eqref{lomu}, we obtain that
\bea
 J_1+J_3&\leq& \frac{\underline{\mu}}{4}\|\nabla \u\|^2+\frac{|b|^2\lambda_0^2}{\underline{\mu}}\|\theta\|_{L^\infty}^2\|\nabla\phi\|_{\mathbf{L}^4}^4\non\\
 &\leq& \frac{\underline{\mu}}{4}\|\nabla \u\|^2+ \frac{c_1^2|b|^2\lambda_0^2}{\underline{\mu}}\|\theta\|_{L^\infty}^2\|\phi\|_{H^2}^2\|\phi\|_{L^\infty}^2\non\\
 &\leq&  \frac{\underline{\mu}}{4}\|\nabla \u\|^2+\frac{4c_1^2c_2^2|b|^2\lambda_0^2}{\underline{\mu}}\| \theta\|_{L^\infty}^2\|\phi\|_{L^\infty}^2\non\\
&&\quad \times \Big(\|\Delta\phi-W'(\phi)\|^2+\|W'(\phi)\|^2+\|\phi\|^2+\|\phi_b\|_{H^\frac32(\Gamma)}^2\Big)\non\\
&\leq& \frac{\underline{\mu}}{4}\|\nabla \u\|^2+\frac{ a\lambda_0\gamma}{2}\|\Delta\phi-W'(\phi)\|^2+C.\label{J11J33}
 \eea
 The term $J_2$ in \eqref{up1} can be estimated in the same way as in \eqref{J2}.
 Therefore, we can deduce from \eqref{up1}, \eqref{J2} and \eqref{J11J33} that:
\bea
&&\frac{d}{dt}\Big(\|\u\|^2+a\lambda_0\|\nabla\phi\|^2+2a\lambda_0\int_{\Omega}W(\phi)dx
 \Big)\non\\
 &&\quad  +\underline{\mu}\|\nabla{\u}\|^2 +a\lambda_0\gamma\|\Delta\phi-W'(\phi)\|^2 \non\\
 &\leq& C,  \non
 \eea
 where $C$ is a constant depending on $\|\phi_0\|_{L^\infty}$,  $\|\phi_b\|_{H^\frac32(\Gamma)}$, $\|\theta_0\|_{L^\infty}$ and coefficients of the system.
 Integrating the above inequality with respective to time, we infer from Young's inequality that for any $T\in(0,+\infty)$, it holds
 \be
 \sup_{t\in[0,T]}(\|\u(t)\|^2+\|\phi(t)\|_{H^1}^2)+\int_0^T(\|\nabla \u(t)\|^2+\|\phi(t)\|_{H^2}^2) dt\leq C_T,\label{ueslow1}
 \ee
 where $C_T$ is a constant depending on $\|\u_0\|$, $\|\phi_0\|_{H^1\cap L^\infty}$,  $\|\phi_b\|_{H^\frac32(\Gamma)}$, $\|\theta_0\|_{L^\infty}$, $T$, $\Omega$ and coefficients of the system.

 Next, keeping the assumption $\|\theta_0\|_{L^\infty}\leq \Theta_1$ in mind, using the temperature transformation \eqref{KK}, we can apply Lemmas \ref{mtheta}, \ref{newt} to problem \eqref{newtheta1} with the diffusivity $\tilde{\kappa}(\vartheta)=\frac{1}{\chi'(\vartheta)}$ to get (here $n=2$)
\be
\frac{d}{dt}\|\nabla \vartheta\|^2 +\frac{1}{\overline{\kappa}} \|\vartheta_t\|^2\leq C(1+\|\u\|^2\|\nabla \u\|^2)\|\nabla \vartheta\|^2.\non
\ee
It follows from \eqref{ueslow1} and the Gronwall lemma that
\be
\|\nabla \vartheta(t)\|^2+\int_0^t\|\vartheta_t(\tau)\|^2 d\tau\leq C_T,\quad \forall \, t\in [0,T],\label{varthH1}
\ee
 which together with the estimate \eqref{h2the} ($n=2$), \eqref{ueslow1} and \eqref{varthH1} yields that
\be
\int_0^t \|\vartheta(\tau)\|_{H^2}^2 d\tau\leq C_T,\quad \forall \, t\in [0,T].\label{varthH1h2}
\ee
 Then we easily infer from \eqref{varthH1}, \eqref{varthH1h2}, the relations \eqref{ttt1}, \eqref{ttt3} and \eqref{GNbb}  that
\be
\|\nabla \theta(t)\|^2+\int_0^t(\|\theta_t(\tau)\|^2+\| \theta(\tau)\|_{H^2}^2) d\tau\leq C_T,\quad \forall \ t\in [0,T],\label{theh1}
\ee
 where $C_T$ is a constant depending on $\|\u_0\|$, $\|\phi_0\|_{H^1\cap L^\infty}$,  $\|\phi_b\|_{H^\frac32(\Gamma)}$, $\|\theta_0\|_{H^1\cap L^\infty}$, $T$, $\Omega$ and coefficients of the system. The proof is complete.
 \end{proof}

 If the initial temperature $\theta_0$ fulfills a slightly stronger assumption, indeed we can show that problem \eqref{1}--\eqref{ini} admits a dissipative energy law:

 \begin{lemma} \label{BEL2}
 Let $n=2$. We assume that $|\phi_0|\leq 1$ a.e. in $\Omega$, $|\phi_b|\leq 1$ on $\Gamma$ and  $\theta_0$ satisfies $\|\theta_0\|_{L^\infty}\leq \Theta_2$, where $\Theta_2$ is given in \eqref{2.7}.
 Introduce the energy functional
 \be \mathcal{E}(t)=
\|\u(t)\|^2+a\lambda_0\|\nabla\phi(t)\|^2+2a\lambda_0\int_{\Omega}W(\phi(t))dx
+\zeta\|\nabla\theta(t)\|^2+\omega\|\theta(t)\|^2,
 \label{toenergy}
 \ee
 where the
constants $\zeta, \omega>0$ may depend on $\Omega$, $\Theta_2$ and
coefficients of problem \eqref{1}--\eqref{ini}. Then the following energy inequality holds:
 \be
\frac{d }{dt}\mathcal{E}(t)+\frac{\underline{\mu}}{2}\|\nabla
\u\|^2+a\lambda_0\gamma\|\Delta\phi+W'(\phi)\|^2+\frac{\zeta\underline{\kappa}}{2}\|\Delta\theta\|^2
\leq 0, \quad \forall\, t>0.
 \label{bela}
  \ee
\end{lemma}

\begin{proof}
Let $\zeta>0$ be a constant to be determined
later (see \eqref{zeta} below). Multiplying equation \eqref{4} with
$-\zeta\Delta\theta$, integrating over $\Omega$ and adding the resultant with \eqref{up1}, we have
  \bea
&&\frac12\frac{d}{dt}\Big(\|\u\|^2+a\lambda_0\|\nabla\phi\|^2+2a\lambda_0\int_{\Omega}F(\phi)dx+\zeta\|\nabla\theta\|^2
 \Big)\non\\
 &&\ +2\int_\Omega \mu(\theta) |\mathcal{D}{\u}|^2dx +a\gamma\lambda_0\|\Delta\phi-W'(\phi)\|^2+\zeta\int_\Omega \kappa(\theta)|\Delta\theta|^2dx \non\\
&=&J_1+J_2+J_3+\zeta\int_{\Omega}(\u\cdot\nabla)\theta\Delta\theta \,dx -\zeta\int_\Omega \kappa'(\theta) |\nabla \theta|^2\Delta \theta dx,\non
 \eea
 where $J_1$, $J_2$ and $J_3$ are the same as in \eqref{up1} and we denote the last two terms on the right-hand side of the above equality by $J_4, J_5$, respectively.

 By Lemmas \ref{mphi}, \ref{mtheta}  and the assumption $\|\theta_0\|_{L^\infty}\leq  \Theta_2\in (0,\Theta_1]$ (see \eqref{2.7}), the terms in the second last line of \eqref{J11J33} can be re-estimated as follows
 \be
 \frac{4c_1^2c_2^2|b|^2\lambda_0^2}{\underline{\mu}}\|\theta\|_{L^\infty}^2\|\phi\|_{L^\infty}^2\|\Delta\phi-W'(\phi)\|^2
 \leq \frac{a\lambda_0\gamma}{2}\|\Delta\phi-W'(\phi)\|^2,\non
 \ee
 and by Agmon's inequality ($n=2$), it holds
 \bea
 &&\frac{4c_1^2c_2^2|b|^2\lambda_0^2}{\underline{\mu}}\|\theta\|_{L^\infty}^2\|\phi\|_{L^\infty}^2\Big(\|W'(\phi)\|^2+\|\phi\|^2+\|\phi_b\|_{H^\frac32(\Gamma)}^2\Big)\non\\
 &\leq& C\|\Delta \theta\|\|\theta\|\non\\
 &\leq& \frac{\zeta\underline{\kappa}}{4}\|\Delta \theta\|^2  +C\|\theta\|^2.\non
 \eea
 As a result, we have
 \be
 J_1+J_3\leq \frac{\underline{\mu}}{4}\|\nabla
\u\|^2+\frac{a\lambda_0\gamma}{2}\|\Delta\phi-W'(\phi)\|^2+\frac{\zeta\underline{\kappa}}{4}\|\Delta \theta\|^2+C\|\theta\|^2.\non
 \ee
 Again, the term $J_2$ can be estimated in the same way as \eqref{J2}.
 Next, for $J_4$ and $J_5$,  using the assumption $\|\theta_0\|_{L^\infty}\leq  \Theta_2$, Lemma \ref{mtheta} and  the Gagliardo--Nirenberg inequality \eqref{GNbb}, we deduce that
 \bea
J_4&=&-\zeta\int_{\Omega}\u
\cdot\nabla\left(\frac{|\nabla\theta|^2}{2}\right)dx
+\zeta\int_{\Omega} \u\cdot [\nabla \cdot(\nabla\theta\otimes \nabla\theta)]\,dx\non\\
&=& -\zeta\int_{\Omega} \nabla \u: (\nabla\theta\otimes \nabla\theta)\,dx\non\\
&\leq& \zeta\|\nabla \u\|\|\nabla\theta\|_{\mathbf{L}^4}^2 \non\\
&\leq& c_3\zeta\|\nabla \u\|\|\Delta\theta\|\|\theta\|_{L^\infty}\non\\
&\leq& \frac{\zeta\underline{\kappa}}{4}\|\Delta\theta\|^2+\frac{\zeta c_3^2}{\underline{\kappa}}\|\theta_0\|_{L^\infty}^2\|\nabla \u\|^2\non\\
&\leq&\frac{\zeta\underline{\kappa}}{4}\|\Delta\theta\|^2+\frac{\zeta c_3^2\Theta_2^2}{\underline{\kappa}}\|\nabla\u\|^2,\non
 \eea
and
 \bea
 J_5&\leq& \zeta\|\kappa'(\theta)\|_{L^\infty}\|\nabla \theta\|_{\mathbf{L}^4}^2\|\Delta \theta\|\non\\
 &\leq&
 \zeta c_3 \|\kappa'(\theta)\|_{L^\infty} \|\theta \|_{L^\infty}\|\Delta \theta \|^2 \non\\
 &\leq& \frac{\zeta\underline{\kappa}}{4}\|\Delta\theta\|^2.\non
 \eea
 Now taking
 \be
 \zeta=\frac{\underline{\mu}\underline{\kappa}}{4c_3^2\Theta_2^2}>0,\label{zeta}
 \ee
 we infer from the above estimates that \bea
&&\frac{d}{dt}\left(\|\u\|^2+a\lambda_0\|\nabla\phi\|^2+2a\lambda_0\int_{\Omega}F(\phi)dx+\zeta\|\nabla\theta\|^2
 \right)\non\\
 &&\ \ \ + \frac{\underline{\mu}}{2}\|\nabla \u\|^2
 + a\lambda_0\gamma\|\Delta\phi-F'(\phi)\|^2+\frac{\zeta\underline{\kappa}}{2}\|\Delta\theta\|^2\non\\
 & \leq& C_1 \|\theta\|^2.\label{part1bel}
\eea

  Next, multiplying the equation \eqref{4} by $2\omega\theta$, with
$\omega=\frac{C_1}{2\underline{\kappa}}>0$, integrating over $\Omega$, and using
Poincar\'e's inequality for $\theta$, we obtain
 \be
\omega\frac{d}{dt}\|\theta\|^2 = -2\omega \int_\Omega \kappa(\theta)|\nabla\theta|^2dx \leq
-2\omega \underline{\kappa}\|\nabla \theta\|^2=-C_1\|\nabla \theta\|^2.
 \label{part2bel}
 \ee
Adding \eqref{part1bel} with \eqref{part2bel}, we arrive at our conclusion \eqref{bela}. The proof is complete.
 \end{proof}

As a direct consequence of the above lemma, we can derive the following \emph{uniform-in-time} estimates for $(\u, \phi, \theta)$:

\begin{proposition}\label{low-estimate}
Let $n=2$. Under the assumptions of Lemma \ref{BEL2}, the smooth solution $(\u, \phi, \theta)$ to
problem \eqref{1}--\eqref{ini} satisfies the following energy inequality
 \be
\mathcal{E}(t)+ \int_0^t \left(\frac{\underline{\mu}}{2}\|\nabla
\u\|^2+a\lambda_0\gamma\|\Delta\phi-W'(\phi)\|^2+\frac{\zeta\underline{\kappa}}{2}\|\Delta\theta\|^2\right)
dt\leq \mathcal{E}(0),\quad \forall\, t\geq 0,\non
 \ee
which yields that
 \bea
 && \|\u(t)\|^2+\|\phi(t)\|_{H^1}^2+\|\theta(t)\|_{H^1}^2\leq C, \quad \forall\, t\geq 0,\label{low1} \\
 && \int_{0}^{+\infty}\big(\|\nabla \u(t)\|^2+\|\Delta\phi(t)-W'(\phi(t))\|^2+\|\Delta\theta(t)\|^2 \big)dt \leq C,\label{low2}
 \eea
 where $C > 0$ is a constant depending on $\|\u_0\|, \|\phi_0\|_{H^1\cap L^\infty}, \|\theta_0\|_{H^1}$, $\Theta_2$, $\Omega$ and coefficients of the
 system, but not on the time $t$.
\end{proposition}

 \subsubsection{Proof of Theorems \ref{weake1}}
 Using the global \textit{a priori} estimates obtained in Proposition \ref{BEL1} if $\|\theta_0\|_{L^\infty}\leq \Theta_1$ (or Proposition \ref{low-estimate} if $\|\theta_0\|_{L^\infty}\leq \Theta_2$) and the Sobolev embedding theorem ($n=2$), we can see that for any $\mathbf{v}\in \mathbf{V}$, it holds
 \bea
  |\langle \u_t, \mathbf{v}\rangle_{\mathbf{V}', \mathbf{V}}|&\leq&
  |\int_\Omega (\u\cdot\nabla)\u\cdot \mathbf{v} dx| + 2|\int_\Omega \mu(\theta) \mathcal{D} {\u} : \mathcal{D}{\mathbf{v}} dx| \non\\
   &&\quad + |\int_\Omega [\lambda(\theta)\nabla\phi\otimes\nabla\phi] : \nabla \mathbf{v} dx|+ |\int_\Omega  \theta \mathbf{g} \cdot \mathbf{v} dx|\non\\
   &\leq &  |\int_\Omega (\u\otimes \u):\nabla \mathbf{v} dx| + 2\|\mu(\theta)\|_{L^\infty}\|\mathcal{D} \u\|\|\mathcal{D} \mathbf{v}\|\non\\
      &&\quad + \|\lambda(\theta)\|_{L^\infty}\|\nabla \phi\|_{\mathbf{L}^4}^2\|\nabla\mathbf{v}\|+C\|\theta\|\|\mathbf{v}\|\non\\
   &\leq & C\|\u\|_{\mathbf{L}^4}^2\|\nabla \mathbf{v}\|+ C\|\nabla \u\|\|\nabla \mathbf{v}\|+C\|\nabla \phi\|_{\mathbf{L}^4}^2\|\nabla\mathbf{v}\|+C\|\theta\|\|\nabla \mathbf{v}\|   \non\\
   &\leq & C(\|\u\|\|\nabla \u\|+\|\nabla \u\|+\|\phi\|_{H^2}\|\phi\|_{H^1}+\|\theta\|)\|\nabla \mathbf{v}\|,\non
 \eea
 which implies that $\u_t\in L^2(0,T; \mathbf{V}')$. On the other hand, it is also easy to check that $\phi_t, \theta_t\in L^2(0,T;L^2(\Omega))$. Then the existence of global weak solutions to problem \eqref{1}--\eqref{ini} in $2D$ can be proved by working with a suitable semi-Galerkin approximate scheme (for instance, Type B in Appendix) and the standard compactness argument. The details are omitted here.

\begin{remark}
Different from the classical two-dimensional Navier--Stokes equations, uniqueness of global weak solutions to problem \eqref{1}--\eqref{ini} is an open issue in the $2D$ case. One of the technical difficulties comes from the estimate on the higher-order nonlinear term
$$-2\int_\Omega [\mu(\theta_1)-\mu(\theta_2)]\, \mathcal{D} \u_2 : \mathcal{D}  (\u_1-\u_2) dx,$$
 which is induced by the non-constant temperature-dependent viscosity $\mu(\theta)$. On the other hand, a weak-strong uniqueness result would be possible and we leave the details to interested readers.
\end{remark}


\subsection{Global strong solutions}

In what follows, we shall prove the existence and uniqueness of global strong solutions to problem \eqref{1}--\eqref{ini} in $2D$,
under the same assumptions on $\|\theta_0\|_{L^\infty}$ as for the case of global weak solutions. Since existence and uniqueness of local strong solutions in $2D$  have been proved in Section 3, below we only need to derive necessary global-in-time estimates for the strong solutions and then apply the extension theorem.

 The following lemma for the Stokes problem with variable viscosity coefficient will be helpful to prove higher-order spatial estimates for the velocity field $\u$. It can be obtained by the localization and freezing coefficients method (see e.g., \cite[Lemma 2.1]{SZ13} for a slightly different version).
\begin{lemma}\label{stokes}
For $n\geq 2$, consider the Stokes problem
\be
\begin{cases}
& -\nabla \cdot(2 \mu(x)\mathcal{D} \u) +\nabla P=\mathbf{f},\quad\ \ x\in \Omega, \\
& \nabla \cdot \u=0,\qquad\qquad\qquad\qquad\quad\, x\in \Omega,\\
& \int_\Omega P dx =0, \\
& \u=\mathbf{0}, \qquad \qquad\qquad\qquad \qquad\ \ x\in \Gamma,
\end{cases}\label{sto}
\ee
where the coefficient $\mu(x)\in H^2(\Omega)$ and satisfies
$
0<\underline{\mu}\leq \mu(x)\leq \overline{\mu} <+\infty.
$
If  $\mathbf{f}\in \mathbf{V}'$, then problem \eqref{sto} admits a unique weak solution $(\u, P)\in \mathbf{V}\times L^2(\Omega)$ such that
\be
 \|\u\|_{\mathbf{H}^1}+\|P\|\leq C\|\mathbf{f}\|_{\mathbf{H}^{-1}},
 \label{PL2}
 \ee
where $C=C(\Omega, n, \underline{\mu}, \overline{\mu})$ is a positive constant. Moreover, if $n=2$ and $\mathbf{f}\in \mathbf{L}^2(\Omega)$, then $(\u, P)\in (\mathbf{H}^2(\Omega)\cap \mathbf{V})\times H^1(\Omega)$ and there exists a constant $C=C(\Omega, [\mu]_{\delta}, \underline{\mu}, \overline{\mu})$ such that the following estimate holds:
\be
\|\Delta \u\|+\|\nabla P\|\leq C\Big(\|\mathbf{f}\|+\|\mu\|_{H^1}\|\mu\|_{H^2}\|\nabla \u\|+\|P\|\Big), \label{stoe}
\ee
 for some $\delta \in (0,1)$.
\end{lemma}
Besides, the following H\"older estimate for $\theta$ (see \cite[Lemma 3.2]{SZ13}) will be necessary in order to make use of Lemma \ref{stokes}:
\begin{lemma} \label{mthetaH}
Suppose that $n=2$ and $\u\in L^4(0, T; \mathbf{L}^4(\Omega))$. Consider the initial boundary value problem \eqref{eqthe}.  If $\theta\in L^\infty(0,T;  L^2(\Omega))\cap
L^2(0, T; H^1_0(\Omega))$ is a weak solution of \eqref{eqthe} and $\theta_0\in C^\alpha(\overline{\Omega})$ for some $\alpha\in (0, 1)$, then there exist two positive constants $C$ and $\delta\in (0,\alpha]$ depending only on $\Omega, \underline{\kappa}, \|\theta_0\|_{L^\infty}, [\theta_0]_\alpha$ and $\|\u\|_{L^4(0,T;\mathbf{L}^4)}$ such that
 $$[\theta]_{C^{\frac{\delta}{2}, \delta}}\leq C,$$
  where the H\"older semi-norm of $\theta$ is given by
 $[\theta]_{C^{\frac{\delta}{2}, \delta}}=\sup_{(t,x)\neq (\tau, y)\in (0,T)\times \Omega}\frac{|\theta(t,x)-\theta(\tau, y)|}{(|t-\tau|^\frac12+|x-y|)^\delta}$.
\end{lemma}

\subsubsection{\emph{A priori} estimates}
First, we derive some useful higher-order differential inequalities:

\begin{lemma}\label{high2d}
Suppose that $n=2$, $|\phi_0|\leq 1$ in $\Omega$, $|\phi_b|\leq 1$ on $\Gamma$ and $\theta_0$ satisfies either $\|\theta_0\|_{L^\infty}\leq \Theta_1$  or  $\|\theta_0\|_{L^\infty}\leq \Theta_2$. Let $(\u, \phi, \theta)$ be a smooth solution to problem \eqref{1}--\eqref{ini}.
 Then we have the following differential inequalities:
 \bea
 &&\frac{d}{dt}\Big(\int_\Omega \mu(\theta)|\mathcal{D}\u|^2dx+\frac{1}{2}\|\Delta\phi-W'(\phi)\|^2\Big)+(1-3\epsilon)\|
\u_t\|^2\non\\
&& \quad +(\gamma-4\epsilon)\|\nabla(\Delta\phi-W'(\phi))\|^2  \non\\
&\leq& \left(\epsilon+C_3\epsilon^{-2}\|\nabla \theta\|^4\|\nabla\phi\|^8 \right)\|\theta\|_{H^3}^2 + 4\epsilon \|\Delta \u\|^2\non\\
&&\quad +C(\|\u\|^2+\|\nabla \phi\|^2+1)(\|\nabla \u\|^4+\|\Delta\phi-W'(\phi)\|^4+\|\theta_t\|^4)\non\\
&&\quad +C(\|\nabla \theta\|^{8}+\|\nabla \phi\|^{16}+1),\label{uphiH1}
  \eea
  and
  \bea
&& \frac12\frac{d}{dt}\|\theta_t\|^2 + (\underline{\kappa}- 3\epsilon )\|\nabla \theta_t\|^2
\leq \epsilon^3 \|\u_t\|^2+ C (\|\theta_t\|^4+\|\Delta \theta\|^4),\label{thetL2}
\eea
  where $\epsilon\in(0,1)$ is an arbitrary small constant; the constant $C_3$ may depend on $\Omega$, $\lambda_0$, $b$ and $\|\theta_0\|_{L^\infty}$, but not on $\epsilon$ and $t$; the constant $C$ may depend on $\Omega$, $\|\phi_0\|_{L^\infty}$, $\|\phi_b\|_{H^\frac52(\Gamma)}$, $\|\theta_0\|_{L^\infty}$, $\epsilon$ and coefficients of the system, but not on $t$.
\end{lemma}
\begin{proof}

First, we try to re-estimate the terms $K_1, ..., K_4$ in \eqref{tA-part1} and the terms $K_5, K_6, K_7$ in \eqref{tA-part2}, respectively. Let $\epsilon\in(0,1)$ be an arbitrary small constant.
Using Lemma \ref{mtheta} and the Gagliardo--Nirenberg inequality $(n=2)$, we have
 \bea K_1 &\leq& \|\u_t\|\|\u\|_{\mathbf{L}^4}\|\nabla \u\|_{\mathbf{L}^4}
 \non\\
 &\leq& c\|\u_t\|\|\u\|^\frac12\|\nabla \u\|\|\Delta \u\|^\frac12
 \non\\
 &\leq&
\epsilon\|\u_t\|^2+ \epsilon\|\Delta \u\|^2+ C\|\u\|^2\|\nabla \u\|^4, \non
 \eea
 \bea
 K_2&\leq& c\|\mu'(\theta)\|_{L^\infty}\|\theta_t\|\|\mathcal{D} \u\|_{\mathbf{L}^4}^2\non\\
 & \leq & C \|\theta_t\|\|\nabla \u\|\|\Delta \u\|\non\\
 &\leq& \epsilon \|\Delta \u\|^2+ C \|\theta_t\|^2\|\nabla \u\|^2,\non
 \eea
 \bea
 K_4 &\leq& C\|\theta\|\|\u_t\|\leq  \epsilon\|\u_t\|^2+C\|\theta\|_{L^\infty}^2\non\\
 &\leq&  \epsilon\|\u_t\|^2+C.\non
 \eea
Concerning $K_3$, it follows from the H\"older inequality and Young's inequality that
\bea
 K_3 & \leq & \epsilon\|\u_t\|^2+ C\|\nabla \cdot[\lambda(\theta) \nabla \phi\otimes\nabla \phi]\|^2\non\\
    &\leq& \epsilon\|\u_t\|^2+C\|\lambda'(\theta)\|_{L^\infty}\|\nabla \theta\|_{\mathbf{L}^6}^2\|\nabla \phi\|_{\mathbf{L}^6}^4\non\\
    &&\ \ + C\|\lambda(\theta)\|_{L^\infty}^2(\|\Delta\phi\|^2+\|\nabla^2\phi\|^2)\|\nabla \phi\|_{\mathbf{L}^\infty}^2\non\\
    &:=& \epsilon\|\u_t\|^2+K_{3a}+K_{3b}.\label{Kes3bb}
 \eea
The term $K_{3a}$ can be estimated by using the Gagliardo--Nirenberg inequality, Lemma \ref{mphi} and Young's inequality such that
 \bea
 K_{3a}
&\leq& C\|\lambda'(\theta)\|_{L^\infty}\|\nabla \theta\|_{\mathbf{L}^6}^2\|\nabla \phi\|_{\mathbf{L}^6}^4\non\\
&\leq& C\lambda_0|b|\|\theta\|_{H^3}^{\frac23}\|\nabla \theta\|^{\frac43}\|\phi\|_{H^3}^{\frac43}\|\nabla\phi\|^{\frac83} \non\\
&\leq& C\|\theta\|_{H^3}^{\frac23}\|\nabla \theta\|^{\frac43}\|\nabla\phi\|^\frac83\Big(\|\nabla(\Delta\phi-W'(\phi))\|\non\\
&& \quad +\|\Delta\phi-W'(\phi)\|+\|W''(\phi)\nabla \phi\|+\|W'(\phi)\|+\|\phi_b\|_{H^\frac52(\Gamma)}+\|\phi\|\Big)^\frac43  \non\\
&\leq&  \epsilon\|\nabla(\Delta\phi-W'(\phi))\|^2+\big(\epsilon+C_3\epsilon^{-2} \|\nabla \theta\|^4\|\nabla\phi\|^8\big) \|\theta\|_{H^3}^2\non\\
&&\quad +C\|\Delta\phi-W'(\phi)\|^4+ C(\|\nabla \theta\|^{8}+\|\nabla \phi\|^{16}+1),\non
 \eea
 where the constant $C_3>0$ depends on $\Omega$, $\lambda_0$ and $|b|$, but not on $\epsilon$. \\
 For the term $K_{3b}$, using Lemma \ref{mtheta}, Agmon's inequality ($n=2$) and Young's inequality, we get
 \bea
 K_{3b} &\leq&
C(1+\|\theta\|_{L^\infty}^2)\|\phi\|_{H^2}^2\|\nabla \phi\|\|\nabla \phi\|_{\mathbf{H}^2}\non\\
&\leq& C(\|\Delta\phi-W'(\phi)\|^2+\|W'(\phi)\|^2+\|\phi_b\|_{H^\frac32(\Gamma)}^2+\|\phi\|^2)\|\nabla \phi\|\Big(\|\nabla(\Delta\phi-W'(\phi))\|\non\\
&& \ \ +\|\Delta\phi-W'(\phi)\|+\|W''(\phi)\nabla \phi\|+\|W'(\phi)\|+\|\phi_b\|_{H^\frac52(\Gamma)}+\|\phi\|\Big)\non\\
&\leq&  \epsilon\|\nabla(\Delta\phi-W'(\phi))\|^2+C\|\nabla \phi\|^2\|\Delta\phi-W'(\phi)\|^4\non\\
&& \ \ +C(\|\Delta\phi-W'(\phi)\|^4+\|\nabla \phi\|^4+1).\non
 \eea
Next, $K_5$ can be simply estimated in the following way
 \bea
 K_5&\leq&  \gamma \|W''(\phi)\|_{L^\infty}\|\Delta\phi-W'(\phi)\|^2\non\\
 &\leq& C\|\Delta\phi-W'(\phi)\|^2, \non
 \eea
while for $K_6$, using integration by parts and the incompressibility condition for $\u$, we have
\bea
K_6&=&2\int_{\Omega}\nabla(\Delta\phi-W'(\phi))\cdot (\nabla\u \nabla\phi)\,dx   \non\\
&\leq&C\|\nabla(\Delta\phi-W'(\phi))\|\|\nabla
\u\|_{\mathbf{L}^4}\|\nabla\phi\|_{\mathbf{L}^4}  \non\\
&\leq&\epsilon\|\nabla(\Delta\phi-W'(\phi))\|^2+C\|\nabla
\u\|\|\Delta \u\|\|\nabla\phi\|\non\\
&&\quad \times \big(\|\Delta\phi-W'(\phi)\|+\|W'(\phi)\|+\|\phi_b\|_{H^\frac32(\Gamma)}+\|\phi\|\big) \non\\
 &\leq& \epsilon \|\Delta \u\|^2+  \epsilon\|\nabla(\Delta\phi-W'(\phi))\|^2\non\\
 && \ \ +C\|\nabla \phi\|^2(\|\nabla \u\|^2+\|\nabla \u\|^4+\|\Delta\phi-W'(\phi)\|^4). \non
 \eea
 Similarly,
 \bea
 K_7&\leq& \|\Delta \u\| \| \nabla \phi\|_{\mathbf{L}^4} \| \Delta \phi-W'(\phi)\|_{L^4}\non\\
 &\leq& \epsilon \|\Delta \u\|^2+C\|\nabla (\Delta \phi-W'(\phi))\|\|\Delta \phi-W'(\phi)\|\|\nabla \phi\|\non\\
 &&\quad \times \big(\|\Delta\phi-W'(\phi)\|+\|W'(\phi)\|+\|\phi_b\|_{H^\frac32(\Gamma)}+\|\phi\|\big) \non\\
 &\leq& \epsilon \|\Delta \u\|^2+  \epsilon\|\nabla(\Delta\phi-W'(\phi))\|^2 \non\\
 && \quad +C\|\nabla \phi\|^2(\|\Delta\phi-W'(\phi)\|^4+\|\Delta\phi-W'(\phi)\|^2).\non
 \eea
Collecting the above estimates together, we infer from \eqref{tA-part1}, \eqref{tA-part2} and Young's inequality that the differential inequality \eqref{uphiH1} holds.

Finally, we re-estimate the terms $K_{11}$, $K_{12}$, $K_{13}$ in the inequality \eqref{temperature-1}. First we recall that $K_{13}=0$. Then using the Sobolev embedding theorem, \eqref{GNbb} and Lemma \ref{mtheta}, we get
 \bea
K_{11}  &\leq& \|\kappa'(\theta)\|_{L^\infty}\|\nabla \theta_t\|\|\theta_t\|_{L^4}\|\nabla \theta\|_{\mathbf{L}^4}\non\\
&\leq& \epsilon \|\nabla \theta_t\|^2+c\|\kappa'(\theta)\|_{L^\infty}^2\|\nabla \theta_t\|\|\theta_t\|\|\Delta \theta\|\|\theta\|_{L^\infty}
\non\\
&\leq& 2\epsilon \|\nabla \theta_t\|^2+ C \|\theta_t\|^2\|\Delta \theta\|^2,\non
 \eea
 and
 \bea
 K_{12}
 &\leq& \|\u_t\|\|\nabla \theta\|_{\mathbf{L}^4}\|\theta_t\|_{L^4}\non\\
 &\leq& \epsilon^3 \|\u_t\|^2+ C\epsilon^{-3} \|\Delta\theta\|\|\theta\|_{L^\infty}\|\nabla \theta_t\|\|\theta_t\|\non\\
 &\leq& \epsilon^3 \|\u_t\|^2+ \epsilon \|\nabla \theta_t\|^2+ C \|\theta_t\|^2\|\Delta \theta\|^2.\non
 \eea
As a consequence, we can conclude the differential inequality \eqref{thetL2} from the above estimates, \eqref{temperature-1} and Young's inequality. The proof is complete.
\end{proof}

 Now we proceed to derive global higher-order estimates for $(\u, \phi, \theta)$.

\begin{proposition}\label{high2D}
Let $n=2$. Suppose that $(\u, \phi, \theta)$ is a smooth solution to problem \eqref{1}--\eqref{ini}.

(1) If $|\phi_0|\leq 1$  in $\Omega$, $|\phi_b|\leq 1$ on $\Gamma$ and  $\theta_0$ satisfies $\|\theta_0\|_{L^\infty}\leq \Theta_1$, then for arbitrary but fixed $T\in (0,+\infty)$, we have the following estimates
\bea
&& \|\u\|_{\mathbf{V}}+\|\phi\|_{H^2}+\|\theta\|_{H^2}+\|\phi_t\|+\|\theta_t\|\leq C_T,\quad \forall\, t\in [0,T],\label{ha}\\
&& \int_0^T \Big(\|\u\|^2_{\mathbf{H}^2}+\|\phi\|_{H^3}^2+\|\theta\|_{H^3}^2+ \|\u_t\|^2+ \|\phi_t\|_{H^1}^2+\|\theta_t\|_{H^1}^2\Big)dt \leq C_T,\label{ha1}
\eea
where $C_T$ is a constant depending on $\Omega$, $\|\u_0\|_{\mathbf{V}}$, $\|\phi_0\|_{H^2}$, $\|\phi_b\|_{H^\frac52(\Gamma)}$, $\|\theta_0\|_{H^2}$, the coefficients of the system as well as $T$.

(2) If $|\phi_0|\leq 1$  in $\Omega$, $|\phi_b|\leq 1$ on $\Gamma$ and $\theta_0$ satisfies $\|\theta_0\|_{L^\infty}\leq \Theta_2$, then we have
\bea
&& \|\u\|_{\mathbf{V}}+\|\phi\|_{H^2}+\|\theta\|_{H^2}+\|\phi_t\|+\|\theta_t\|\leq C,\quad \forall\, t\geq 0,\label{hb}\\
&& \int_0^T \Big(\|\u\|^2_{\mathbf{H}^2}+\|\phi\|_{H^3}^2+\|\theta\|_{H^3}^2+ \|\u_t\|^2+ \|\phi_t\|_{H^1}^2+\|\theta_t\|_{H^1}^2\Big)dt \leq C_T,\label{hb1}
\eea
where $T\in (0,+\infty)$ and the constant $C$ in \eqref{hb} is independent of $t$. Moreover, the following decay property holds
\be
\lim_{t\to +\infty} (\|\u(t)\|_{\mathbf{V}}+\|\Delta \phi(t)-W'(\phi(t))\|+\|\theta(t)\|_{H^2})=0.\label{hb2}
\ee
\end{proposition}

\begin{proof}
In order to make use of the differential inequalities obtained in Lemma \ref{high2d}, the key point is to treat the quantities $\|\Delta \theta\|$, $\|\theta\|_{H^3}$ and $\|\Delta \u\|$ on the right-hand side of \eqref{uphiH1} and  \eqref{thetL2}, which involve higher-order spatial derivatives. In the subsequent proof, let $\epsilon\in(0,1)$ be the same small constant as in Lemma \ref{high2d}. We consider two cases in which the initial temperature $\theta_0$ satisfies different constraints on its $L^\infty$-norm.

\medskip

\textbf{Case 1. Suppose $\|\theta_0\|_{L^\infty}\leq \Theta_1$}.

 In this case, we have global lower-order estimates for $(\u, \phi, \theta)$ as in \eqref{es1} (see Proposition \ref{BEL1}). Below we use $C_i$ for constants that may depend on $\|\u_0\|$, $\|\phi_0\|_{H^1\cap L^\infty}$,  $\|\phi_b\|_{H^\frac52(\Gamma)}$, $\|\theta_0\|_{H^2}$, $\Omega$, coefficients of the system and possibly the time $T$, but they are independent of the parameter $\epsilon$.

Applying the elliptic regularity theorem to equation \eqref{newtheta1} for the transformed temperature $\vartheta$, using the Sobolev embedding theorem, Agmon's inequality ($n=2$), the Poincar\'e inequality and Young's inequality, we have
\bea
 \|\vartheta\|_{H^3}
&\leq& c(\|\chi'(\vartheta)(\vartheta_t+\u\cdot\nabla \vartheta)\|_{H^1}+\|\vartheta\|)\non\\
&\leq& c\|\chi'(\vartheta)\|_{L^\infty}\|\vartheta_t+\u\cdot\nabla \vartheta\|_{H^1}+c\|\nabla \chi'(\vartheta)\|_{\mathbf{L}^\infty}\|\vartheta_t+\u\cdot\nabla \vartheta\|+c\|\vartheta\|\non\\
&\leq& c \underline{\kappa}^{-1} (\|\vartheta_t\|_{H^1}+\|\u\|_{\mathbf{L}^4}\|\nabla \vartheta\|_{\mathbf{L}^4}+\|\nabla \u\|_{\mathbf{L}^4}\|\nabla \vartheta\|_{\mathbf{L}^4}+\|\u\|_{\mathbf{L}^4}\|\nabla^2\vartheta\|_{\mathbf{L}^4})\non\\
&& \quad +c \Big\|\frac{\kappa'(\theta)}{\kappa(\theta)^3}\Big\|_{L^\infty}\|\nabla \theta\|_{\mathbf{L}^\infty}(\|\vartheta_t\|+\|\u\|_{\mathbf{L}^4}\|\nabla \theta\|_{\mathbf{L}^4})+c\|\vartheta\|\non\\
&\leq& c \underline{\kappa}^{-1}\|\nabla \vartheta_t\|+ C\|\nabla \u\|^\frac12\|\u\|^\frac12\|\Delta \vartheta\|^\frac12\|\nabla \vartheta\|^\frac12
\non\\
&& \quad +C\|\Delta \u\|^\frac12\|\nabla \u\|^\frac12\|\Delta \vartheta\|^\frac12\|\nabla \theta\|^\frac12 +C\|\nabla \u\|^\frac12\|\u\|^\frac12\|\vartheta\|_{H^3}^\frac12\|\Delta\vartheta\|^\frac12\non\\
&& \quad +C\|\vartheta\|_{H^3}^\frac12\|\nabla \vartheta\|^\frac12(\|\vartheta_t\|+\|\nabla \u\|^\frac12\|\u\|^\frac12\|\Delta \vartheta\|^\frac12\|\nabla \vartheta\|^\frac12)+c\|\vartheta\|\non\\
&\leq& \frac12 \|\vartheta\|_{H^3} + c \underline{\kappa}^{-1}\|\nabla \vartheta_t\|+ \frac12\epsilon^\frac32\|\Delta \u\|+C\|\nabla \u\|\|\u\|+C\|\Delta \vartheta\|\|\nabla \vartheta\| \non\\
&&\quad +C\|\Delta \vartheta\|\|\nabla \vartheta\|\|\nabla \u\|+C\|\Delta \vartheta\|\|\nabla \u\|\|\u\|+C\|\vartheta_t\|^2\|\nabla \vartheta\|\non\\
&&\quad +C\|\Delta \vartheta\|\|\nabla \vartheta\|^2\|\nabla \u\|\|\u\|+C\|\vartheta\|,\label{varthH3}
\eea
where the constant $C$ may depend on $\Omega$, $\|\theta_0\|_{L^\infty}$, $\epsilon$ and coefficients of the system, but it is independent of $t$.

By Proposition \ref{BEL1} and the estimates \eqref{es1}, \eqref{varthH1}, we infer from  \eqref{h2the} that
\be
\|\Delta \vartheta\|\leq C\|\theta_t\|+C_T\|\nabla \u\|,\non
\ee
which together with \eqref{varthH3} and Young's inequality yields that
\bea
\|\vartheta\|_{H^3}&\leq& C_4\|\nabla \theta_t\| + \epsilon^\frac32\|\Delta \u\|+ C_T(\|\nabla \u\|^2+\|\theta_t\|^2) +C_T,\label{vtta}
\eea
where the constant $C_4$ may depend on $\Omega$, $\|\theta_0\|_{L^\infty}$ and coefficients of the system, but independent of $\epsilon$ and $t$, while the constant $C_T$ may depend on $\epsilon$, $\Omega$, $\|\u_0\|$, $\|\phi_0\|_{H^1\cap L^\infty}$, $\|\theta_0\|_{H^1\cap L^\infty}$, coefficients of the system and $T$. As a consequence, we deduce from \eqref{vtta}, the Gagliardo--Nirenberg inequality, the relations \eqref{ttt3} and \eqref{ttt5} that
\bea
\|\Delta \theta\|&\leq& C\|\theta_t\|+C_T\|\nabla \u\|,\label{thetaH2}\\
\|\theta\|_{H^3} &\leq&  C(\|\vartheta\|_{H^3}+\|\vartheta\|_{H^3}^\frac12\|\Delta \vartheta\|\|\nabla \vartheta\|^\frac12+\|\vartheta\|_{H^3}\|\nabla \vartheta\|^2+\|\Delta \vartheta\|)\non\\
 &\leq& C_T\|\vartheta\|_{H^3} +C_T( \|\Delta \vartheta\|^2+\|\Delta \vartheta\|) \non\\
&\leq& \frac12 C_{5}^\frac12\|\nabla \theta_t\| + \frac12 C_{6}^\frac12\epsilon^\frac32\|\Delta \u\|+ C_T(\|\nabla \u\|^2+\|\theta_t\|^2)+C_T,\label{thetaH3}
\eea
where the constants $C_{5}, C_{6}$ may depend on $\|\u_0\|$, $\|\phi_0\|_{H^1\cap L^\infty}$,  $\|\phi_b\|_{H^\frac52(\Gamma)}$, $\|\theta_0\|_{H^1\cap L^\infty}$, $T$, but independent of $\epsilon$.

Now it remains to estimate $\|\Delta \u\|$. We set
$$\mathbf{f}= -\u_t-\u\cdot\nabla \u-\nabla \cdot \Big[\lambda(\theta)\nabla \phi\otimes \nabla \phi+\lambda(\theta)\Big(\frac12|\nabla \phi|^2+W(\phi)\Big)\mathbb{I}\Big]+ \mathbf{g}\theta$$
and write the equation \eqref{1} as
\be -\nabla \cdot(2\mu(\theta)\mathcal{D} \u)+\nabla P=\mathbf{f}.\label{pp}\ee
It follows from the definition of $\mathbf{f}$ that
\bea
\|\mathbf{f}\|&\leq& \|\u_t\|+\|\u\|_{\mathbf{L}^4}\|\nabla \u\|_{\mathbf{L}^4}+\|\lambda(\theta)\|_{L^\infty}\|\phi\|_{H^2}\|\nabla \phi\|\non\\
&& \quad +\|\lambda'(\theta)\|_{L^\infty}\|\nabla \theta\|_{\textbf{L}^3}\|\nabla \phi\|_{\textbf{L}^3}^2+\|\lambda'(\theta)\|_{L^\infty}\|\nabla \theta\|\|W(\phi)\|_{L^\infty}\non\\
&&\quad +\|\lambda(\theta)\|_{L^\infty}\|W'(\phi)\|_{L^\infty}\|\nabla \phi\|+C\|\theta\|.\non
\eea
 Then by Proposition \ref{BEL1}, \eqref{thetaH2} and the elliptic estimate $$\|\phi\|_{H^2}\leq C(\|\Delta\phi-W'(\phi)\|+\|W'(\phi)\|+ \|\phi_b\|_{H^\frac32(\Gamma)}+\|\phi\|),$$
 we obtain
 \bea
 \|\mathbf{f}\|&\leq& \|\u_t\|+ C\|\u\|^\frac12\|\Delta \u\|^\frac12\|\nabla \u\|+C\|\phi\|_{H^2}\|\nabla \phi\|\non\\
 &&\quad +C\|\Delta \theta\|^\frac13\|\nabla \theta\|^\frac23\|\phi\|_{H^2}^\frac23\|\nabla \phi\|^\frac43+C_T\non\\
 &\leq& \epsilon_1\|\Delta \u\|+ \|\u_t\| +  C_T\|\nabla \u\|^2 +C_T(\|\Delta\phi-W'(\phi)\|+\|\Delta\theta\|+1)\non\\
 &\leq& \epsilon_1\|\Delta \u\|+ \|\u_t\|  + C_T(\|\nabla \u\|^2+\|\theta_t\|+\|\Delta\phi-W'(\phi)\|+1),\non
 \eea
 where $\epsilon_1>0$ is an arbitrary small constant (independent of $\epsilon$).

 It follows from the estimate \eqref{es1} and the Sobolev embedding theorem ($n=2$) that $\|\u\|_{L^4(0,T;\mathbf{L}^4(\Omega))}\leq C_T$. Since $H^2(\Omega) \hookrightarrow C^\alpha(\overline{\Omega})$ with $\alpha\in (0,1)$ when $n=2$, we infer from Lemma \ref{mthetaH} that $[\theta]_{C^{\frac{\delta}{2}, \delta}}\leq C_T$ for some $\delta\in (0, \alpha)$. Due to this H\"older estimate for $\theta$, we are able to apply Lemma \ref{stokes} to the Stokes equation \eqref{pp} and deduce from the estimates  \eqref{PL2}--\eqref{stoe} that
\bea
&& \|\Delta \u\|+\|\nabla P\|\non\\
&\leq& C\|\mathbf{f}\|+C\|\mu(\theta)\|_{H^1}\|\mu(\theta)\|_{H^2}\|\nabla \u\|+C\|P\|\non\\
&\leq& C\epsilon_1\|\Delta \u\|+ C\|\u_t\|+C_T\|\nabla \u\|^2 +C_T(\|\Delta \theta\|+\|\nabla \theta\|_{\mathbf{L}^4}^2) \|\nabla \u\|\non\\
&& \quad + C_T(\|\nabla \u\|+\|\theta_t\|+\|\Delta\phi-W'(\phi)\|)+C_T\non\\
&\leq& C\epsilon_1\|\Delta \u\|+ C\|\u_t\| +C_T(\|\nabla \u\|^2+\|\theta_t\|^2+\|\Delta\phi-W'(\phi)\|^2)+C_T,\non
\eea
where the constant $C$ only depends on $\Omega$ and $\|\theta_0\|_{H^2}$, but not on $\epsilon_1$. Then taking $\epsilon_1$ in the above inequality sufficiently small, we have
\bea
\|\Delta \u\|&\leq&  \frac12 C_{7}^\frac12 \|\u_t\| +C_T(\|\nabla \u\|^2+\|\theta_t\|^2+\|\Delta\phi-W'(\phi)\|^2+1),\label{deltau}
\eea
where the constant $C_{7}$ only depends on $\Omega$ and $\|\theta_0\|_{H^2}$.

Let $\eta>0$ be a constant (not necessary small) that will be determined later (see \eqref{eta} below). We introduce the quantity
\be
\mathcal{Y}(t):=2\int_\Omega \mu(\theta)|\mathcal{D} \u|^2dx+\|\Delta\phi-W'(\phi)\|^2+\eta\|\theta_t\|^2. \label{Y}
\ee
Multiplying the differential inequality \eqref{thetL2} by $\eta$ and adding the resultant with  \eqref{uphiH1} (see Lemma \ref{high2d}), using Proposition \ref{BEL1}, the estimates \eqref{thetaH2}, \eqref{thetaH3}, \eqref{deltau} and the fact $\epsilon\in (0,1)$, we obtain that
 \bea
 &&\frac{1}{2}\frac{d}{dt}\mathcal{Y}(t)+(1-3\epsilon)\|
\u_t\|^2+(\gamma-4\epsilon)\|\nabla(\Delta\phi-W'(\phi))\|^2\non\\
&& \quad +\eta(\underline{\kappa}- 3\epsilon )\|\nabla \theta_t\|^2  \non\\
&\leq& \eta\epsilon^3 \|\u_t\|^2+ \left(\epsilon + C_8\epsilon^{-2}\right) \left(C_{5}\|\nabla \theta_t\|^2+C_{6}\epsilon^3\|\Delta \u\|^2\right) +4\epsilon \|\Delta \u\|^2
\non\\
&& \quad +C_T\left(\|\nabla \u\|^4+\|\Delta\phi-W'(\phi)\|^4+\|\theta_t\|^4\right)+C_T\non\\
&\leq& \Big[\eta\epsilon^2 + \frac12 C_{6}C_{7}(\epsilon^3+C_8) + 2C_{7}\Big] \epsilon \|\u_t\|^2+(\epsilon + C_8\epsilon^{-2})C_{5}\|\nabla \theta_t\|^2
\non\\
&& \quad +C_T\left(\|\nabla \u\|^4+\|\Delta\phi-W'(\phi)\|^4+\|\theta_t\|^4\right)+C_T\non\\
&\leq& \left(\eta \epsilon^2 +C_9\right) \epsilon \|\u_t\|^2+C_{10}\epsilon^{-2}\|\nabla \theta_t\|^2
\non\\
&&\quad +C_T\left(\|\nabla \u\|^4+\|\Delta\phi-W'(\phi)\|^4+\|\theta_t\|^4\right)+C_T, \label{diff1a}
  \eea
where the constants $C_9$ and $C_{10}$ do not depend on $\epsilon$ and $\eta$.

For arbitrary but fixed $T\in(0,+\infty)$, on the time interval $[0,T]$, we shall properly choose the constants $\epsilon\in (0,1)$ and $\eta>0$ in the differential inequality \eqref{diff1a} such that the following relations are satisfied
\be
\begin{cases}
1-3\epsilon- (\eta\epsilon^2 +C_9) \epsilon\geq \displaystyle{\frac12},\\
\gamma-4\epsilon\geq \displaystyle{\frac{\gamma}{2}},\\
\eta(\underline{\kappa} -3\epsilon) - C_{10}\epsilon^{-2} = \displaystyle{\frac{\eta\underline{\kappa}}{2}}.
\end{cases}\non
\ee
One can verify that a possible choice for $\epsilon$, $\eta$ is as follows
\be
\begin{cases}
0<\epsilon\leq \min\left\{\displaystyle{\frac{1}{8}, \  \frac{\gamma}{8}, \  \frac{\underline{\kappa}}{12}, \
\frac{\underline{\kappa}}{8(4C_{10}+C_9\underline{\kappa})}}\right\},\\
 \eta = \displaystyle{\frac{2C_{10}}{(\underline{\kappa}-6\epsilon)\epsilon^2}}.
 \end{cases}
\label{eta}
\ee
After fixing the constants $\epsilon$ and $\eta$, using the fact $\|\nabla \u\|^2\leq 2\underline{\mu}^{-1} \int_\Omega \mu(\theta)|\mathcal{D}\u|^2 dx$, we infer from \eqref{diff1a} that
\bea
 &&\frac{d}{dt}\mathcal{Y}(t)+\|
\u_t\|^2+\gamma\|\nabla(\Delta\phi-W'(\phi))\|^2+\eta\underline{\kappa}\|\nabla \theta_t\|^2  \non\\
&\leq& C_T\mathcal{Y}(t)^2+C_T. \label{diff1}
  \eea
  On the other hand, we infer from Lemma \ref{mtheta} and Proposition \ref{BEL1} that
  $$2 \int_\Omega \mu(\theta)|\mathcal{D} \u|^2dx+\|\Delta\phi-W'(\phi)\|^2+\|\theta_t\|^2\in L^1(0,T).$$
 Hence, it follows from \eqref{diff1} and the Gronwall type inequality in \cite[Lemma 6.2.1]{Z04} that for arbitrary but fixed $T\in(0,+\infty)$,
 \bea
 &&\sup_{t\in[0,T]}\Big(\int_\Omega \mu(\theta)|\mathcal{D} \u|^2dx+\|\Delta\phi-W'(\phi)\|^2+\|\theta_t\|^2\Big)\leq C_T,\non\\
 && \int_0^T \left(\|\u_t\|^2+\|\nabla(\Delta\phi-W'(\phi))\|^2+\|\nabla \theta_t\|^2 \right)dt\leq C_T,\non
 \eea
 which together with the estimates \eqref{thetaH2}, \eqref{thetaH3} and \eqref{deltau} yield our conclusions \eqref{ha} and \eqref{ha1}.

 \medskip

\textbf{Case 2. Suppose $\|\theta_0\|_{L^\infty}\leq \Theta_2$}.

In this case, we can obtain the uniform-in-time estimates \eqref{low1} and \eqref{low2} (see Proposition \ref{low-estimate}) instead of the time-dependent estimates \eqref{es1}.
Then by corresponding modifications in the argument for Case 1, we can see that for $t>0$, the following differential inequality holds (comparing with \eqref{diff1})
\be
 \frac{d}{dt}\mathcal{Y}(t)+\|
\u_t\|^2+\gamma\|\nabla(\Delta\phi-W'(\phi))\|^2+\eta \underline{\kappa}\|\nabla \theta_t\|^2  \leq C\mathcal{Y}(t)^2+C, \label{diff2}
  \ee
 where $C$ is a constant that depend on $\Omega$, $\|\u_0\|$, $\|\phi_0\|_{H^1\cap L^\infty}$, $\|\phi_b\|_{H^\frac52(\Gamma)}$, $\|\theta_0\|_{H^2}$ and coefficients of the system, but not on time $t$.

 Using the uniform-in-time estimate \eqref{low1}, we have
 \bea
  \|\theta_t\|&\leq& \|\u\cdot\nabla \theta\|+ \|\nabla \cdot(\kappa(\theta)\nabla \theta)\| \non\\
  &\leq& \|\u\|_{\mathbf{L}^4}\|\nabla \theta\|_{\mathbf{L}^4}+ \|\kappa(\theta)\|_{L^\infty}\|\Delta \theta\|+ \|\kappa'(\theta)\|_{L^\infty}\|\nabla \theta\|_{\mathbf{L}^4}^2\non\\
  &\leq& C\|\nabla \u\|^\frac12\|\Delta \theta\|^\frac12+C\|\Delta \theta\|(1+\|\nabla \theta\|)\non\\
  &\leq& C(\|\nabla \u\|+\|\Delta \theta\|)\label{tttt}
  \eea
  where $C$ is independent of $t$.   Then we deduce the $L^1$-integrability of $\mathcal{Y}(t)$ on $\mathbb{R}^+$ from the above estimate and \eqref{low2}:
  \be
  \int_0^{+\infty} \mathcal{Y}(t) dt<+\infty.\label{Y1}
  \ee
Then it follows from \eqref{diff2}--\eqref{Y1} and the Gronwall type lemma in \cite[Lemma 6.2.1]{Z04} that
\be
\sup_{t\geq 0} \mathcal{Y}(t)\leq C\quad \text{ and } \ \lim_{t\to+\infty} \mathcal{Y}(t)=0,\label{uni}
\ee
which easy imply the uniform estimates \eqref{hb} and \eqref{hb1}. By \eqref{ttt3}, we see that
\bea
\|\Delta \theta\|&\leq& C(\|\nabla^2\vartheta\|+\|\nabla \vartheta\|_{\mathbf{L}^4}^2)\non\\
&\leq& C(\|\vartheta_t\|+\|\u\|\|\nabla \u\|\|\nabla \vartheta\|)(1+\|\nabla \vartheta\|)\non\\
&\leq& C(\|\theta_t\|+\|\nabla \u\|),\non
\eea
where the constant $C$ is independent of $t$. Then \eqref{tttt} and \eqref{uni} further indicate that $\lim_{t\to+\infty}\|\Delta \theta\|=0$. As a consequence, the decay property \eqref{hb2} holds.

The proof of Proposition \ref{high2D} is complete.
\end{proof}

\subsubsection{Proof of Theorem \ref{str2D}}
Based on the local well-posedness result, the existence of global strong solutions in $2D$ easily follows from the global-in-time \textit{a priori} estimates obtained in Proposition \ref{high2D}. Besides, uniqueness of global strong solutions can be proved in the same way as for Theorem \ref{str3D} (see \eqref{conti}). Hence, the details are omitted here.\medskip

\begin{remark}
The continuous dependence estimate \eqref{conti} implies that in the $2D$ case, for any $(\u_0, \phi_0, \theta_0)\in \mathbf{V}\times H^2\times (H^2\cap H^1_0)$ we are able to define a \emph{closed
semigroup} $\Sigma(t)$ for  $t\geq 0$ in the sense of \cite{PZ07}
by setting $(\u(t), \phi(t), (\theta(t))=\Sigma(t)(\u_0, \phi_0, \theta_0)$,  where $(\u,\phi,\theta)$ is the global
strong solution to problem \eqref{1}--\eqref{ini}. This observation will enable us to further investigate the associate infinite dimensional dynamical system of problem \eqref{1}--\eqref{ini}, for instance, the existence of a global attractor.
\end{remark}
\begin{remark} \label{theorem on long time behavior}
Since problem \eqref{1}--\eqref{ini} enjoys an dissipative energy law \eqref{bela} under the assumption $\|\theta_0\|_{L^\infty}\leq \Theta_2$, in addition to the decay property \eqref{hb2}, we can further prove the convergence of global strong solution to a \emph{single steady state} $(\mathbf{0}, \phi_\infty, 0)$ as time goes to
infinity. Indeed, there exists a function $\phi_\infty$ satisfying the nonlinear elliptic boundary value problem:
  \be
  \begin{cases}
  & - \Delta \phi_\infty + W'(\phi_\infty)=0, \ \ \ x \in \Omega,\\
  & \phi_\infty(x)=\phi_b(x),\qquad\qquad\,  x\in \Gamma,
  \end{cases}
   \non
   \ee
such that
  \be
 \|\u(t)\|_{\mathbf{V}}+\|\phi(t)-\phi_\infty\|_{H^2}+\|\theta(t)\|_{H^2}
  \leq C(1+t)^{-\frac{\rho}{(1-2\rho)}}, \quad \forall\, t \geq
 0,\non
 \ee
 where $\rho \in (0,\frac12)$ is a constant depending on $\phi_\infty$.
 Based on the dissipative energy law \eqref{bela} and the uniform-in-time higher-order estimate \eqref{hb}, the proof can be carried out by applying the so-called \L ojasiewicz--Simon approach, following the argument in \cite{WX13}. We leave the details to interested readers.
\end{remark}

\section{Conclusion}

In this paper, we study the well-posedness of a non-isothermal diffuse-interface model  proposed in \cite{LSFY05,SLX09,GLW} that describes the Marangoni effects in the mixture of two incompressible Newtonian fluids due to the thermo-induced surface tension heterogeneity on the interface. This is an interesting physical phenomenon, and some numerical
schemes as well as numerical simulations have been investigated in the recent literature \cite{SLX09,GLW}. The first theoretical
results concerning well-posedness and long-time behavior of solutions are obtained in \cite{WX13}, in which a simplified version of the system \eqref{1}--\eqref{4} with only constant fluid  viscosity and thermal diffusivity was considered. Here, we study the more general case such that the surface tension, fluid viscosity and thermal diffusivity are allowed to be temperature dependent. More precisely, for the initial-boundary value problem \eqref{1}--\eqref{ini}, we have proved that (1) for general regular initial data, strong solutions are locally well-posed in both $2D$ and $3D$ (see Theorem \ref{str3D}); (2) under the assumption that the $L^\infty$-norm of initial temperature is bounded only with respect to the coefficients of problem \eqref{1}--\eqref{ini}, global weak solutions exist in $2D$ (see Theorem \ref{weake1}); (3) under the same bound on the initial temperature variation, problem \eqref{1}--\eqref{ini} also admits a unique global strong solution in $2D$ (see Theorem \ref{str2D}). We believe that establishing well-posedness property of the diffuse-interface model could be viewed as a useful step towards its validation. On the other hand, our results on the long-term dynamics (see Theorem \ref{str2D} and Remark \ref{theorem on long time behavior}) will be helpful for people to understand the complicated nonlinear phenomena and construct suitable numerical schemes. Finally, we expect to extend the results to the case where the phase-filed function $\phi$ satisfies a forth-order Cahn--Hilliard type equation instead of the second-order Allen--Cahn equation. In this case, the mathematical analysis is more challenging due to the loss of maximum principle, which launches an interesting problem for the future study.


\section{Appendix}
\setcounter{equation}{0}

In what follows, we sketch the semi-Galerkin approximate schemes that are used in our previous proofs. They are inspired by \cite{LL95} on the simplified Ericksen--Leslie system for nematic liquid crystal flows. Furthermore, different schemes will be used for the local strong solutions in both $2D$ and $3D$, and for the global weak solutions in $2D$.

Let the family $\{\mathbf{v}_i\}_{i=1}^{\infty}$ be a basis of the Hilbert space $\mathbf{V}$, which is given by eigenfunctions of the Stokes problem
$$ (\nabla \mathbf{v}_i, \nabla \mathbf{w})=\lambda_i(\mathbf{v}_i, \mathbf{w}),\quad \forall\, \mathbf{w}\in \mathbf{V}, \ \text{with}\ \|\mathbf{v}_i\|=1,$$  %
where $\lambda_i$ is the eigenvalue corresponding to $\mathbf{v}_i$. It is well-known that $0<\lambda_1<\lambda_2<...$ is an unbounded monotonically increasing sequence, $\{\mathbf{v}_i\}_{i=1}^{\infty}$ forms a complete orthonormal basis in $\mathbf{H}$ and it is also orthogonal in $\mathbf{V}$ (see \cite{Te01}). In a similar way, let the family $\{w_i\}_{i=1}^{\infty}$ be the Hilbert basis of $H^1_0(\Omega)$, which is given by the eigenfunctions of the Laplacian
$$ -\Delta w_i =\eta_i w_i,\quad \ w_i|_\Gamma=0, \quad  \text{with}\ \ \|w_i\|=1.$$  %
Then the eigenvalues $0<\eta_1<\eta_2<...$ also form an unbounded monotonically increasing sequence, $\{w_i\}_{i=1}^{\infty}$ forms a complete orthonormal basis in $L^2(\Omega)$ and it is also orthogonal in $H^1_0(\Omega)$.  By the elliptic regularity theory, we have $\mathbf{v}_i\in \mathbf{C}^\infty$ and $w_i\in C^\infty$ for all $i\in \mathbb{N}$.

For every $m\in\mathbb{N}$, we denote by $\mathbf{V}_m=\mathrm{span}\{\mathbf{v}_1,...,\mathbf{v}_m\}$ and $W_m=\mathrm{span}\{w_1,...,w_m\}$ the finite dimensional subspaces of $\mathbf{V}$ and $H^1_0(\Omega)$ spanned by their first $m$ basis functions, respectively. Moreover, we use $\Pi_m$ for the orthogonal projection from $\mathbf{H}$ onto $\mathbf{V}_m$ and $\widetilde{\Pi}_m$ for the orthogonal projection from $L^2(\Omega)$ onto $W_m$.
\medskip

\subsection{Semi-Galerkin Approximation Type A: for local strong solutions in both $2D$ and $3D$}

For every $m\in\mathbb{N}$ and arbitrary $T>0$, we consider the following approximate problem (AP1):
looking for functions
\be
\u^m(t,x)=\sum_{i=1}^m g^m_i(t)\mathbf{v}_i(x),\quad \theta^m(t,x)=\sum_{i=1}^m r^m_i(t)w_i(x),\non
\ee
and $\phi^m(t,x)$ such that
\be
\mathrm{(AP1)}\begin{cases}
  \langle \u^m_t, \mathbf{v}\rangle_{\mathbf{V}', \mathbf{V}}
 + \int_\Omega (\u^m\cdot\nabla)\u^m\cdot \mathbf{v} dx
 + 2\int_\Omega \mu(\theta^m) \mathcal{D} \u^m : \mathcal{D}{\mathbf{v}} dx \\
  \qquad = \int_\Omega [\lambda(\theta^m)\nabla\phi^m\otimes\nabla\phi^m] : \nabla \mathbf{v} dx
    + \int_\Omega  \theta^m \mathbf{g} \cdot \mathbf{v} dx,\qquad \forall\,
 \mathbf{v}\in \mathbf{V}_m,\\
  \phi^m_t+\u^m\cdot \nabla\phi^m+\gamma (-\Delta \phi^m+W'(\phi^m)) =0,\quad \text{a.e.\ in}\  (0, T)\times \Omega, \\
   \langle\theta^m_t, w\rangle_{H^{-1},H^{1}_0} + \int_\Omega (\u^m\cdot \nabla\theta^m) w dx +\int_\Omega \kappa(\theta^m)\nabla \theta^m\cdot \nabla w dx=0,\qquad \forall \,w \in W_m,\\
  \phi^m=\phi_b,\quad \text{on} \ (0,T)\times \Gamma,\\
  \u^m|_{t=0}=\u_{0}^m:=\Pi_m \u_0, \quad \phi^m|_{t=0}=\phi_0,\quad \theta^m|_{t=0}=\theta_{0}^m:=\widetilde{\Pi}_m\theta_0,\quad \text{in}\ \Omega.
  \end{cases}\non
 \ee
 \begin{remark}
 In problem $\mathrm{(AP1)}$, we assume that $\mu(\cdot)$, $\kappa(\cdot)$ are taken in such a way as in subsection \ref{uplow} (recall also Remark \ref{uplowremark}).
 \end{remark}
 \bp\label{ppn3}
Suppose $n=2, 3$. We assume that $(\u_0, \phi_0, \theta_0)$ $\in$ $\mathbf{H} \times (H^1(\Omega)\cap L^\infty(\Omega))\times
 H^1_0(\Omega)$, $\phi_b\in H^\frac32(\Gamma)$ with $\phi_0|_\Gamma=\phi_b$ satisfying $|\phi_0|\leq 1$ a.e. in $\Omega$ and $|\phi_b|\leq 1$ on $\Gamma$.
 For every $m\in\mathbb{N}$, there is a time $T_m>0$ depending on $\u_0$, $\phi_0$, $\theta_0$, $m$ and $\Omega$ such that problem $\mathrm{(AP1)}$ admits a unique solution on $[0,T_m]$ satisfying $\u^m\in H^1(0,T_m; \mathbf{V}_m)$,  $\theta^m\in H^1(0,T_m; W_m)$ and $\phi^m\in L^\infty(0, T_m; {H}^1(\Omega)\cap L^\infty(\Omega)) \cap L^2(0,
T_m; {H}^2(\Omega))$. Moreover, if we further assume $\phi_0\in H^2(\Omega)$ and $\phi_b\in H^\frac52(\Gamma)$, then $\phi^m\in L^\infty(0, T_m; {H}^2(\Omega)) \cap L^2(0,T_m; {H}^3(\Omega))$.
\ep

 \begin{proof}

Let $T>0$ and $M=2+2\|\u_0\|^2+2\|\theta_0\|^2$. Consider arbitrary given functions
 $$\mathbf{w}^m(t,x)=\sum_{i=1}^mg^m_i(t)\mathbf{v}_i\in C([0,T]; \mathbf{V}_m),\quad \psi^m(t,x)=\sum_{i=1}^m r^m_i(t)w_i\in C([0,T]; W_m),$$
  with $$g_i^m(0)=(\u_0, \mathbf{v}_i),\ \ r_i^m(0)=(\theta_0, w_i),\quad \sup_{t\in[0,T]}\sum_{i=1}^m (|g^m_i(t)|^2+|r^m_i(t)|^2)\leq M.$$
  It is obvious that
\be
 \sup_{t\in[0,T]}\|\mathbf{w}^m(t,x)\|^2\leq M,\quad  \sup_{t\in[0,T]}\|\mathbf{w}^m(t,x)\|^2_{\mathbf{L}^\infty}\leq M \max_{1\leq i\leq m}\|\mathbf{v}_i\|_{\mathbf{L}^\infty}^2\leq C_m M.\label{aesw}
\ee

\textit{Step 1.} We consider the following semilinear parabolic equation for $\phi^m$ with convention term under the given velocity $\mathbf{w}^m$:
\begin{equation}
\begin{cases}
\phi^m_t+\mathbf{w}^m\cdot \nabla\phi^m+\gamma (-\Delta \phi^m+W'(\phi^m)) =0,\quad \text{a.e.\ in}\  (0, T)\times \Omega,\\
\phi^m=\phi_b,\qquad  \text{on} \ (0,T)\times\Gamma,\\
 \phi^m|_{t=0}=\phi_0,\quad \text{for}\ x\in \Omega,
\end{cases}\label{apphi1}
\end{equation}
Well-posedness of problem \eqref{apphi1} can be obtained by a standard fixed point argument, similar to the liquid crystal system \cite{LL95,C09}. Indeed, we have
\bl\label{phiphiap}
Assume that $\mathbf{w}^m\in C([0,T];\mathbf{V}_m)$, $\phi_0\in H^1(\Omega)\cap L^\infty(\Omega)$, $\phi_b\in H^\frac32(\Gamma)$ with $\phi_0|_\Gamma=\phi_b$ satisfying $|\phi_0|\leq 1$ a.e. in $\Omega$ and $|\phi_b|\leq 1$ on $\Gamma$. Then there exists a unique weak solution $\phi^m\in L^\infty(0, T; {H}^1(\Omega)\cap L^\infty(\Omega)) \cap L^2(0,
T; {H}^2(\Omega))$ to problem \eqref{apphi1} such that $|\phi^m(t,x)|\leq 1$ a.e. in $[0,T]\times \Omega$. Moreover, if $\phi_0\in H^2(\Omega)$ and $\phi_b\in H^\frac52(\Gamma)$, then $\phi^m\in L^\infty(0, T; {H}^2(\Omega)) \cap L^2(0,T; {H}^3(\Omega))$.
\el
It is easy to verify that the weak solution $\phi^m$ satisfies the following energy inequality (see e.g., \cite{LL95})
\bea
&&\frac{d}{dt}\int_\Omega\left( \frac12|\nabla \phi^m|^2+W(\phi^m) \right)dx +\gamma \|-\Delta \phi^m+W'(\phi^m)\|^2 \non\\
&=&-\int_\Omega (\mathbf{w}^m\cdot\nabla \phi^m)(-\Delta \phi^m+W'(\phi^m))dx\non\\
&\leq& \frac{\gamma}{2}\|-\Delta \phi^m+W'(\phi^m)\|^2+\frac{1}{2\gamma}\|\mathbf{w}^m\|_{\mathbf{L}^\infty}^2\|\nabla \phi^m\|^2\non\\
&\leq& \frac{\gamma}{2}\|-\Delta \phi^m+W'(\phi^m)\|^2+\frac{C_m M}{\gamma}\left(\frac12\|\nabla \phi^m\|^2+\int_\Omega W(\phi^m) dx\right),
\eea
where in the last step we have used \eqref{aesw}. Then by the Gronwall inequality, it follows that
\bea
&&\sup_{t\in[0,T]}\int_\Omega\left( \frac12|\nabla \phi^m|^2+W(\phi^m) \right)dx+\frac{\gamma}{2}\int_0^T \|-\Delta \phi^m+W'(\phi^m)\|^2 dt\non\\
&\leq& \left( \frac12\|\nabla \phi_0\|^2+\int_\Omega W(\phi_0) dx\right)(\gamma^{-1} C_m M T+1) e^{\gamma^{-1} C_m M T},
\eea
which together with the weak maximum principle Lemma \ref{mphi} and the elliptic regularity theorem yields
\bea
&& \sup_{t\in[0,T]}\|\phi^m(t)\|_{H^1}^2+\int_0^T \|\phi^m(t)\|_{H^2}^2 dt\non\\
&\leq& K(\|\phi_0\|_{H^1}, \|\phi_b\|_{H^\frac32(\Gamma)}, \gamma, T, M, m):=K.\label{aesw2}
\eea
We note that the constant $K$ on the right-hand side of \eqref{aesw2} can be chosen independent of $T$ if $T\in (0,1]$. Besides, for the semilinear parabolic equation \eqref{apphi1}, it is easy to prove the continuous dependence on the initial data as well as the  given velocity field $\mathbf{w}^m$. Therefore, the solution operator defined by problem \eqref{apphi1} $$\Phi^m: C([0,T]; \mathbf{V}_m)\times C([0,T]; W_m) \to  L^\infty(0, T; {H}^1(\Omega)) \cap L^2(0, T; {H}^2(\Omega))$$
such that $\phi^m=\Phi^m(\mathbf{w}^m, \psi^m)$, is continuous (indeed $\Phi^m$ does not depend on $\psi^m$).

\textit{Step 2.} Once $\phi^m$ is determined, we turn to look for functions
$$\mathbf{u}^m(t,x)=\sum_{i=1}^m \tilde{g}^m_i(t)\mathbf{v}_i,\quad \theta^m(t,x)=\sum_{i=1}^m \tilde{r}^m_i(t)w_i,$$
that satisfy the following system, for $i=1,...,m$,
\be
\begin{cases}
  \langle \u^m_t, \mathbf{v}_i\rangle_{\mathbf{V}', \mathbf{V}}
 + \int_\Omega (\u^m\cdot\nabla)\u^m\cdot \mathbf{v}_i dx
 + 2\int_\Omega \mu(\theta^m) \mathcal{D} \u^m : \mathcal{D}{\mathbf{v}}_i dx \\
  \qquad = \int_\Omega [\lambda(\theta^m)\nabla\phi^m\otimes\nabla\phi^m] : \nabla \mathbf{v}_i dx
    + \int_\Omega  \theta^m \mathbf{g} \cdot \mathbf{v}_i dx,\\
   \langle\theta^m_t, w_i\rangle_{H^{-1},H^{1}_0} + \int_\Omega (\u^m\cdot \nabla\theta^m) w_i dx +\int_\Omega \kappa(\theta^m)\nabla \theta^m\cdot \nabla w_i dx=0,\\
  \u^m|_{t=0}=\Pi_m \u_0, \quad \theta^m|_{t=0}=\widetilde{\Pi}_m\theta_0,\quad \text{in}\ \Omega,
  \end{cases}\label{ODE}
 \ee
which is equivalent to a system consisting of $2m$ nonlinear ordinary differential equations for the coefficients $\{\tilde{g}^m_i(t)\}_{i=1}^m$ and $\{\tilde{r}_i^m(t)\}_{i=1}^m$. Due to our assumptions on the smoothness of $\lambda$, $\kappa$ and $\mu$, it is standard to show the local well-posedness of the above initial problem using the classical theory of ODEs (see, e.g., \cite{LB99, Z04}). Namely, we have
 \bl
 Let $\phi^m\in L^\infty(0, T; {H}^1(\Omega)\cap L^\infty(\Omega)) \cap L^2(0, T; {H}^2(\Omega))$. Assume that $\u_0\in \mathbf{H}$, $\theta_0\in H^1_0(\Omega)$, then problem \eqref{ODE} admits a unique solution on $[0,T_1)$ such that
$\mathbf{u}^m(t,x)=\sum_{i=1}^m \tilde{g}^m_i(t)\mathbf{v}_i\in H^1(0,T_1; \mathbf{V}_m)$, $\theta^m(t,x)=\sum_{i=1}^m \tilde{r}^m_i(t)w_i\in H^1(0,T_1; W_m)$,
  where $T_1\in (0,T]$ may depend on $M$, $\phi^m$ and $m$.
 \el
 In particular, recalling subsection \ref{uplow} and Remark \ref{uplowremark}, the following estimates hold for $\u^m$ and $\theta^m$:
 \be
 \frac12\frac{d}{dt}\|\theta^m\|^2 + \int_\Omega \kappa(\theta^m)|\nabla \theta^m|^2dx= -\int_\Omega (\u_m\cdot\nabla \theta^m)\theta^m dx=0\non
 \ee
 and
 \bea
 &&\frac12\frac{d}{dt}\|\u^m\|^2 + 2\int_\Omega \mu(\theta^m)|\mathcal{D}\u^m|^2dx\non\\
 &=& \int_\Omega [\lambda(\theta^m)\nabla\phi^m\otimes\nabla\phi^m] : \nabla \mathbf{u}^m dx
    + \int_\Omega  \theta^m \mathbf{g} \cdot \mathbf{u}^m dx\non\\
 &\leq& C(1+\|\theta^m\|_{L^\infty})\|\nabla \phi^m\|^2 \|\nabla \mathbf{u}^m\|_{\mathbf{L}^\infty}+C\|\theta^m\|\|\nabla \u^m\|\non\\
 &\leq& \frac{\underline{\mu}}{2}\|\nabla \u^m\|^2+ \frac{C_m}{\underline{\mu}} \|\nabla \phi^m\|^4(1+\|\theta^m\|^2)+\frac{C}{\underline{\mu}}\|\theta^m\|^2,
 \eea
 where we have used the inverse inequalities $\|\theta^m\|_{L^\infty}\leq C_m \|\theta^m\|$ and $\|\nabla \u^m\|_{\mathbf{L}^\infty}\leq C_m \|\nabla \u^m\|$ because $(\u^m, \theta^m)$ is indeed finite dimensional. By the fact $2\|\mathcal{D}\u^m\|^2=\|\nabla \u^m\|^2$ and the Gronwall inequality, we obtain
 \bea
 &&\sup_{t\in [0,T_1]}\|\theta^m(t)\|^2+2\underline{\kappa}\int_0^{T_1}\|\nabla \theta^m\|^2 dt\leq \|\theta_0\|^2,\label{aesthm}\\
 &&\sup_{t\in [0,T_1]}\|\u^m(t)\|^2+\underline{\mu}\int_0^{T_1}\|\nabla \u^m\|^2 dt\non\\
 && \qquad \leq \|\u_0\|^2 + \frac{C_m}{\underline{\mu}} T_1(1+\|\theta_0\|^2)\sup_{t\in[0,T_1]}\|\nabla \phi^m\|^4+\frac{C}{\underline{\mu}}T_1\|\theta_0\|^2.\label{aesum}
 \eea
Since $\phi^m$ is bounded in $L^\infty(0,T_1; H^1(\Omega))\cap L^2(0,T_1;H^2(\Omega))$, inserting it back to the ODE system \eqref{ODE}, we can further deduce that $\tilde{g}^m_i(t)$, $\tilde{r}^m_i(t)$ $\in H^1(0,T_1)$ and thus $\u^m$ and $\theta^m$ are bounded in $H^1(0,T_1; \mathbf{V}_m)$ and $H^1(0,T_1; W_m)$, respectively, by a constant that depends on $M$, $m$. Besides, for the ODE system \eqref{ODE}, it is easy to prove the continuous dependence on its initial data and the given function $\phi^m$. As a consequence, the solution operator defined by problem \eqref{ODE}
$$\Psi^m: L^\infty(0,T_1; H^1(\Omega))\cap L^2(0,T_1;H^2(\Omega)) \to  H^1(0,T_1; \mathbf{V}_m)\times H^1(0,T_1; W_m)$$
such that $(\u^m,\theta^m)=\Psi^m(\phi^m)$, is continuous.

 \textit{Step 3.} We now prove the existence of solutions to problem (AP1) for sufficiently short time intervals. From the previous steps, we can see that the operator
 $$\Psi^m\circ\Phi^m:C([0,T_1]; \mathbf{V}_m)\times C([0,T_1]; W_m)\to H^1(0,T_1; \mathbf{V}_m)\times H^1(0,T_1; W_m)$$ such that
 $\Psi^m\circ\Phi^m(\mathbf{w}^m, \psi^m)=(\u^m, \theta^m)$, is continuous, where $(\u^m, \theta^m)$ is the solution to problem \eqref{ODE}. Moreover, the compactness of $H^1(0,T_1; \mathbf{V}_m)\times H^1(0,T_1; W_m)$ into $C([0,T_1]; \mathbf{V}_m)\times C([0,T_1]; W_m)$ (because $\mathbf{V}_m$ and $W_m$ are actually finite dimensional spaces) implies that $\Psi^m\circ\Phi^m$ is an compact operator from $C([0,T_1]; \mathbf{V}_m)\times C([0,T_1]; W_m)$ into itself. Finally, due to our choice of $M$ and the estimates \eqref{aesw2}, \eqref{aesthm}, \eqref{aesum}, it holds
 \be
 \sup_{t\in[0,T_1]}\left(\|\u^m(t)\|^2+\|\theta^m(t)\|^2\right)\leq \frac{M}{2}+ \frac{T_1}{2\underline{\mu}}\left(C_m M K^2+CM\right).
 \ee
 Hence, we can take $T_m\in(0,T_1)$ to be sufficiently small such that $\|\u^m(t)\|^2+\|\theta^m(t)\|\leq M$ for all $t\in [0,T_m]$.

 We are ready to apply the Schauder's fixed point theorem to conclude that there exists at least one fixed point $(\u^m, \theta^m)$
 in the bounded closed convex set
 \bea
 &&\Big\{ (\u^m, \theta^m) \in C([0,T_m]; \mathbf{V}_m\times  W_m)\ \mid\ \sup_{t\in[0,T_m]}\left(\|\u^m(t)\|^2+\|\theta^m(t)\|^2\right)\leq M,\non\\
 &&\qquad \qquad \qquad \qquad \qquad\qquad \quad\qquad\quad \quad\left. \text{with} \ \u^m(0)=\Pi_m \u_0, \ \theta^m(0)=\widetilde{\Pi}_m \theta_0. \right\}\non
 \eea
 such that $\u^m\in H^1(0,T_m; \mathbf{V}_m)$,  $\theta^m\in H^1(0,T_m; W_m)$ and $\phi^m\in L^\infty(0, T_m; {H}^1(\Omega)\cap L^\infty(\Omega)) \cap L^2(0,
T_m; {H}^2(\Omega))$. Uniqueness of the approximate solution $(\u^m, \phi^m, \theta^m)$ is an easy consequence of the energy method and its further regularity  follows from the regularity theory for parabolic equations (cf. e.g., \cite{lieb}). Thus, Proposition \ref{ppn3} is proved.
\end{proof}

\subsection{Semi-Galerkin Approximation Type B: for global weak solutions in $2D$}

It has been shown in Section \ref{sec:glow2D} that the weak maximum principle for $\theta$ (i.e., Lemma \ref{mtheta}) plays an important role in obtaining global lower-order estimates of solutions to problem \eqref{1}--\eqref{ini}. However, in the semi-Galerkin approximation of Type A given in the previous section, the $\theta$ equation is approximated by the Galerkin ansatz and Lemma \ref{mtheta} does not apply any more. Hence, we shall make use of an alternative scheme, in which only the equation \eqref{1} for the velocity $\u$ is approximated by the Galerkin method. To this end, for every $m\in\mathbb{N}$ and arbitrary $T>0$, we consider the following approximate problem (AP2): looking for functions
$$\u^m(t,x)=\sum_{i=1}^m g^m_i(t)\mathbf{v}_i(x),\quad \phi^m(t,x)\quad \mathrm{and}\quad \theta^m(t,x)$$
 such that
\be
\mathrm{(AP2)}\begin{cases}
  \langle \u^m_t, \mathbf{v}\rangle_{\mathbf{V}', \mathbf{V}}
 + \int_\Omega (\u^m\cdot\nabla)\u^m\cdot \mathbf{v} dx
 + 2\int_\Omega \mu(\theta^m) \mathcal{D} \u^m : \mathcal{D}{\mathbf{v}} dx \\
  \qquad = \int_\Omega [\lambda(\theta^m)\nabla\phi^m\otimes\nabla\phi^m] : \nabla \mathbf{v} dx
    + \int_\Omega  \theta^m \mathbf{g} \cdot \mathbf{v} dx,\quad \forall\,
 \mathbf{v}\in \mathbf{V}_m,\\
  \phi^m_t+\u^m\cdot \nabla\phi^m+\gamma (-\Delta \phi^m+W'(\phi^m)) =0,\quad \text{a.e.\ in}\  (0, T)\times \Omega, \\
   \theta^m_t +  \u^m\cdot \nabla\theta^m =\nabla \cdot( \kappa(\theta^m)\nabla \theta^m),\qquad \text{a.e.\ in}\  (0, T)\times \Omega,\\
  \phi^m=\phi_b,\quad \theta^m=0,\quad \text{on} \ (0,T)\times \Gamma,\\
  \u^m|_{t=0}=\u_{0m}:=\Pi_m \u_0, \quad \phi^m|_{t=0}=\phi_0,\quad \theta^m|_{t=0}=\theta_0,\quad \text{in}\ \Omega.
  \end{cases}\non
 \ee
 \bp\label{ppn2}
Suppose $n=2$. We assume that $(\u_0, \phi_0, \theta_0)$ $\in$ $\mathbf{H} \times (H^1(\Omega)\cap L^\infty(\Omega))\times
 (H^1_0(\Omega)\cap L^\infty(\Omega))$, $\phi_b\in H^\frac32(\Gamma)$ with $\phi_0|_\Gamma=\phi_b$ satisfying $|\phi_0|\leq 1$ a.e. in $\Omega$ and $|\phi_b|\leq 1$ on $\Gamma$.
 For every $m\in\mathbb{N}$, there is a time $T_m>0$ depending on $\u_0$, $\phi_0$, $\theta_0$, $m$ and $\Omega$ such that problem $\mathrm{(AP2)}$ admits a unique solution  on $[0,T_m]$ satisfying $\u^m\in H^1(0,T_m; \mathbf{V}_m)$,  $\phi^m\in L^\infty(0, T_m; {H}^1(\Omega)\cap L^\infty(\Omega)) \cap L^2(0,
T_m; {H}^2(\Omega))$ and  $\theta^m\in L^\infty(0, T_m; {H}_0^1(\Omega)\cap L^\infty(\Omega)) \cap L^2(0, T_m; {H}^2(\Omega))$.
Moreover, if we further assume $\phi_0\in H^2(\Omega)$, $\phi_b\in H^\frac52(\Gamma)$ and $\theta_0\in H^2(\Omega)$, then $\phi^m\in L^\infty(0, T_m; {H}^2(\Omega)) \cap L^2(0,T_m; {H}^3(\Omega))$ and $\theta^m\in L^\infty(0,T_m; H^2(\Omega))\cap L^2(0,T_m; H^3(\Omega))$.
\ep
\begin{proof}
Again, the existence of local solutions to the approximate problem (AP2) on certain time interval $[0,T_m]$ follows from a fixed point argument.
Let $T>0$ and $M=2+2\|\u_0\|^2$. Consider an arbitrary given vector $$\mathbf{w}^m=\sum_{i=1}^mg^m_i(t)\mathbf{v}_m \in C([0,T]; \mathbf{V}_m)$$ with $$g_i^m(0)=(\u_0, \mathbf{v}_i)\quad \mathrm{and} \quad \sup_{t\in[0,T]}\sum_{i=1}^m |g^m_i(t)|^2\leq M.$$
 It is obvious that an estimate like \eqref{aesw} still holds for $\mathbf{w}^m$.

\textit{Step 1.} We investigate parabolic equations for $\phi^m$ and $\theta^m$ with convention term under the given velocity $\mathbf{w}^m$.
For $\phi^m$, one shall consider the problem \eqref{apphi1} again and obtain the same results as in \emph{Step 1} of the proof for Proposition \ref{ppn3} (here we note that the temperature variable $\theta^m$ does not appear in the $\phi^m$-equation \eqref{apphi1}).

Next, we consider the equation for $\theta^m$:
\begin{equation}
\begin{cases}
  \theta^m_t+ \mathbf{w}^m\cdot \nabla\theta^m= \nabla\cdot( \kappa(\theta^m)\nabla \theta^m),\quad \text{a.e.\ in}\  (0, T)\times \Omega,\\
  \theta^m=0,\quad \text{on} \ (0,T)\times \Gamma,\\
  \theta^m|_{t=0}=\theta_0(x),\quad \text{in}\ \Omega.
\end{cases}\label{apthe1}
\end{equation}
Well-posedness of problem \eqref{apthe1} can be obtained by a standard Galerkin method (see e.g., \cite{LB99}). Indeed, we can prove
\bl\label{thetaap}
Assume that $\mathbf{w}^m\in C([0,T];\mathbf{V}_m)$, $\theta_0\in H^1_0(\Omega)\cap L^\infty(\Omega)$. Then there exist a unique weak solution $\theta^m\in L^\infty(0, T; {H}^1_0(\Omega)\cap L^\infty(\Omega)) \cap L^2(0,T; {H}^2(\Omega))$ to problem \eqref{apthe1}.
\el
Then we derive some estimates for the weak solution $\theta^m$ to problem \eqref{apthe1}. It is easy to see that the weak maximum principle Lemma \ref{mtheta} now applies to $\theta^m$, i.e.,
 \be
 \|\theta^m(t)\|_{L^\infty}\leq \|\theta_0\|_{L^\infty},\quad \forall\,t\in [0,T].\label{aesthmc}
 \ee
 By the continuity of $\kappa(\cdot)$, we see that $\kappa(\theta^m)$ has positive lower and upper bounds $0<\underline{\kappa}<\overline{\kappa}<+\infty$ on $[0,T]$, which depend on $\|\theta_0\|_{L^\infty}$.  Similar to \eqref{aesthm}, we have
 \be
 \sup_{t\in [0,T]}\|\theta^m(t)\|^2+2\underline{\kappa}\int_0^{T}\|\nabla \theta^m\|^2 dt\leq \|\theta_0\|^2.\label{aesthma}
 \ee
 Concerning the $H^1$-estimate for $\theta^m$, we introduce the transformation $\vartheta^m=\int_0^{\theta^m} \kappa(s) ds$
 and deduce the following equation for $\vartheta^m$ (cf. \eqref{newtheta1}):
 \be
\begin{cases}
 \vartheta^m_t+\mathbf{w}^m\cdot \nabla \vartheta^m-\tilde{\kappa}(\vartheta^m)\Delta \vartheta^m=0,\\
 \vartheta^m|_{\Gamma}=0, \\
 \vartheta^m|_{t=0}=\vartheta^m_0(x)=\int_0^{\theta_0(x)}\kappa(s)ds,
 \end{cases}\label{apnewtheta1}
\ee
 with $\underline{\kappa}\leq \tilde{\kappa}(\vartheta^m)\leq \overline{\kappa}$.
  Multiplying the above equation by $-\Delta \vartheta^m$ and integrating over $\Omega$, we get
\bea
\frac12\frac{d}{dt}\|\nabla \vartheta^m\|^2 +\underline{\kappa} \|\Delta \vartheta^m\|^2 &\leq& \int_\Omega (\mathbf{w}^m\cdot\nabla \vartheta^m )\Delta \vartheta^m dx \non\\
&\leq& \|\mathbf{w}^m\|_{\mathbf{L}^\infty}\|\nabla \vartheta^m\|\|\Delta \vartheta^m\|\non\\
&\leq& \frac{\underline{\kappa}}{2}\|\Delta \vartheta^m\|^2+\frac{C_m M}{2\underline{\kappa}}\|\nabla \vartheta^m\|^2,
\eea
where we have used the fact \eqref{aesw} again to estimate $\|\mathbf{w}^m\|_{\mathbf{L}^\infty}$. By the Gronwall inequality, it holds
\be\sup_{t\in[0,T]}\|\nabla \vartheta^m(t)\|^2+\underline{\kappa}\int_0^T \|\Delta  \vartheta^m(t)\|^2 dt
\leq \|\nabla \vartheta^m_0\|^2(\underline{\kappa}^{-1} C_m M T+1) e^{\underline{\kappa}^{-1} C_m M T},\non
\ee
which implies
\be
\sup_{t\in[0,T]}\|\theta^m(t)\|_{H^1}^2+\int_0^T \|\theta^m(t)\|_{H^2}^2 dt\leq K'(\|\theta_0\|_{H^1}, \|\theta_0\|_{L^\infty}, T, M, m, \underline{\kappa}, \overline{\kappa}).\label{aesw2a}
\ee
We note that the constant $K'$ on the right-hand side of \eqref{aesw2a} will be independent of $T$ if $T\in (0,1]$.

Next, we shall prove a continuous dependence result for problem \eqref{apthe1} on its initial data and the given velocity field $\mathbf{w}^m$.
Due to the restriction from Sobolev embedding theorems, this result is only available in the current regularity class for $\theta^m$ provided that the spatial dimension $n=2$.

Let $\theta_1^m$ and $\theta_2^m$ be two weak solutions to problem \eqref{apthe1} on $[0,T]$ corresponding to the initial data
 $\theta_{01}$, $\theta_{02}\in H^1_0(\Omega)\cap L^\infty(\Omega)$ and given vectors $\mathbf{w}^m_1, \mathbf{w}^m_2\in C([0,T]; \mathbf{V}_m)$, respectively.
 Denote the differences by
$\bar{\theta}^m=\theta_1^m-\theta_2^m$, $\bar{\mathbf{w}}^m=\mathbf{w}^m_1- \mathbf{w}^m_2$. We can see that $\bar \theta^m$ satisfies
 \be
 \begin{cases}
  \bar{\theta}^m_t+ \mathbf{w}^m_1 \cdot \nabla \bar{\theta}^m+\bar{\mathbf{w}}^m \cdot\nabla \theta^m_2\\
  \qquad =\nabla \cdot(\kappa(\theta_1^m)\nabla \bar{\theta}^m)+\nabla \cdot[(\kappa(\theta_1^m)-\kappa(\theta_2^m))\nabla \theta_2^m],\\
  \bar \theta^m|_\Gamma=0,\\
 \bar\theta^m|_{t=0}=\theta_{01}-\theta_{02}\, \quad \text{in}\ \Omega.
 \end{cases}
 \ee
Multiplying the above equation by $\bar{\theta}^m$, integrating over $\Omega$, by the Sobolev embedding theorem ($n=2$), the H\"older inequality, Young's inequality and the estimates \eqref{aesthmc}, we have
\bea
&&\frac12\frac{d}{dt}\|\bar{\theta}^m\|^2+\int_\Omega \kappa(\theta_1^m)|\nabla \bar{\theta}^m|^2dx \non\\
&=& -\int_\Omega (\mathbf{w}^m_1 \cdot \nabla \bar{\theta}^m+\bar{\mathbf{w}}^m \cdot\nabla \theta^m_2) \bar{\theta}^m dx-\int_\Omega (\kappa(\theta_1^m)-\kappa(\theta_2^m))\nabla \theta_2^m\cdot \nabla \bar{\theta}^m dx\non\\
&\leq&  \|\bar{\mathbf{w}}^m \|_{\mathbf{L}^4}\|\nabla \theta^m_2\|_{\mathbf{L}^4}\| \bar{\theta}^m\|\non\\
&&\quad  +\|\int_0^1\kappa'(\tau\theta_1^m+(1-\tau)\theta_2^m)\bar{\theta}^m d\tau\|_{L^4}\|\nabla \theta_2^m\|_{\mathbf{L}^4}\|\nabla \bar{\theta}^m \|\non\\
&\leq& C \|\bar{\mathbf{w}}^m \|_{\mathbf{V}}\|\nabla \theta^m_2\|^\frac12\|\Delta \theta^m_2\|^\frac12\| \bar{\theta}^m\|\non\\
&&\quad  + C\|\kappa'\|_{L^\infty}\|\nabla \theta_2^m\|^\frac12\|\Delta \theta_2^m\|^\frac12 \|\bar{\theta}^m \|^\frac12 \|\nabla \bar{\theta}^m \|^\frac32\non\\
&\leq& \frac{\underline{\kappa}}{2}\|\nabla \bar{\theta}^m\|^2+ \|\bar{\mathbf{w}}^m\|_{\mathbf{V}}^2 + L(t) \|\bar{\theta}^m \|^2,\label{aesw2ag}
\eea
where $$L(t)=\frac{C}{\underline{\kappa}^3}\|\nabla \theta^m_2\|^2\|\Delta \theta^m_2\|^2+C\|\nabla \theta^m_2\|\|\Delta \theta^m_2\|. $$
It follows from the estimate \eqref{aesw2a} and the Cauchy--Schwarz inequality that
$$
\int_0^T L(t) dt\leq C(K',T)+\frac{C}{\underline{\kappa}^3}\sup_{t\in[0,T]}\|\nabla \theta^m_2\|^2\int_0^T\|\Delta \theta^m_2\|^2dt<+\infty.
$$
Then by \eqref{aesw2ag} and the Gronwall inequality, we get
\be
\|\bar{\theta}^m(t)\|^2 \leq \left(\|\theta^m_{01}-\theta^m_{02}\|^2+\int_0^t\|\bar{\mathbf{w}}^m\|_{\mathbf{V}}^2d\tau \right)e^{\int_0^t L(\tau)d\tau},\quad \forall\,t\in [0,T].\non
\ee
As a consequence, the solution operator defined by problems \eqref{apphi1} and \eqref{apthe1}
\bea
&&\Phi^m: C([0,T]; \mathbf{V}_m) \to  \non\\
&&\quad [L^\infty(0, T; {H}^1(\Omega) \cap L^2(0, T; {H}^2(\Omega))]\times [L^\infty(0, T; {H}^1_0(\Omega)\cap L^\infty(\Omega)) \cap L^2(0,T; {H}^2(\Omega))]
\non
\eea
 such that $(\phi^m, \theta^m)=\Phi^m(\mathbf{w}^m)$, is continuous.

\textit{Step 2.} Once the functions $\phi^m$ and $\theta^m$ have been  determined, we turn to look for a vector $\mathbf{u}^m(t,x)=\sum_{i=1}^m \tilde{g}^m_i(t)\mathbf{v}_i$
that satisfies a system of $m$ nonlinear ordinary differential equations for $\{\tilde{g}^m_i(t)\}_{i=1}^m$ such that for $i=1,...,m$,
\be
\begin{cases}
  \langle \u^m_t, \mathbf{v}_i\rangle_{\mathbf{V}', \mathbf{V}}
 + \int_\Omega (\u^m\cdot\nabla)\u^m\cdot \mathbf{v}_i dx
 + 2\int_\Omega \mu(\theta^m) \mathcal{D} \u^m : \mathcal{D}{\mathbf{v}}_i dx \\
  \qquad = \int_\Omega [\lambda(\theta^m)\nabla\phi^m\otimes\nabla\phi^m] : \nabla \mathbf{v}_i dx
    + \int_\Omega  \theta^m \mathbf{g} \cdot \mathbf{v}_i dx,\\
  \u^m|_{t=0}=\Pi_m \u_0,\quad \text{in}\ \Omega.
  \end{cases}\label{ODE1}
 \ee
Due to our assumptions on the coefficients $\lambda$, $\kappa$ and $\mu$, it is standard to prove
 \bl
 Let $n=2$, $\phi^m\in L^\infty(0, T; {H}^1(\Omega)\cap L^\infty(\Omega)) \cap L^2(0, T; {H}^2(\Omega))$ and $\theta^m \in L^\infty(0, T; {H}^1_0(\Omega)$ $\cap L^\infty(\Omega)) \cap L^2(0,T; {H}^2(\Omega))$. Assume that $\u_0\in \mathbf{H}$, then problem \eqref{ODE1} admits a unique solution on $[0,T_2]$ such that
$\mathbf{u}^m(t,x)=\sum_{i=1}^m \tilde{g}^m_i(t)\mathbf{v}_i\in H^1(0,T_2; \mathbf{V}_m)$, where $T_2\in (0,T)$ may depend on $M$, $\phi^m$, $\theta^m$, $m$.
 \el
 In a similar way as before, we see that the estimate \eqref{aesum} still holds for $\u^m$ and $\u^m$ is bounded in $H^1(0,T_2; \mathbf{V}_m)$, by a constant depending on $M$, $m$. Besides, for the ODE system \eqref{ODE1}, it is standard to prove the continuous dependence on its initial data and $\phi^m$, $\theta^m$. As a consequence, the solution operator defined by problem \eqref{ODE1}
 \begin{align}
 \Psi^m:\ &[L^\infty(0,T_2; H^1(\Omega)\cap L^\infty(\Omega))\cap L^2(0,T_2;H^2(\Omega))]\non\\
         &\qquad \times [L^\infty(0, T_2; {H}^1_0(\Omega)\cap L^\infty(\Omega)) \cap L^2(0,T_2; {H}^2(\Omega))]\non\\
         &\to  H^1(0,T_2; \mathbf{V}_m)\non
 \end{align}
  such that $\u^m=\Psi^m(\phi^m, \theta^m)$, is continuous.

 \textit{Step 3.} From the previous steps, we see that the operator
 $$\Psi^m\circ\Phi^m:C([0,T_2]; \mathbf{V}_m)\to H^1(0,T_2; \mathbf{V}_m)$$
  such that
 $\Psi^m\circ\Phi^m(\mathbf{w}^m)=\u^m$, is continuous, where $\u^m$ is the solution to problem \eqref{ODE1}. Again, compactness of $H^1(0,T_2; \mathbf{V}_m)$ into $C([0,T_2]; \mathbf{V}_m)$ implies that the operator $\Psi^m\circ\Phi^m$ is compact from $C([0,T_2]; \mathbf{V}_m)$ into itself. Thanks to our current choice of $M$ and the estimate \eqref{aesum}, we can find a time $T_m\in(0,T_2)$ to be sufficiently small such that $\|\u^m(t)\|^2\leq M$ for all $t\in [0,T_m]$. The Schauder's fixed point theorem implies that there exists at least one fixed point $\u^m$ in the bounded closed convex set
 \bea
 &&\Big\{ \u^m \in C([0,T_m]; \mathbf{V}_m)\ \mid\ \sup_{t\in[0,T_m]}\|\u^m(t)\|^2\leq M,\ \text{with} \ \u^m(0)=\Pi_m \u_0. \Big\}\non
 \eea
 such that $\u^m\in H^1(0,T_m; \mathbf{V}_m)$, $\phi^m\in L^\infty(0, T_m; {H}^1(\Omega)\cap L^\infty(\Omega)) \cap L^2(0,
T_m; {H}^2(\Omega))$ and $\theta^m\in L^\infty(0, T_m; {H}^1_0(\Omega)\cap L^\infty(\Omega)) \cap L^2(0,T_m; {H}^2(\Omega))$. Uniqueness of the approximate solution $(\u^m, \phi^m, \theta^m)$ is an easy consequence of the energy method and its further regularity follows from the classical regularity theory for parabolic equations (cf. e.g., \cite{lieb}). Thus, Proposition \ref{ppn2} is proved.
\end{proof}


\section*{Acknowledgements.} The author is grateful to the referees for their very helpful
comments and suggestions. Part of the research was supported by National Science Foundation of China 11371098 and Shanghai Center for Mathematical Sciences at Fudan University.


\end{document}